\documentclass[10pt, reqno]{amsart}
\usepackage[margin=1.0in]{geometry}
\usepackage{amsmath,amssymb}
\usepackage{xcolor} % A package to add color.
\usepackage{soul}
\usepackage{amsthm}
\usepackage{mathrsfs}
\usepackage{amsfonts}
\usepackage{subcaption}
\usepackage{graphicx}
\usepackage{csquotes}
\usepackage{natbib}
\usepackage[dvpis]{epsfig}
\usepackage[colorlinks=true]{hyperref}
\hypersetup{urlcolor=blue, citecolor=red}
\usepackage{enumerate}

\makeatletter
\newcommand*{\rom}[1]{\expandafter\@slowromancap\romannumeral #1@}
\makeatother
\newcommand{\N}{\mathbb{N}}
\newcommand{\R}{\mathbb{R}}
\newcommand{\Z}{\mathbb{Z}}

\newcommand{\ve}{\varepsilon}
\newcommand{\bp}{\begin{pmatrix}}
\newcommand{\ep}{\end{pmatrix}}
\newcommand{\intt}{\int_{\T^d}}
\newcommand{\p}{\partial}
\newcommand{\iinttr}{\iint_{\T^d\times\R^d}}

\newtheorem{theorem}{Theorem}[section]
\newtheorem{corollary}[theorem]{Corollary}
\newtheorem{lemma}[theorem]{Lemma}
\newtheorem{proposition}[theorem]{Proposition}

\theoremstyle{definition}

\theoremstyle{remark}

\newtheorem{remark}[theorem]{Remark}

\numberwithin{equation}{section}
\DeclareMathOperator*{\esssup}{ess\,sup}

\newcommand{\T} {\mathbb T}
\newcommand{\pa}{\partial}

\newcommand{\e}{\varepsilon}
\newcommand{\lt}{\left}
\newcommand{\rt}{\right}
\newcommand{\bq}{\begin{equation}}
\newcommand{\eq}{\end{equation}}

\newcommand{\bfk}{{\bf k}}

\newcommand{\bfv}{{\bf v}}

\newcommand{\bbI}{\mathbb I}

\newcommand{\bbN}{\mathbb N}

\newcommand{\calC}{\mathcal C}
\newcommand{\calD}{\mathcal D}
\newcommand{\calE}{\mathcal E}
\newcommand{\calF}{\mathcal F}

\newcommand{\calN}{\mathcal N}

\newcommand{\calP}{\mathcal P}

\newcommand{\weakto}{\rightharpoonup}

\newcommand{\intr}{\int_{\R^d}}
 
\newcommand{\dx}{\textnormal{d}x}
\newcommand{\dv}{\textnormal{d}v}

\newcommand{\ds}{\textnormal{d}s}
\newcommand{\dt}{\textnormal{d}t}

\usepackage[normalem]{ulem}
\usepackage{cancel}

\renewcommand{\sout}[1]{}
\renewcommand{\cancel}[1]{}
\newcommand{\dmu}{\textnormal{d}\mu}
\newcommand{\ddt}{\frac{\textnormal{d}}{\textnormal{d}t}}
\newcommand{\nd}{\textnormal{d}}

\begin{document}

\title[On the derivation of ionic Euler--Poisson system from kinetic models]{Hydrodynamic limit from kinetic models with massless electrons to the ionic Euler--Poisson system}

\author{Young-Pil Choi}
\address{Department of Mathematics, Yonsei University, Seoul 03722, Republic of Korea}
\email{ypchoi@yonsei.ac.kr}

\author{Dowan Koo}
\address{Department of Mathematics, Yonsei University, Seoul 03722, Republic of Korea}
\email{dowan.koo@yonsei.ac.kr}

\author{Sihyun Song}
\address{Department of Mathematics, Yonsei University, Seoul 03722, Republic of Korea}
\email{ssong@yonsei.ac.kr}

\date{\today}
%\subjclass{35B40, 35Q83, 35Q35, 35A01}


\begin{abstract}

We study the derivation of ion dynamics, namely, the ionic Euler--Poisson system, from kinetic descriptions. The kinetic framework consists of the ionic Vlasov--Poisson equation coupled with either a nonlinear Fokker--Planck operator or a local alignment term. In both kinetic and fluid models, the massless electrons are assumed to be in thermodynamic equilibrium, leading to an electric potential governed by the Poisson--Boltzmann equation. The exponential nonlinearity in this semilinear elliptic problem creates significant mathematical difficulties, which we overcome by exploiting the physical structure of the system, in particular, the role of the electron velocity field hidden in the limiting equation. Our first main result establishes the hydrodynamic limit from the kinetic model to the ionic Euler--Poisson system, providing quantitative error estimates via the modulated energy method. As a second contribution, we prove the global-in-time existence of weak entropy solutions to the kinetic equations, ensuring consistency with the hydrodynamic limit framework.
 \end{abstract}

% ----------------------------------------------------------------

\maketitle

% ----------------------------------------------------------------
%\setcounter{tocdepth}
\tableofcontents
%%%%%%%%%%%%%%%%%%%%%%%%%%%%%%%%%%%%%%%%%%%%%%%%%%%%%%%%%
%
%
%
%
%
%%%%%%%%%%%%%%%%%%%%%%%%%%%%%%%%%%%%%%%%%%%%%%%%%%%%%%%%%

%\pagebreak
\section{Introduction}
%%%%%%%%%%%%%%%%%%%%%%%%%%%%%%%%%%%%%%%%%%%%%%%%%%%%%%%%%
%
%
%
%
%
%%%%%%%%%%%%%%%%%%%%%%%%%%%%%%%%%%%%%%%%%%%%%%%%%%%%%%%%%

We are interested in deriving the following ionic Euler--Poisson equations, which arise in plasma physics when electrons are considered \emph{massless} and in thermodynamic equilibrium:
\begin{align}
\label{eq : ionic EP}
\begin{cases}
    \p_t \rho + \nabla_x \cdot (\rho u) = 0, \quad t > 0, \quad x \in \T^d,\\
    \p_t (\rho u) + \nabla_x \cdot (\rho u \otimes u + \kappa \rho \bbI) = -\rho \nabla \Phi,\\
    -\Delta \Phi = \rho - e^{\Phi}.
    \end{cases}
\end{align}
Here $\rho=\rho(t,x)$ and $u=u(t,x)$ represent the ion density and velocity at time $t$ and position $x \in \T^d$ with dimension $d \geq 1$, respectively. The constant $\kappa \geq 0$ denotes the ratio of the ion temperature to the electron temperature. When $\kappa > 0$, the term $\kappa \rho$ models the ion pressure associated with thermal effects. In contrast, the pressureless case $\kappa = 0$ corresponds to the cold ion regime, where the ion temperature is negligible compared to that of the electrons.

The semilinear elliptic equation $\eqref{eq : ionic EP}_3$, known as the Poisson--Boltzmann equation, describes the self-consistent electric potential generated by the plasma under the assumption that electrons are in thermodynamic equilibrium and effectively massless. The exponential term $e^{\Phi}$ represents the density of thermalized electrons. This additional nonlinearity complicates the analysis, and as a result, the ionic model has been studied less extensively than its classical counterpart.

To motivate the system \eqref{eq : ionic EP} and clarify the emergence of the exponential nonlinearity in the Poisson--Boltzmann equation, we briefly review its derivation from a two-fluid Euler--Poisson model for ion and electron species. A representative example can be found in \cite{grenierguopausadersuzuki2020}, where the dynamics of ions and electrons near a constant equilibrium density is considered. The system reads:
\begin{align}
\begin{cases}
    \partial_t \rho_e + \text{div}_x (\rho_e u_e) = 0,\\
    m_e(\partial_t (\rho_e u_e) + \text{div}_x (\rho_e u_e \otimes u_e)) + \nabla p_e(\rho_e) =  c_e \rho_e \nabla \Phi,\\
    \partial_t \rho_i + \text{div}_x (\rho_i u_i) = 0,\\
    m_i(\partial_t (\rho_i u_i) + \text{div}_x (\rho_i u_i \otimes u_i)) + \nabla p_i(\rho_i) = - c_e \rho_i \nabla \Phi, \\
    - \Delta \Phi = 4\pi c_e (\rho_i - \rho_e),
\end{cases}
\label{coupled ion electron}
\end{align}
where $c_e>0$ is the electric charge, $\Phi$ is the electric potential, $m_e$ and $m_i$ denote the electron and ion mass, $\rho_e$ and $\rho_i$ denote the densities, $u_e$ and $u_i$ the velocities, and $p_e(\rho_e)$, $p_i(\rho_i)$ are the respective pressure laws. 

Motivated by the fact that electrons are much lighter than ions, we consider the singular limit
\begin{align*}
    \eta = \frac{m_e}{m_i} \to 0,
\end{align*}
while assuming $m_i = O(1)$. This regime corresponds to neglecting electron inertia and yields the so-called \emph{massless electron limit}. Setting $p_e(\rho_e) = \rho_e$ and $p_i(\rho_i) = \kappa \rho_i$, and formally taking $m_e = 0$ in \eqref{coupled ion electron} leads to
\begin{align*}
    \nabla \log \rho_e = c_e \nabla \Phi,
\end{align*}
and thus, up to a constant,
\begin{align}\label{eq: rhoe ePhi}
    \rho_e = e^{c_e \Phi}.
\end{align}
Substituting this relation into the Poisson equation yields
\begin{align*}
    -\Delta \Phi = 4\pi c_e (\rho_i - e^{c_e \Phi})
\end{align*}
so that the electron density is explicitly determined by the potential. This allows the system to decouple into an effective model for ion dynamics. After appropriate nondimensionalization, the system reduces to the ionic Euler--Poisson equations \eqref{eq : ionic EP}. For a rigorous justification of this derivation, we refer again to \cite{grenierguopausadersuzuki2020}. The local well-posedness of \eqref{eq : ionic EP} is studied in \cite{lanneslinaressaut2012}, and global smooth irrotational solutions with small amplitude are constructed in \cite{GP11} in three dimensions. On the other hand, the formation of finite-time singularities for the pressureless case has been studied in \cite{BCK24, CKKTpre, Liu06}.

The goal of the present work is to rigorously derive the ionic Euler--Poisson system \eqref{eq : ionic EP} as the hydrodynamic limit of a kinetic model describing the evolution of ion particles under the influence of collisions and self-consistent electric forces. To this end, we consider a kinetic equation of Vlasov--Poisson--Fokker--Planck type that relaxes toward \eqref{eq : ionic EP}. Our main result is a quantitative convergence estimate based on the modulated energy method. As part of the analysis, we construct global-in-time weak solutions to the kinetic model satisfying a suitable kinetic entropy inequality. This entropy structure plays a central role in obtaining uniform estimates and ensuring convergence to the limiting fluid system.

%%%%%%%%%%%%%%%%%%%%%%%%%%%%%%%%%%%%%%%%%%%%%%%%%%%%%%%%%
%
%
%
%
%
%%%%%%%%%%%%%%%%%%%%%%%%%%%%%%%%%%%%%%%%%%%%%%%%%%%%%%%%%

\subsection{Kinetic models and related notions}
From a kinetic perspective, the ionic dynamics can be described by the following collisionless Vlasov--Poisson system with exponential nonlinearity:
\begin{align}
\label{iVP}
    \begin{cases}
    \p_t f + v \cdot \nabla_x f - \nabla \Phi[f]\cdot \nabla_v f = 0, \\
    -\Delta \Phi[f] = \int_{\R^d}f\,\dv - e^{\Phi[f]},
\end{cases}
\end{align}
where $f = f(t,x,v)$ denotes the ion distribution function and $\Phi[f]$ is the associated self-consistent potential determined by the Poisson--Boltzmann equation. This model captures the interaction between ions and thermalized electrons assumed to be in equilibrium.

The existence of global weak solutions was established by Bouchut \cite{bouchut1991}. More recently, well-posedness, propagation of moments, and stability in the Wasserstein distance were addressed in \cite{cesbroniacobelli2021, griffinpickeringiacobelli2021b, griffinpickeringiacobelli2021a}, with a comprehensive survey provided in \cite{griffinpickeringiacobelli2021c}. Regarding limiting procedures, Bardos et al. \cite{bardosgolsenguyensentis2018} derived the ionic Vlasov--Poisson system as a massless-electron limit from the coupled ion-electron Vlasov--Poisson equations, under the presence of electron collisions modeled via Boltzmann or BGK-type operators. The mean-field limit of the system \eqref{iVP} was rigorously justified in \cite{griffinpickeringiacobelli2020} through a deterministic approach, while a probabilistic version was later developed in \cite{griffinpickering2024}. In a different direction, the quasineutral limit was analyzed by Han-Kwan \cite{hankwan2011}, who formally derived macroscopic models such as the compressible Euler and shallow water equations. More recently, global-in-time existence of Lagrangian and renormalized solutions was established in \cite{choikoosong2025}, assuming initial data with $L^{1+}$ integrability. That work also addressed the well-posedness of the Poisson--Boltzmann equation $-\Delta \Phi = \rho - e^\Phi$ for density functions $\rho \in L^p(\mathbb{T}^d)$ with $p > 1$.

In the presence of collisional effects, several recent works have extended \eqref{iVP} by incorporating collision operators. In \cite{FG24}, the Vlasov--Poisson system with Landau-type collisions was formally derived from a kinetic ion-electron model via the massless-electron limit. In a related development, the preprint \cite{li2025globalclassicalsolutionsionic} investigates the global well-posedness of classical solutions near equilibrium for the ionic Vlasov--Poisson system with the Boltzmann collision operator.

We now turn to models involving explicit relaxation mechanisms and their hydrodynamic limits. The derivation of compressible fluid equations with isothermal pressure has been well-studied as the hydrodynamic limit of the following nonlinear Fokker--Planck equation:
\begin{equation}\label{FP}
\pa_t f + v\cdot \nabla_x f + F \cdot  \nabla_v  f = \frac{1}{\tau} \calN_\kappa [f],
\end{equation}
where $\tau > 0$ is the relaxation time, $F$ is a given external force, and $\calN_\kappa$ is a Fokker--Planck-type collision operator defined by
\begin{align}\label{CO}
    \calN_{\kappa}[f] := \kappa\Delta_v f - \nabla_v \cdot ((u_f - v) f).
\end{align}
Here, $\kappa \geq 0$ represents the diffusion strength, and the macroscopic quantities are given by
\begin{align*}
    \rho_f := \int_{\R^d} f\,\dv, \qquad 
    u_f := \begin{cases} 
   \frac{\int_{\R^d} vf\,\dv}{\rho_f} & \rho_f \ne 0,\\
    0 & \textnormal{otherwise.}
    \end{cases}
\end{align*}
The quantity $\rho_f$ corresponds to the macroscopic density, while $u_f$ is often referred to as the bulk velocity. When $\kappa = 0$, the operator \eqref{CO} reduces to a local alignment-type interaction operator, as studied in \cite{CC20, kang2018, kangvasseur2015}. We also refer to \cite{KMT14} for a rigorous derivation of the local alignment operator from an underlying kinetic flocking model.

The Fokker--Planck operator is a standard tool in plasma modeling, used to describe slow stochastic momentum exchanges among particles; see \cite{lifshitzpitaevskii1981, spohn1980} for its physical derivation. The hydrodynamic limit of \eqref{FP} as $\tau \to 0$ was first established rigorously by Berthelin and Vasseur \cite{berthelinvasseur2005} using the relative entropy method, assuming a sufficiently regular external force. Related advances include the analysis of coupled Navier--Stokes--Fokker--Planck systems in \cite{CJ20, CJ23, MelletVasseur2008}, as well as the incompressible Euler and Navier--Stokes limits from nonlinear Vlasov--Fokker--Planck equations studied in \cite{CJ24, CJpre}. We also refer to \cite{carrillochoijung2021, choi2021, FK19, kang2018, kangvasseur2015, KMT2015} and the references therein for further developments on asymptotic analysis and relaxation limits in kinetic models.

We briefly recall properties of the collision operator $\calN_\kappa$ defined in \eqref{CO}. First, $\calN_\kappa$ conserves local mass and momentum locally:
\begin{equation}\label{cmm}
\intr (1,v) \calN_\kappa[f]\,\dv =0.
\end{equation}
This operator admits a family of local equilibria, denoted by $M_\kappa[f] := M_\kappa^{(\rho_f, u_f)}$, which are commonly referred to as the local Maxwellians (for $\kappa > 0$) or the monokinetic ansatz (for $\kappa = 0$). They are explicitly given by
\begin{equation}
\label{def : local maxwellian euler poisson}
\begin{split}
    M_\kappa^{(\rho, u)} (v) := \begin{cases}
        \frac{\rho}{(2\pi \kappa)^{d/2}}\exp\left(-\frac{|v-u|^2}{2\kappa} \right) &\kappa > 0,\\
        \rho \delta(v - u) &\kappa = 0,
    \end{cases}
    \end{split}
\end{equation}
for any $\rho \in [0,\infty)$ and $u \in \R^d$, where $\delta$ denotes the Dirac mass centered at zero.

In the case $\kappa>0$, the operator $\calN_\kappa$ can be written as
\begin{align*}
    \calN_{\kappa}[f] = \nabla_v\cdot \left( \kappa f \nabla_v \log \left(\frac{f}{M_{\kappa}[f]}\right) \right)
\end{align*}
which clearly indicates that $f = M_\kappa[f]$ is an equilibrium. Thus, in the singular limit $\tau \to 0$ of \eqref{FP}, one formally expects
\bq\label{near}
f \approx M_{\kappa}[f],
\eq
that is, $f$ remains close to the local equilibrium at each time.

The local equilibria $M_\kappa[f]$ enjoy several important properties that are crucial in deriving fluid-type limits. First, it satisfies $\calN_\kappa (M_\kappa[f])=0$ and preserves the macroscopic quantities of $f$:
\[
\intr (1,v) M_\kappa [f] \,\dv = (\rho_f, \rho_f u_f).
\]
Moreover, for $\kappa > 0$, it satisfies the isothermal pressure identity:
\begin{equation}\label{P:law}
\intr (v-u_f) \otimes (v-u_f) M_\kappa[f]\,\dv = \kappa \rho_f \bbI,
\end{equation}
which plays a key role in recovering the pressure law in the fluid equations. In addition, $M_\kappa[f]$ satisfies a \textit{compatibility condition} and \textit{minimization principle} regarding the following kinetic free energy functional:
\begin{align}\label{eq: min prin}
    \frac{1}{2}\rho_f |u_f|^2 + \kappa \rho_f \log \rho_f  = \int_{\R^d} H_\kappa [M_{\kappa}[f]](v)\,\dv \le \int_{\R^d} H_\kappa [f](v)\,\dv ,
\end{align}
where the entropy density $H_\kappa$ is defined by
\begin{equation}\label{H}
    H_\kappa[f](v) :=   \frac{|v|^2}{2}f + \kappa \lt(f \log f + \frac{d}{2}\log(2\pi \kappa)f\rt).
\end{equation}

We adopt the convention that the term in parentheses vanishes when $\kappa = 0$, in which case the entropy becomes purely kinetic. The presence of the linear term in $f$ ensures that equality in \eqref{eq: min prin} holds precisely when $f = M_\kappa[f]$. The left-hand side of \eqref{eq: min prin} defines the macroscopic entropy
\begin{align}\label{def: eta}
\eta_\kappa(\rho, \rho u):=  \frac{1}{2}\rho|u|^2 + \kappa \rho \log \rho,
\end{align}
which corresponds to the entropy of the limiting isothermal Euler system (or its pressureless variant when $\kappa = 0$). These structural properties of $\calN_\kappa$ and $M_\kappa[f]$ are fundamental in deriving compressible fluid models from the kinetic description as $\tau \to 0$ in \eqref{FP}.

In order to derive the ionic Euler--Poisson system \eqref{eq : ionic EP} from a kinetic formulation, we consider a system that couples the collisionless Vlasov equation with a relaxation mechanism toward local equilibrium. This leads us to the following ionic Vlasov--Poisson--Fokker--Planck  system (VPFP):
\begin{align}
\label{main eq}
    \begin{cases}
    \p_t f + v \cdot \nabla_x f  -\nabla \Phi[f] \cdot \nabla_v f = \frac{1}{\tau}\lt( \kappa\Delta_v f - \nabla_v \cdot ((u_f - v) f)\rt), \\
    -\Delta \Phi[f] = \int_{\R^d}f\,\dv - e^{\Phi[f]},
\end{cases}
\end{align}
with initial datum $f(0,\cdot,\cdot) =: f_0$. Here, $f=f(t,x,v)$ denotes the probability density of ions at time $t$ in the phase space $\T^d\times\R^d$. The associated electric field is given by
\[
E[f] := - \nabla \Phi[f],
\]
and the term $e^{\Phi[f]}$ corresponds to the thermalized electron density under the massless-electron approximation.

Integrating both sides of \eqref{main eq} against $(1, v)$ and using the conservation property \eqref{cmm}, we obtain the following local balance laws:
\begin{equation}
\label{eq : local balance laws}
\begin{split}
    &\p_t \rho_{f} + \nabla_x\cdot (\rho_{f} u_{f}) = 0,\\
    &\p_t (\rho_{f}u_{f}) + \nabla_x \cdot \left(\rho_{f} u_{f}\otimes u_{f} + \int_{\R^d} (u_{f} - v) \otimes (u_{f} - v) f \dv \right) = -\rho_{f} \nabla \Phi[f].
    \end{split}
\end{equation}
In light of the formal relaxation \eqref{near}, one expects $f \approx M_\kappa^{(\rho, u)}$ as $\tau \to 0$. Plugging this approximation into the momentum equation and using the identity \eqref{P:law}, we formally recover the ionic Euler--Poisson system \eqref{eq : ionic EP}.

We now introduce the notion of the \emph{kinetic entropy inequality}, which serves as a key a priori estimate in the analysis of the relaxation limit. A formal computation shows that any sufficiently regular solution $f$ to \eqref{main eq} satisfies the inequality 
\begin{align}
    \calE_\kappa[f(t)] + \frac{1}{\tau} \int_0^t \iint_{\T^d\times\R^d} \calD_\kappa[f](s) \,\dx\dv\ds \le \calE_\kappa[f_0],
    \label{eq : KEI}
\end{align}
where the energy functional is defined as
\begin{align}
\label{def : energy.}
    &\calE_\kappa[f] := 
    \iint_{\T^d\times\R^d} H_\kappa[f](v) \, \dx\dv + \int_{\T^d} \frac{|\nabla\Phi[f]|^2}{2} + (\Phi[f] - 1)e^{\Phi[f]} + 1\, \dx.
\end{align}
Here, the entropy density $H_\kappa[f]$ is as defined in \eqref{H}, and the dissipation term in \eqref{eq : KEI} is given by
\begin{align*}
%\label{eq : def : dissipation}
    \calD_\kappa[f] := \frac{1}{f}|\kappa\nabla_v f - (u_f - v)f|^2.
\end{align*}
The inequality \eqref{eq : KEI} implies, in particular, that the total energy $\calE_\kappa[f(t)]$ is non-increasing in time.

In the case $\kappa > 0$, although the Boltzmann entropy term $f \log f$ in \eqref{H} does not have a definite sign, it can be controlled by the second moment of $f$. In particular, as shown in Lemma \ref{lem:flogf:ctrl} below, the following inequality holds:
\[
\intr \frac{1}{2}|v|^2f + \kappa f\log f \,\dv \ge \intr \frac{1}{4}|v|^2f \,\dv - C_{\kappa,d}
\]
for some $C_{\kappa,d}\ge 0$ depending only on $\kappa$ and $d$. Combining this with the definition of $H_\kappa[f]$, we obtain the lower bound
\[
\iinttr H_{\kappa}[f]\,\dx\dv \ge \iinttr \frac{1}{4}|v|^2f\,\dx\dv - C_{\kappa,d} + \kappa \log (2\pi\kappa)\|f\|_{L^1(\T^d\times \R^d)}.
\]
Since the additional potential energy terms in \eqref{def : energy.}, namely $|\nabla \Phi|^2$ and $e^\Phi(\Phi - 1) + 1$, are nonnegative, we may combine this estimate with \eqref{eq : KEI} and the conservation of mass $\|f(t)\|_{L^1} = \|f_0\|_{L^1}$ to deduce the following uniform-in-time moment bound:
\begin{align}\label{eq: moments tau}
    \iinttr |v|^2f(t)\,\dx\dv \le 4(\calE_\kappa[f_0] +C_{\kappa,d}'(1+\|f_0\|_{L^1(\T^d\times \R^d)}))
\end{align}
for some constant $C_{\kappa,d}' \ge 0$. Moreover, from \eqref{eq : KEI} we obtain the dissipation estimate
\begin{align}\label{small:diss}
     \int_0^t \iint_{\T^d\times\R^d} \calD_\kappa[f](s) \,\dx\dv\ds \le \tau( \calE_{\kappa} [f_0] + C_{\kappa,d}'(1+\|f_0\|_{L^1(\T^d\times \R^d)})).
\end{align}
We remark that in the case $\kappa = 0$, both estimates \eqref{eq: moments tau} and \eqref{small:diss} hold with $C_{\kappa,d}' = 0$.

In what follows, for simplicity of notation, we will often omit the subscript $\kappa$ from $\calE_\kappa$, $\calD_\kappa$, and $H_\kappa$ when there is no ambiguity.

%%%%%%%%%%%%%%%%%%%%%%%%%%%%%%%%%%%%%%%%%%%%%%%%%%%%%%%%%
%
%
%
%
%
%%%%%%%%%%%%%%%%%%%%%%%%%%%%%%%%%%%%%%%%%%%%%%%%%%%%%%%%%

\subsection{Main result I: hydrodynamic limit}
We first establish that the ionic Euler--Poisson system \eqref{eq : ionic EP} arises as the hydrodynamic limit of the kinetic VPFP system \eqref{main eq} in the relaxation regime $\tau \to 0$. Given a weak solution $f^\tau$ to \eqref{main eq}, we define the associated macroscopic quantities by
\[
U^\tau := (\rho^\tau, \rho^\tau u^\tau) := \left(\intr f^\tau \,\dv, \intr vf^\tau\,\dv\right), \quad \Phi^\tau:= \Phi[f^\tau],
\]
and collect them into the augmented state vector
\[
\tilde U^\tau := \lt(\rho^\tau, \rho^\tau u^\tau, -\nabla \Phi^\tau, e^{\Phi^\tau} \rt).
\]
Similarly, given a smooth solution $(\rho, u, \Phi)$ to \eqref{eq : ionic EP}, we define
\begin{align*}
    U := (\rho, \rho u), \quad \tilde U := \left(\rho, \rho u, -\nabla \Phi, e^{\Phi} \right).
\end{align*}
The comparison between $f^\tau$ and the expected macroscopic limit $(\rho, u, \Phi)$ is carried out through the following \emph{modulated energy functional}:
\[
\calF(\tilde U^\tau | \tilde U) = \intt \frac{{\rho}^\tau}{2}| u^\tau - u|^2 + \kappa P( \rho^\tau |\rho) + \frac{1}{2}|\nabla {\Phi}^\tau - \nabla \Phi|^2 + e^{{\Phi}^\tau}({\Phi}^\tau - \Phi) - (e^{{\Phi}^\tau} - e^{\Phi})\,\dx,
\]
where $P(p|q) := p \log(p/q) - p + q$ denotes the standard relative entropy for scalar densities.

We now state our first main result.

\begin{theorem}\label{thm:hdr}
Let $d \geq 1$, $T>0$, and  let $(\rho, u,\Phi)$ be the unique classical solution to \eqref{eq : ionic EP} satisfying 
\bq\label{eq:iEP:lss}
(\rho, u,\Phi)\in C([0,T];H^{s}(\T^d)\times H^{s}(\T^d) \times H^{s+1}(\T^d)),\quad \inf_{(t,x)\in [0,T]\times \T^d}\rho(t,x) >0,
\eq
for some $s > \frac d2 + 1$. Let $\{f^\tau\}_{\tau>0}$ be a family of weak solutions to \eqref{main eq} satisfying the entropy inequality \eqref{eq : KEI} and the integrability condition
 \bq\label{sr:h}
\rho^\tau \nabla \Phi^\tau \in L^1 ((0,T) \times \T^d).
 \eq
Suppose that the initial data are well-prepared in the sense that
\begin{align*}
&\textnormal{\bf(H1)}\quad \calF(\tilde U^\tau_0|\tilde U_0)  \le C\sqrt{\tau},\\
&\textnormal{\bf(H2)}\quad\iint_{\T^d\times\R^d} H[f^\tau_0](v)\,\dx\dv \le  \int_{\T^d} \eta(U^\tau_0)\,\dx + C\sqrt{\tau}.
\end{align*}
Then for all $\kappa \ge 0$, the modulated energy satisfies the quantitative bound
\begin{align}\label{eq: conv mod egy}
    \calF( \tilde U^\tau| \tilde U) \le C\sqrt{\tau}.
\end{align}
As a consequence, we obtain the following convergence results:
\begin{enumerate}[(i)]
\item \textbf{Convergence of the electric field and electron density}:
\begin{align}
    \|\nabla(\Phi[f^\tau] - \Phi)\|_{L^\infty(0,T;L^2)}^2 \le C\sqrt{\tau}, \label{eq: tauconv l2} \\
    \|e^{\Phi[f^\tau]} - e^{\Phi}\|_{L^\infty(0,T;L^1)}^2 \le C\sqrt{\tau}. \label{eq: tauconv exp}
\end{align}
\item \textbf{Convergence in the pressured case ($\kappa > 0$)}: 
\begin{equation} \label{eq: convergences, kappa>0}
\begin{split}
&\|\rho^\tau - \rho\|_{L^\infty(0,T;L^1)}^2 \le C\sqrt{\tau}, \\
&\|\rho^\tau u^\tau - \rho u\|_{L^\infty(0,T;L^1)}^2 \le C\sqrt{\tau}, \\
&\|\rho^\tau u^\tau\otimes u^\tau - \rho u\otimes u\|_{L^\infty(0,T;L^1)}^2 \le C\sqrt{\tau}, \\
& \lt\|\int_{\R^d} (v-u^\tau)\otimes (v-u^\tau) f^\tau \dv -\kappa \rho \mathbb{I} \rt\|_{L^1(0,T;L^1)}^2 \le C\sqrt{\tau}.
\end{split}
\end{equation}
Moreover, the kinetic distribution $f^\tau$ converges to the local Maxwellian $M_\kappa^{(\rho, u)}$ defined in \eqref{def : local maxwellian euler poisson} in the sense that
\begin{align} 
\left\|f^\tau - M_\kappa^{(\rho,u)}\right\|_{L^1((0,T)\times\T^d\times\R^d)}^2 \le C\sqrt{\tau}.  \label{eq: conv of dist g1 k>0}
\end{align}

\item \textbf{Convergence in the pressureless case ($\kappa = 0$)}: Assume additionally that the initial mass distribution converges in the sense that
\begin{align*}
    &\textnormal{\textbf{(H3)}} \quad d_{\rm BL}^2(\rho^\tau_0,\rho_0) \le C\sqrt{\tau},
\end{align*}
where $d_{\rm BL}$ denotes the bounded Lipschitz distance. Then the following convergence estimates hold:
\begin{equation}\label{eq: convergences kappa=0}
\begin{split}
        &d_{\rm BL}^2(\rho^\tau,\rho) \le C\sqrt{\tau}, \\
        &d_{\rm BL}^2(\rho^\tau u^\tau, \rho u) \le C\sqrt{\tau},\\
        &d_{\rm BL}^2(\rho^\tau u^\tau\otimes u^\tau ,  \rho u \otimes u) \le C\sqrt{\tau}.
    \end{split}
\end{equation}
Moreover, the kinetic distribution $f^\tau$ converges to the monokinetic ansatz $\rho \delta(v-u)$ in the sense that
\begin{equation}\label{eq: conv of dist g1 k=0}
\int_0^T d_{\rm BL}^2\Big( f^\tau(t,\cdot,\cdot) \;, \; \rho(t,\cdot)  \delta(\cdot - u(t,\cdot)) \Big) \,\dt \le C\sqrt{\tau}.
\end{equation}
\end{enumerate}
\end{theorem} 

\begin{remark}
Notice that \eqref{eq : KEI} and \eqref{sr:h} ensure that the weak formulations of the local balance laws \eqref{eq : local balance laws} are well-defined. This condition is used to validate the modulated energy estimates in Lemma \ref{lem:fml}. We refer to Remark \ref{rmk:justify} for further comments. Also, in Remark \ref{rem: dim} below, we verify that the weak solutions constructed in Theorem \ref{thm:gws} satisfy these assumptions at least in low spatial dimensions $d = 2,3$.
\end{remark}

%%%%%%%%%%%%%%%%%%%%%%%%%%%%%%%%%%%%%%%%%%%%%%%%%%%%%%%%%
%
%
%
%
%
%%%%%%%%%%%%%%%%%%%%%%%%%%%%%%%%%%%%%%%%%%%%%%%%%%%%%%%%%

\subsubsection*{Proof strategy for Theorem \ref{thm:hdr}} \label{subsubsec: strat1}

We now outline the strategy for proving Theorem \ref{thm:hdr}.

\medskip

$\bullet$ {\bf Set-up for the modulated energy.}  Our approach relies on the modulated energy (or relative entropy) method, following the seminal framework developed by Dafermos \cite{dafermos1979, dafermos2000}. To this end, we first motivate our choice of modulated energy based on physical considerations.

We begin by noting that the following energy functional,
\begin{align}\label{eq:F}
    \int_{\T^d} \frac{|m|^2}{2\rho} + \kappa P(\rho) + \frac{1}{2}|\nabla \Phi|^2 + e^{\Phi}(\Phi - 1) + 1\, \dx, \quad m:=\rho u,\quad  P(\rho):= \rho \log \rho, 
\end{align}
is formally conserved for smooth solutions to \eqref{eq : ionic EP}. The first three terms in \eqref{eq:F} are convex in $(\rho, m, \nabla \Phi)$, but the term $e^{\Phi}(\Phi - 1) + 1$ is not convex in $\Phi$; in fact, the mapping
\begin{equation}\label{eq: nonconvex}
\Phi \mapsto e^\Phi(\Phi-1)+1
\end{equation}
is concave for $\Phi < -1$. This lack of convexity poses a significant obstacle in applying the classical modulated energy method, since it prevents interpreting the energy functional as a mathematical distance between two states of the system. Such an interpretation is essential for deriving robust quantitative stability or convergence estimates in singular limit problems.

A key observation is that the term $e^{\Phi}$ corresponds to the \emph{density of massless electrons}, as given by \eqref{eq: rhoe ePhi}. Introducing the notation
\[
\rho_e:=e^\Phi,
\]
allows us to rewrite \eqref{eq: nonconvex} as
\[
\rho_e \mapsto \rho_e(\log\rho_e-1)+1 = e^\Phi(\Phi-1)+1,
\]
which is convex in $\rho_e$. This suggests enlarging the state space by considering the extended variable
\[
\tilde{U}:= \begin{pmatrix}
\rho ,
m,
\nabla \Phi, 
\rho_e 
\end{pmatrix}^T,
\]
so that the energy functional
\[
\tilde{U} \mapsto F(\tilde{U}):= \frac{|m|^2}{2\rho} + \kappa P(\rho) + \frac{1}{2} |\nabla \Phi|^2 + \rho_e(\log\rho_e-1)+1
\]
is convex in $\tilde{U}$. Based on this convexity, we define the modulated energy via the first-order Taylor expansion of $F$ around the target state:
\begin{equation}
\label{eq: first-order expansion of eta}
\begin{split}
F(\tilde U^\tau | \tilde U)&:=F(\tilde{U}^\tau) - F(\tilde{U}) - \textnormal{d}F(\tilde{U}) \cdot (\tilde{U}^\tau-\tilde{U}) \\
&=\frac{\rho^\tau}{2}|u^\tau-u|^2 + \kappa P(\rho^\tau|\rho) + \frac{1}{2}|\nabla\Phi^\tau - \nabla \Phi|^2 + P(\rho_e^\tau|\rho_e),
\end{split}
\end{equation}
where $\textnormal{d}F$ denotes the derivative of $F$ with respect to $\tilde U$, and $P(\cdot | \cdot)$ stands for the modulated pressure defined as
\[
P({\rho}^\tau |\rho) := {\rho}^\tau \log({\rho}^\tau/\rho) - ({\rho}^\tau - \rho) .
\]
In particular, we obtain
\begin{align*}
P(\rho_e^\tau|\rho_e) = \rho_e^\tau \log(\rho_e^\tau/\rho_e) - (\rho_e^\tau - \rho_e) = e^{\Phi^\tau}(\Phi^\tau - \Phi) -(e^{\Phi^\tau}-e^\Phi).
\end{align*}
The first two terms in \eqref{eq: first-order expansion of eta} correspond to the classical relative entropy $\eta(U^\tau | U)$ (see e.g., \cite{berthelinvasseur2005}). The remaining terms capture the contribution from the electric potential and electron density. Notably, the expression $e^{\Phi^\tau}(\Phi^\tau - \Phi) - (e^{\Phi^\tau} - e^\Phi)$ has appeared in previous studies, such as the work of Han-Kwan \cite{hankwan2011} in the context of quasineutral limits. This reformulation in terms of $\rho_e$ is natural, given that the ionic Euler--Poisson equations assume an isothermal pressure law for electrons. The rest of the proof focuses on the time evolution of the modulated energy.

\medskip
$\bullet$ {\bf Velocity field for thermalized electrons.} To effectively handle the contribution of the electric potential in the time derivative of the modulated energy, we utilize the physical interpretation $\rho_e = e^{\Phi}$ and revisit the structure of the original ion-electron model \eqref{coupled ion electron}, from which the ionic Euler--Poisson system is formally derived. 

In the coupled model, electrons satisfy a mass conservation law, which suggests that even in the massless regime, it is meaningful to associate a velocity field with the electron density. We therefore introduce an auxiliary velocity field $u_e \in L^\infty(0,T;C^1(\T^d))$ satisfying
\begin{equation}
\label{eq:u_e}
\pa_t \rho_e + \nabla_x\cdot(\rho_e u_e) = 0, \quad \rho_e:=e^\Phi.
\end{equation}
This choice allows us to make use of the hidden conservation structure present in \eqref{coupled ion electron}. The existence and properties of such a velocity field are rigorously justified in Proposition \ref{lem:u_e}, and the resulting formulation yields a modulated energy identity consistent with all physically relevant macroscopic quantities, as seen in equation \eqref{eq:fml}.

\medskip
$\bullet$ {\bf Control of critical nonlinear term.} The introduction of $u_e$ facilitates the reformulation of nonlinear terms and reveals the structural parallel between the kinetic and hydrodynamic models. Nonetheless, the most delicate part in the analysis is still the control of error terms involving $\rho_e^\tau = e^{\Phi^\tau}$, particularly due to the lack of uniform bounds on $\rho_e^\tau$ as $\tau \to 0$. Among these, a major difficulty lies in estimating the term
\[
\intt (\rho_e^\tau - \rho_e)(u - u_e)\cdot \nabla(\Phi^\tau-\Phi) \, \dx, 
\]
which cannot be handled directly using standard relative entropy methods. The difficulty is particularly significant when $\rho_e^\tau$ becomes large, since in that regime the exponential nonlinearity amplifies the error and prevents us from directly closing the energy estimate.

To overcome this, we exploit the identity $\Phi^\tau = \log \rho_e^\tau$ and reformulate the integrand in terms of the relative entropy $P(\rho_e^\tau|\rho_e)$ and its derivative. The key idea is to regard the difference $u - u_e$ as a regular and bounded velocity field, and transfer derivatives onto this ``good'' term through integration by parts. This yields an expression that can be controlled by relative entropy and energy dissipation terms.

However, since the behavior of the ratio $\rho_e^\tau/\rho_e$ differs in the regions where it is greater or less than one, we introduce a smooth cutoff to separate the domain accordingly. This allows us to handle the contribution from both regimes in a unified framework, while ensuring that the integration by parts remains valid. The resulting estimate, proven in Lemma \ref{lem:key}, shows that the problematic term is indeed bounded by a linear combination of $P(\rho_e^\tau|\rho_e)$ and $|\nabla(\Phi^\tau - \Phi)|^2$.

This careful decomposition and reformulation are essential in deriving a Grönwall-type inequality, which ultimately yields the convergence estimate
\[
\sup_{t\in [0,T]} \calF(\tilde U^\tau| \tilde U) \lesssim \sqrt \tau.
\]

%%%%%%%%%%%%%%%%%%%%%%%%%%%%%%%%%%%%%%%%%%%%%%%%%%%%%%%%%

%%%%%%%%%%%%%%%%%%%%%%%%%%%%%%%%%%%%%%%%%%%%%%%%%%%%%%%%%

\subsection{Main result II: global existence of kinetic models} 

Theorem \ref{thm:hdr} establishes a rigorous hydrodynamic limit from the VPFP system \eqref{main eq} to the ionic Euler--Poisson system \eqref{eq : ionic EP}, under the assumption that both systems admit sufficiently regular solutions. In particular, we require the existence of a regular solution to the limiting system \eqref{eq : ionic EP} and a weak solution to the VPFP system \eqref{main eq} at least locally in time. Since the local well-posedness theory for \eqref{eq : ionic EP} is well understood \cite{lanneslinaressaut2012} (see Lemma \ref{prop:lwp}), our focus in this subsection is to complete the justification of Theorem \ref{thm:hdr} by establishing the global-in-time existence of weak entropy solutions to \eqref{main eq}.
 
We now state our second main theorem:
\begin{theorem} \label{thm:gws}
 Let $d\ge 2$ and $\kappa\ge 0$. Suppose the initial data $f_0$ satisfies
    \begin{align*}
        0 \le f_0 \in L^\infty(\T^d\times\R^d), \qquad \iinttr  (1+|v|^2) f_0 \,\dx\dv < + \infty.
    \end{align*}
Then there exists a global weak solution $f \ge 0$ to \eqref{main eq} such that:
    \begin{align*}
    \begin{cases}
        (1+|v|^2 + \kappa |\log f|)f \in L^\infty([0,\infty);L^1(\T^d\times\R^d)),\\
        f\in L^\infty_{\rm loc}([0,\infty)\times\T^d; L^\infty(\R^d)), \\
        \kappa \nabla_v f \in L^2_{\rm loc}([0,\infty)\times\T^d; L^2(\R^d)), 
        \end{cases} 
    \end{align*}
and the equation \eqref{main eq} is satisfied in the sense of distributions:
    \begin{equation}\label{eq:weakform}
    \begin{split}
        &\iinttr f(t,x,v)\phi(t,x,v) \,\dx\dv - \iinttr f_0(x,v)\phi(0,x,v)\,\dx\dv \\
        &\quad = \int_0^t \iinttr f (\p_s \phi + v\cdot \nabla_x \phi + E[f]\cdot \nabla_v \phi) \,\dx\dv\ds \\
        &\qquad + \dfrac{1}{\tau} \displaystyle \int_0^t \iinttr -\nabla_v\phi \cdot \Big(\kappa \nabla_v f - ((u_f- v)f) \Big) \,\dx\dv\ds 
    \end{split} \end{equation}
for each $t\in [0,\infty)$ and $\phi\in C_c^\infty([0,t] \times\T^d\times\R^d)$.
Moreover, $f$ satisfies the kinetic entropy inequality \eqref{eq : KEI} for almost every $t\ge 0$.
\end{theorem}

\begin{remark}  \label{rem: dim}
We claim here that the weak solution $f$ constructed in Theorem \ref{thm:gws} satisfies the integrability condition \eqref{sr:h} of Theorem \ref{thm:hdr}, at least in dimensions $d=2,3$. Indeed, from $f \in L^\infty_{\rm loc}([0,\infty);L^\infty(\T^d \times \R^d))$ and the uniform second moment bounds \eqref{eq: moments tau}, interpolation shows $\rho_f \in L^\infty_{\rm loc}([0,\infty);L^{\frac{d+2}{d}}(\T^d))$. Elliptic regularity and the Sobolev embedding then yield
\[
\nabla\Phi[f] \in L^\infty_{\rm loc}([0,\infty);L^{q_d}(\T^d)),
\]
where $q_d$ can be arbitrarily large for $d=2$, and $q_d = 15/4$ for $d=3$, which ensures that the product $\rho_f \nabla \Phi[f]$ belongs to $L^\infty_{\mathrm{loc}}([0,\infty); L^1(\T^d))$ and is thus well-defined in the sense of distributions. For higher dimensions $d \ge 4$, however, we face a critical obstruction. Indeed, we only have 
\[
\frac{d+2}{d} \le \frac{3}{2} \quad \mbox{and} \quad q_d \le \frac{12}{5}, 
\]
which yields 
\[
\frac{d}{d+2} + \frac{1}{q_d} \ge \frac{13}{12}>1. 
\]
As a result, we cannot guarantee that the term $\rho_f \nabla \Phi[f]$ is defined in the distributional sense, and the present analysis does not extend to $d \ge 4$. %This dimensional restriction is a well-known technical barrier in the study of kinetic equations with self-consistent fields, and overcoming it would likely require new structural estimates or refined compactness tools.
\end{remark}

The existence theory for kinetic equations of Fokker--Planck type has been extensively developed in various settings, particularly in the absence of nonlinear coupling or force fields. Let us begin by revisiting works on the classical linear Fokker--Planck operator
\begin{align*}
    \mathscr{L}[f] := \nabla_v\cdot (\nabla_v f + vf),
\end{align*}
whose analytic properties have been thoroughly studied. Foundational work on hypoellipticity by H\"ormander \cite{Hormander1967} revealed that the operator $\mathscr{L}$ exhibits strong smoothing effects, owing to its structure of non-trivial commutators. These hypoelliptic properties have since motivated a large body of work on hypocoercivity and regularity for kinetic equations; we refer to \cite{Bouchut2002, CarrapatosoMischler, GolseImbertMouhutVasseur2019, Zhu2024, Zhu2025} for further developments, including refined velocity averaging techniques and quantitative convergence estimates.

In the presence of an electrostatic field determined by the Poisson equation $-\Delta \Phi[f] = \rho_f$, the corresponding linear VPFP system has also received considerable attention. In this context, the global existence of weak solutions and their long-time behavior have been established under various geometric configurations, including periodic domains and bounded settings \cite{bonillacarrillosoler1997, Carrillo1998}. These works rely heavily on entropy dissipation methods combined with elliptic regularity for the field equations.

Regarding the nonlinear operator $\eqref{CO}$, the analysis becomes significantly more intricate. In \cite{Choi2016}, global classical solutions were constructed in a perturbative framework.  
Although the force field is not present in that setting, the result already captures key nonlinear features relevant to our model. More delicate situations, involving non-constant temperature and additional nonlinearities in both drift and diffusion, were addressed in \cite{choihwangyoo}, again without the presence of a self-consistent field.

Theorem \ref{thm:gws} extends these previous results by establishing the global existence of weak entropy solutions to a nonlinear and self-consistent kinetic equation with Fokker--Planck-type diffusion, nonlocal interaction, and nonlinear relaxation.

%%%%%%%%%%%%%%%%%%%%%%%%%%%%%%%%%%%%%%%%%%%%%%%%%%%%%%%%%

%%%%%%%%%%%%%%%%%%%%%%%%%%%%%%%%%%%%%%%%%%%%%%%%%%%%%%%%%
\subsubsection*{Proof strategy for Theorem \ref{thm:gws}}

We now outline the main steps of the proof of Theorem \ref{thm:gws}. The strategy follows a multi-step approximation and compactness procedure, carefully adapted to ensure the propagation of key energy and moment bounds, and the preservation of the entropy structure.

\medskip

$\bullet$ {\bf Regularized system.} As a first step, we introduce a regularized system, which allows us to handle both the nonlinearity of the drift $u_f$ and the electric field $E[f]$. In treating the nonlinear Fokker--Planck operator $\calN_\kappa[f]$, we follow the idea of \cite{karpermellettrivisa2013}, with one subtle difference being that we additionally impose a regularization of the truncated velocity via convolution. This is done so as to ensure the sufficient regularity of the drift term. This leads us to use two parameters $\e, \delta > 0$, where $\ve$ is used to truncate the drift while $\delta$ is imposed to enhance the regularity of the drift. For the electric field $E[f]$, we adopt the so-called \emph{double regularization} of the electric field using the parameter $\ve$, inspired by the works \cite{griffinpickeringiacobelli2021a, horst1990}. We elaborate on the advantage of this regularization method in the paragraphs to follow. 

We also remark that in the case $\kappa=0$, the lack of diffusion generally makes stability estimates difficult. Hence, in that case, the parameter $\delta$ is used as a vanishing viscosity coefficient so that both cases $\kappa=0$ and $\kappa>0$ can be treated in a uniform framework.

\medskip

$\bullet$ \textbf{Passing to $\delta \to 0$ and uniform-in-$\ve$ estimates.} Once the approximate solutions $f^{\ve,\delta}$ are constructed, we pass to the limit one by one. We first consider the limit $\delta \to 0$, which is only tied with the smoothing of the truncated bulk velocity via convolution (and that of the vanishing viscosity when $\kappa=0$). This convolution in fact breaks the symmetry of the system, and thus we do not have an entropy inequality that holds at the $\ve,\delta$-level. However, uniform-in-$\delta$ estimates are relatively easy to obtain owing to the fact that the additional regularization in $\e$ is still present. By employing velocity averaging lemmas and temporal equicontinuity arguments, we obtain suitable compactness of the distributions and their moments, and consequently pass to the limit $\delta\to 0$ to obtain a solution $f^\ve$ at the $\ve$-regularized level. In particular, the time-equicontinuity estimates, combined with Arzel\`a-Ascoli type arguments, provide a subsequence for which $f^{\ve,\delta}(t) \weakto f^\ve(t)$ weak-$*$, uniformly in time.

After taking $\delta\to 0$, the aforementioned double regularization of the electric field provides enough symmetry so that an approximate entropy inequality can be obtained at the $\ve$-level. This inequality plays a crucial role, since it provides uniform estimates that are vital in passing to the limit $\ve\to 0$, and moreover obtaining the kinetic entropy inequality \eqref{eq : KEI}.

\medskip

$\bullet$ {\bf Compactness and passing to $\ve \to 0$.} Our final step lies in passing to the limit as $\e \to 0$. The entropy bounds mentioned in the previous paragraph permit the application of velocity averaging lemmas, which provide convergence of the macroscopic variables. Based upon a close inspection of the uniform-in-$\ve$ estimates, we also find that $\{f^\ve\}$ is equicontinuous as elements of $C([0,T]; W^{-2,p}_{\rm loc}(\T^d\times\R^d))$ for a suitable $1<p<\infty$. As a result, we are again able to pass to a subsequence (independent of $t$) such that $f^\ve(t)$ is weak-$*$ compact for every $t\ge 0$. This is especially important since the kinetic entropy functional $\calE[f]$ is defined for each $t$, and our goal lies in its validation at the limit. This type of argument is similar to that of Griffin--Pickering and Iacobelli \cite{griffinpickeringiacobelli2021a}, but we adapt it to accommodate the lower regularity of the electric field that arises from merely assuming bounded second moments.

We remark that in the case $\kappa > 0$, strong compactness can be obtained by employing recent results of Sampaio \cite{sampaio2024}. This facilitates the identification of nonlinear terms, which simplifies the limiting procedure. However, such strong compactness results do not apply when $\kappa = 0$, due to the lack of diffusion. Therefore, we apply Sampaio's result only in verifying the entropy inequality in the case $\kappa > 0$, particularly in identifying weak limits involving $\nabla_v \sqrt{f}$. Our strategy, based only upon entropy bounds, velocity averaging, and temporal regularity, thus provides a uniform framework for constructing global weak entropy solutions in both the diffusive and non-diffusive regimes.

%%%%%%%%%%%%%%%%%%%%%%%%%%%%%%%%%%%%%%%%%%%%%%%%%%%%%%%%%
%
%
%
%
%
%%%%%%%%%%%%%%%%%%%%%%%%%%%%%%%%%%%%%%%%%%%%%%%%%%%%%%%%%
\subsection{Organization of the paper}
The remainder of the paper is structured as follows. In Section \ref{sec : hydrodynamic limit}, we establish the modulated energy estimate \eqref{eq: conv mod egy} and complete the proof of the hydrodynamic limit result stated in Theorem \ref{thm:hdr}. Section \ref{sec : weak solutions fokker planck} is devoted to the construction of global-in-time weak entropy solutions to the ionic VPFP system  \eqref{main eq}, thereby proving Theorem \ref{thm:gws}. Finally, Appendix \ref{app: QE} provides the proof of the quantitative convergence results \eqref{eq: tauconv l2}--\eqref{eq: conv of dist g1 k=0} in Theorem \ref{thm:hdr}.

%%%%%%%%%%%%%%%%%%%%%%%%%%%%%%%%%%%%%%%%%%%%%%%%%%%%%%%%%
%
%
%
%
%
%%%%%%%%%%%%%%%%%%%%%%%%%%%%%%%%%%%%%%%%%%%%%%%%%%%%%%%%%

\section{Hydrodynamic limits to ionic Euler--Poisson equations} \label{sec : hydrodynamic limit}
%%%%%%%%%%%%%%%%%%%%%%%%%%%%%%%%%%%%%%%%%%%%%%%%%%%%%%%%%
%
%
%
%
%
%%%%%%%%%%%%%%%%%%%%%%%%%%%%%%%%%%%%%%%%%%%%%%%%%%%%%%%%%

\subsection{Auxiliary electron velocity field $u_e$ in the limiting system}
As discussed in the introduction, one of the main difficulties in establishing the modulated energy estimates lies in controlling the contribution from the electron density. To overcome this issue, we introduce an auxiliary velocity field $u_e$ associated with the electrons, defined as a solution to the continuity equation \eqref{eq:u_e}. In this subsection, we justify the well-posedness and regularity of $u_e$, starting with the following proposition:
\begin{proposition}\label{lem:u_e}
Let $(\rho,u,\Phi)\in C([0,T];H^{s}(\T^d)\times H^{s+1}(\T^d)\times H^{s+1}(\T^d))$ denote the unique solution to the ionic Euler--Poisson system \eqref{eq : ionic EP}.
Set $\rho_e := e^{\Phi}$. Then there exists a unique  velocity field $u_e\in L^\infty(0,T;H^{s+1}(\T^d))$ satisfying the continuity equation  
\begin{align*}
    \p_t \rho_e + \nabla_x \cdot (\rho_e u_e) = 0.
\end{align*}
\end{proposition}

The regularity of $(\rho, u, \Phi)$ in the above statement is ensured by the local well-posedness theory for the ionic Euler--Poisson system \eqref{eq : ionic EP}. We recall this result for completeness: 
\begin{lemma}\label{prop:lwp}
 Let $s>\frac{d}{2}+1$ and let $(\rho_0,u_0) \in H^s(\T^d)\times H^{s+1}(\T^d)$ satisfy
 \[
 \inf_{x\in \T^d}\rho_0(x)>0.
 \]
Then there exists a time $T > 0$ such that the ionic Euler--Poisson system \eqref{eq : ionic EP} admits a unique classical solution $(\rho, u, \Phi)$ on $[0,T]$ with initial data $(\rho_0, u_0)$, satisfying the regularity and lower bound stated in \eqref{eq:iEP:lss}.
\end{lemma}    

We refer to \cite{lanneslinaressaut2012} for the proof. The construction of $u_e$ then reduces to solving a divergence equation of the form $\nabla \cdot \bfv = f$, which we address using the following classical result:
\begin{lemma}\label{lem:bog}
Let $d\ge2$, $s\ge0$. If $f \in H^s(\T^d)$ satisfies
\[
\displaystyle \int_{\T^d} f \,\dx =0,
\]
then there exists $\bfv \in H^{s+1}(\T^d;\R^d)$ such that 
\[
\nabla \cdot \bfv = f.
\]
\end{lemma}

\begin{proof}
We consider the Poisson equation $-\Delta \psi = f$ on $\T^d$, whose unique solution is expressed through the Green's function for the Laplace equation as $\psi = G * f$ (see \cite{griffinpickering2024,titchmarsh1958} for a description of $G$ and its properties).  Taking Fourier transforms, we obtain
\[
4\pi^2|\bfk|^2 \hat{\psi}(\bfk) =  \hat{f}(\bfk), \quad \forall \, \bfk \in \Z^d. 
\]
This yields
\[
16\pi^4\sum_{\bfk \in \Z^d} |\bfk|^{2(s+2)} |\hat{\psi}(\bfk)|^2 = \sum_{\bfk \in \Z^d} |\bfk|^{2s} |\hat{f}(\bfk)|^2,
\]
which implies $\psi \in H^{s+2}(\T^d)$. Defining $\bfv := -\nabla \psi$, we conclude that $\nabla\cdot \bfv = f$ and $\bfv \in H^{s+1}(\T^d)$ as desired.
\end{proof}

Next, we establish a simple yet useful regularity estimate for the time derivative of the electron density $\rho_e := e^\Phi$, which will be instrumental in constructing the auxiliary velocity field $u_e$ in the limiting system.
\begin{lemma}
\label{lem: p_t phi regularity}
    Let $(\rho,u,\Phi)\in C([0,T];H^{s}(\T^d)\times H^{s+1}(\T^d)\times H^{s+1}(\T^d))$ be a classical solution to the ionic Euler--Poisson system \eqref{eq : ionic EP}. Then the following regularity properties hold:
    \begin{equation*}
    \begin{split}
        &e^{\Phi} \in C([0,T];H^{s+1}(\T^d)) \quad \mbox{and} \quad \p_t \Phi \in L^\infty(0,T;H^{s+1}(\T^d)).
        \end{split}
    \end{equation*}
\end{lemma}
\begin{proof}
The regularity of $e^\Phi$ follows directly from Moser-type estimates and the assumed regularity of $\Phi$. To control $\p_t \Phi$, we differentiate the Poisson--Boltzmann equation in time:
    \begin{align}
    \label{eq: poisson--boltzmann for p_t Phi}
        -\Delta \p_t \Phi = - e^{\Phi} \p_t \Phi + \p_t \rho .
    \end{align}
We first show $\p_t \Phi \in L^\infty(0,T;H^1(\T^d))$. Multiplying by $\p_t \Phi$ and integrating, we obtain the basic energy identity:
    \begin{align*}
        \intt |\nabla \p_t \Phi|^2 \,\dx + \intt e^{\Phi} |\p_t \Phi|^2\, \dx &= \intt \p_t\Phi \, \p_t \rho \,\dx.
    \end{align*}
Using that $e^\Phi \ge C[\rho] > 0$ from \cite[Lemma 2.15]{choikoosong2025}, we estimate
    \begin{align*}
        \intt |\nabla \p_t \Phi|^2 \,\dx + C[\rho] \intt |\p_t \Phi|^2\, \dx &= \intt \p_t\Phi \, \p_t \rho \,\dx  \le \intt \frac{C[\rho]}{2}|\p_t \Phi|^2 + \frac{1}{2C[\rho]} |\p_t \rho|^2\, \dx.
    \end{align*}
    By absorbing to the left-hand side and using $\eqref{eq : ionic EP}_1$, we deduce
    \begin{align*}
        \|\p_t \Phi\|_{H^1(\T^d)} \lesssim_\rho \|\p_t \rho\|_{L^2} = \|\rho u\|_{\dot H^1}.
    \end{align*}
Next, we demonstrate $L^p$ a priori estimates for $\p_t \Phi$. From \eqref{eq: poisson--boltzmann for p_t Phi}, we obtain for any $p\in [1,\infty)$
\begin{align*}
    \intt |\p_t \Phi|^p \,\dx &= \intt (\Delta \p_t \Phi) \p_t \Phi |\p_t \Phi|^{p-2} e^{-\Phi} + \p_t \rho \, \p_t \Phi |\p_t \Phi|^{p-2} e^{-\Phi}\,\dx \\
    &= -(p-1)\intt |\nabla \p_t \Phi|^2 |\p_t \Phi|^{p-2} e^{-\Phi} \,\dx  + \intt \p_t \Phi |\p_t \Phi|^{p-2} \nabla \p_t \Phi \cdot \nabla \Phi \, e^{-\Phi} \,\dx\\
    &\quad + \intt \p_t \rho \, \p_t \Phi |\p_t \Phi|^{p-2} e^{-\Phi}\, \dx \\
    &\le \intt |\p_t \Phi|^{p-1} |\nabla \p_t \Phi| \, |\nabla \Phi| e^{-\Phi} \,\dx + \intt |\p_t \rho| \, |\p_t \Phi|^{p-1} e^{-\Phi}\,\dx.
\end{align*}
For the first term on the right-hand side, we use H\"older and Young's inequalities to get
\begin{align*}
\intt |\p_t \Phi|^{p-1} |\nabla \p_t \Phi| \, |\nabla \Phi| e^{-\Phi} \,\dx &\le \|\p_t \Phi\|_{L^\infty}^{p-1} \|e^{-\Phi}\|_{L^\infty} \|\p_t \Phi\|_{H^1} \|\nabla \Phi\|_{L^2}  \cr
&\le C \|\p_t \Phi\|_{L^\infty}^{p-1} \cr
&\le \frac{1}{3}\|\p_t \Phi\|_{L^\infty}^p + 3^{p-1} C^p.
\end{align*}
In a similar fashion, we also estimate the second term as
\begin{align*}
 \intt |\p_t \rho| \, |\p_t \Phi|^{p-1} e^{-\Phi}\,\dx \le \|e^{-\Phi}\|_{L^\infty}\|\p_t \rho\|_{L^p}\|\p_t \Phi\|_{L^p}^{p-1} \le \frac{1}{3}\|\p_t \Phi\|_{L^p}^p + 3^{p-1}C^p .
\end{align*}
Combining the above gives
\begin{align*}
    \|\p_t \Phi\|_{L^p} \le \frac{1}{2}\|\p_t \Phi\|_{L^\infty} + 3C,
\end{align*}
and letting $p\to\infty$, we have $\|\p_t \Phi\|_{L^\infty((0,T)\times\T^d)}\le 6C$.

Finally, we apply a bootstrapping argument. Since both $\p_t \rho$ and $e^\Phi \p_t \Phi$ belong to $H^s(\T^d)$, the right-hand side of \eqref{eq: poisson--boltzmann for p_t Phi} lies in $H^s(\T^d)$ as well. By elliptic regularity, $\p_t \Phi \in H^{s+2}(\T^d)$, and hence by iterating this argument with commutator estimates, we obtain
\[
\p_t \Phi \in L^\infty(0,T;H^{s+1}(\T^d)).
\]
This completes the proof.
\end{proof}

We now return to the proof of Proposition \ref{lem:u_e}, which asserts the existence of a well-defined velocity field $u_e$ satisfying the continuity equation for $\rho_e = e^\Phi$.

\begin{proof}[Proof of Proposition \ref{lem:u_e}]
We consider the following problem:
\begin{equation}\label{eq:bog}
\nabla \cdot \bfv = -\pa_t \rho_e = - e^\Phi \pa_t\Phi
\end{equation}
together with the neutrality condition
\[
\intt \pa_t \rho_e \,\dx = \frac{\textnormal{d}}{\dt}\intt \rho\,\dx = \frac{\textnormal{d}}{\dt} \intt e^{\Phi}\, \dx =  0.
\]
It follows from Lemma \ref{lem: p_t phi regularity} that
\[
e^\Phi \pa_t\Phi \in L^\infty(0,T; H^{s}(\T^d)).
\]
Thus, Lemma \ref{lem:bog} yields a divergence-free correction vector field $\bfv \in L^\infty(0,T; H^{s+1}(\T^d))$ satisfying the above equation. Setting $u_e := e^{-\Phi} \bfv$, we conclude
\[
u_e \in L^\infty(0,T;H^{s+1}(\T^d)).
\]
This completes the proof.
\end{proof}

\begin{remark}\label{rmk:happy}
By Sobolev inequalities, the velocity field $u_e$ to \eqref{eq:bog} as constructed above satisfies
\[
u_e \in L^\infty(0,T;C^1(\T^d)),
\]
which we crucially utilize in estimating the modulated energy.
\end{remark}
%%%%%%%%%%%%%%%%%%%%%%%%%%%%%%%%%%%%%%%%%%%%%%%%%%%%%%%%%
%
%
%
%
%
%%%%%%%%%%%%%%%%%%%%%%%%%%%%%%%%%%%%%%%%%%%%%%%%%%%%%%%%%
\subsection{Modulated energy evolution and its structural decomposition}\label{sec:mod}
In this subsection, we derive a precise evolution identity for the modulated energy functional, which serves as the cornerstone for proving the quantitative bound \eqref{eq: conv mod egy} in Theorem \ref{thm:hdr}. The computation relies on a structural comparison between the kinetic formulation and the limiting Euler--Poisson system. In particular, we exploit the local conservation laws derived from the kinetic equation and the conservative structure of the limiting system to analyze the rate of change of the modulated energy. 

To simplify notation, we denote the macroscopic quantities associated with $f^\tau$ as
\[
(\rho^\tau,u^\tau,\Phi^\tau) := (\rho_{f^\tau},u_{f^\tau}, \Phi[f^\tau]).
\]

The next lemma provides a quantitative identity for the time evolution of the modulated energy, which will be the starting point of our asymptotic analysis. 

\begin{lemma}\label{lem:fml}
Let $f^\tau$ be a weak solution to the kinetic system \eqref{main eq}, satisfying the weak formulation \eqref{eq:weakform} and the kinetic entropy inequality \eqref{eq : KEI}, along with the integrability condition \eqref{sr:h}. Let $(\rho, u, \Phi)$ be a classical solution to \eqref{eq : ionic EP}. Then the modulated energy functional satisfies the identity
\[%\label{eq:fml}
\begin{split}
\calF(\tilde U^\tau| \tilde U) &=  \calF( \tilde U^\tau_0| \tilde U_0)  + \Big(\calF(\tilde U^\tau) - \calF(\tilde U^\tau_0) \Big)  +\int_0^t \intt \nabla u : \intr v\otimes \{(u^\tau - v)f^\tau - \kappa \nabla_v f^\tau\} \, \dv \dx \ds\\
&\quad - \int_0^t\intt \nabla u : \rho^\tau(u^\tau -u)\otimes (u^\tau -u) \,\dx \ds  +\int_0^t\intt \{-\Delta (\Phi^\tau - \Phi) u\} \cdot \nabla (\Phi^\tau - \Phi) \,\dx\ds \\
&\quad +\int_0^t\intt (\rho_e^\tau - \rho_e)(u - u_e)\cdot \nabla(\Phi^\tau -\Phi) \,\dx \ds .
\end{split}
\]
\end{lemma}

\begin{remark}
The last term in the identity,
\[
\int_0^t\intt (\rho_e^\tau - \rho_e)(u - u_e)\cdot \nabla(\Phi^\tau -\Phi) \,\dx \ds
\]
is unique to the ionic model and does not appear in the electronic Vlasov--Poisson setting. The presence of the auxiliary velocity field $u_e$ clarifies the structure of the asymptotic interaction between ions and electrons, and facilitates the dissipation analysis in the modulated energy framework.
\end{remark}

\begin{proof}[Proof of Lemma \ref{lem:fml}]

In order to clearly deliver our ideas, we provide the computations here at a formal level. We leave comments justifying these calculations in Remark \ref{rmk:justify} below.

The goal is to compute the time derivative of the difference $\calF(\tilde U^\tau| \tilde U) - \calF(\tilde U^\tau)$, which measures the deviation of the kinetic solution from the limiting hydrodynamic state $(\rho,u,\Phi)$. To this end, we introduce the modulated energy $\eta(U^\tau|U)$ and use the shorthand notations
\begin{align*}
    &\eta(U^\tau|U) := \frac{\rho^\tau}{2}|u^\tau - u|^2 + \kappa P(\rho^\tau|\rho), \quad  \rho_e := e^{\Phi}, \qquad \rho_e^\tau := e^{\Phi^\tau},
\end{align*}
with $\eta(U)$ defined as in \eqref{def: eta}. From the first-order expansion \eqref{eq: first-order expansion of eta}, we write the modulated free energy as
\begin{align}
\label{def: B_gamma kappa}
    \calF(\tilde U^\tau|\tilde U) = \eta(U^\tau|U) + P(\rho_e^\tau|\rho_e) + \frac{1}{2}|\nabla \Phi^\tau - \nabla \Phi|^2.
\end{align}
Thus, we are led to evaluate
\[%\begin{equation} \label{eq: want}
\begin{split}
    \calF(\tilde U^\tau | \tilde U) - \calF(\tilde U^\tau) &= \eta(U^\tau|U) - \eta(U^\tau)  \cr
    &\quad + P(\rho_e^\tau | \rho_e) - P(\rho_e^\tau)   \cr
    &\quad + \frac{1}{2}|\nabla(\Phi^\tau - \Phi)|^2 - \frac{1}{2}|\nabla\Phi^\tau|^2\cr
    &=: I + II +III.
\end{split}
\]%\end{equation}
We begin by analyzing the first term $I$. The limiting system \eqref{eq : ionic EP} can be written in conservative form. Introducing
\begin{align*} %\label{eq: conservative form}
    &U := \begin{pmatrix} \rho \\ m \end{pmatrix}, \quad m:=\rho u,
    \quad A(U) := \begin{pmatrix} m \\  \frac{m\otimes m}{\rho} + \kappa\rho  \mathbb{I} \end{pmatrix} , \quad R(U) := \begin{pmatrix}
    0 \\  - \rho \nabla \Phi \end{pmatrix},
\end{align*}
we have
\begin{align*}
    \pa_t U + {\rm{div}}_x A(U) = R(U).
\end{align*}

We now apply standard relative entropy estimates (see e.g., \cite[Lemma 3.1]{carrillochoijung2021}), yielding
\[
\begin{split}
\ddt \intt \eta(U^\tau|U)\,\dx -\intt \pa_t\lt(\eta(U^\tau)\rt) 
& = - \intt \nd \eta(U)\cdot (\pa_t U^\tau + {\rm{div}}_x A(U^\tau) - R(U^\tau))\, \dx \\
&\quad +\intt \lt[ \nd^2 \eta(U) {\rm{div}}_x A(U) \cdot (U^\tau - U) + \nd\eta(U){\rm{div}}_x A(U^\tau )\rt] \dx \\
&\quad -\intt \lt[\nd^2\eta(U)R(U)\cdot (U^\tau -U)+ \nd \eta(U)R(U^\tau )\rt]  \dx \\
&=: I_1 + I_2 + I_3.
\end{split}
\]
Here $\textnormal{d}\eta$ denotes the differential of $\eta$ with respect to the conservative variable $U$. A direct computation gives
\begin{align}\label{eq: deta}
    \textnormal{d}\eta(U) = \begin{pmatrix}
        -\frac{|\rho u|^2}{2\rho^2} + \log \rho + 1 \\ u
    \end{pmatrix}, 
    \quad \textnormal{d}^2\eta(U) = \begin{pmatrix}
        * & -\frac{\rho u}{\rho^2} \\[6pt]
        * & \frac{1}{\rho}
    \end{pmatrix}.
\end{align}

We now compute each term in turn. For $I_1$, multiplying the kinetic equation \eqref{main eq} by $(1,v)$ and integrating in $v$, we deduce
\[
%\begin{split}
I_1 = \intt \nabla u :  \lt( \int_{\R^d} v \otimes \lt\{ (u^\tau - v) f^\tau - \kappa\nabla_v f^\tau\rt\} \dv \rt)\dx.
%\end{split}
\]
For $I_2$, we apply the standard identity for hyperbolic systems (see \cite{dafermos1979, KMT2015}) to obtain
\[
%\begin{split}
I_2 = - \intt \nabla_x \textnormal{d}\eta(U):A(U^\tau | U )\,\dx  = - \intt \nabla u : \rho^\tau(u^\tau -u)\otimes (u^\tau -u)\, \dx.
%\end{split}
\]
For $I_3$, a direct calculation using \eqref{eq: deta} gives
\begin{equation*}
%    \begin{split}
        I_3 = \intt \rho^\tau u \cdot (\nabla\Phi^\tau - \nabla\Phi) + \rho^\tau u^\tau \cdot \nabla\Phi \,\dx.
%    \end{split}
\end{equation*}
 
Next, we estimate $II$ and $III$. For $II$, using the definition of $u_e$ from Proposition \ref{lem:u_e}, we deduce
\begin{align*}
\ddt II &= - \ddt \intt \Big\{ P(\rho_e) + P'(\rho_e) (\rho_e^\tau - \rho_e) \Big\} \,\dx\cr
&= - \intt \Big\{ P'(\rho_e) \p_t\rho_e + P''(\rho_e)\p_t\rho_e (\rho_e^\tau - \rho_e) + P'(\rho_e) \p_t(\rho_e^\tau - \rho_e) \Big\} \,\dx \\
&= - \intt P'(\rho_e)\pa_t\rho_e^\tau \,\dx - \intt P''(\rho_e)(\pa_t\rho_e)(\rho_e^\tau-\rho_e) \,\dx  \\
&= - \intt \log\rho_e \, \pa_t\rho_e^\tau \,\dx + \intt \nabla \cdot (\rho_e u_e) \lt(\rho_e^\tau/\rho_e -1\rt)\,\dx \\
&= - \intt \Phi \, \p_t \rho_e^\tau \,\dx - \intt (\rho_e^\tau u_e) \cdot \nabla (\Phi^\tau - \Phi) \,\dx.
\end{align*}
For $III$, using the Poisson--Boltzmann equation, we find
\begin{align*}
\ddt III &= \ddt \intt \Big\{-\nabla\Phi^\tau \cdot \nabla\Phi + \frac{1}{2}|\nabla\Phi|^2 \Big\}\,\dx \cr
&= - \intt \nabla \Phi \cdot \pa_t \nabla \Phi^\tau\,\dx - \intt \pa_t \nabla \Phi \cdot \nabla (\Phi^\tau - \Phi) \,\dx \\
&= - \intt \Phi \, \pa_t(\rho^\tau - \rho_e^\tau)\,\dx - \intt (\Phi^\tau -\Phi)\pa_t(\rho-\rho_e)\,\dx \\
&= - \intt \Phi \, \pa_t(\rho^\tau - \rho_e^\tau)\,\dx + \intt \nabla\cdot(\rho u - \rho_e u_e) \log(\rho_e^\tau /\rho_e)\,\dx \\
&= \intt \Phi \, \p_t\rho_e^\tau \,\dx - \intt \nabla\Phi \cdot \rho^\tau u^\tau \,\dx - \intt (\rho u - \rho_e u_e) \cdot \nabla (\Phi^\tau - \Phi) \,\dx .  
\end{align*}
Combining all terms, we note
\[
\begin{split}
I_3+\ddt II + \ddt III &=  \intt \Big\{ (\rho^\tau - \rho) u - (\rho_e^\tau - \rho_e) u_e \Big\}\cdot \nabla (\Phi^\tau - \Phi) \,\dx \\
&= \intt \Big\{(\rho^\tau -\rho_e^\tau - \rho + \rho_e)u + (\rho_e^\tau - \rho_e)(u - u_e) \Big\} \cdot \nabla(\Phi^\tau-\Phi) \,\dx \\
&= \intt \{-\Delta (\Phi^\tau - \Phi) u\} \cdot \nabla (\Phi^\tau - \Phi) + (\rho_e^\tau - \rho_e)(u - u_e)\cdot \nabla(\Phi^\tau -\Phi) \,\dx.
\end{split}
\]
Together with the previous computations for $I_1$ and $I_2$, this completes the derivation of the time evolution of $\calF(\tilde U^\tau| \tilde U) - \calF(\tilde U^\tau)$, and thus concludes the proof.
\end{proof}

\begin{remark}\label{rmk:justify}
The computations above can be justified by using the local conservation laws \eqref{eq : local balance laws}, and thus it is crucial here that \eqref{sr:h} is satisfied. Notice in particular that thanks to \eqref{sr:h}, the integral arising in $I_3$ is well-defined. Of the computational identities that are in the proof of Lemma \ref{lem:fml}, we refer to \cite{hankwan2011, kang2018} for the derivation of the identity satisfied by $I$. Below let us provide the justification of the computations for $II$ and $III$. We utilize the time-derivative of the Poisson--Boltzmann equation, which holds in the sense of distributions:
\begin{align*}
    -\p_t \Delta \Phi^\tau + \p_t \rho_e^\tau = - \nabla_x\cdot (\rho^\tau u^\tau) \quad \text{in}\quad \calD'([0,T]\times\T^d). 
\end{align*}
Since $\Phi$ is a strong solution, $\Phi = P'(\rho_e)$ is admissible as a test function into the above, which allows us to obtain
\begin{align*}
    &\int_0^t \intt -\nabla\Phi^\tau \cdot \nabla (P''(\rho_e) \p_t \rho_e) - \p_t(P'(\rho_e) ) \, \rho_e^\tau \,\dx\ds + \intt \Big(\nabla\Phi^\tau \cdot \nabla (P'(\rho_e)) + P'(\rho_e) \rho_e^\tau\Big)(t) \,\dx  \\
    &\quad -\intt \Big(\nabla\Phi^\tau \cdot \nabla (P'(\rho_e)) + P'(\rho_e) \rho_e^\tau\Big)(0) \,\dx \\
    &\qquad = \int_0^t \intt P''(\rho_e) \nabla \rho_e \cdot \rho^\tau u^\tau \,\dx\ds.
\end{align*}
Equivalently (from $P'(\rho_e) = \Phi$)
\begin{align*}
    &\intt \Big(\nabla\Phi^\tau \cdot \nabla \Phi + P'(\rho_e) \rho_e^\tau\Big)(t) - \Big(\nabla\Phi^\tau \cdot \nabla \Phi + P'(\rho_e) \rho_e^\tau\Big)(0) \,\dx \\
    &\quad = \int_0^t \intt -\nabla \Phi^\tau \cdot \nabla(P''(\rho_e)\p_t \rho_e) - \p_t (P'(\rho_e)) \rho_e^\tau + P''(\rho_e) \nabla \rho_e\cdot \rho^\tau u^\tau \,\dx\ds. 
\end{align*}
Since the integral of the right-hand side is finite, we find that $t\mapsto II + III$ is absolutely continuous (note that all other terms in $II+III$ arise from the strong solution).
\end{remark}

%%%%%%%%%%%%%%%%%%%%%%%%%%%%%%%%%%%%%%%%%%%%%%%%%%%%%%%%%
%
%
%
%
%
%%%%%%%%%%%%%%%%%%%%%%%%%%%%%%%%%%%%%%%%%%%%%%%%%%%%%%%%%
\subsection{Control of critical nonlinear term}
To complete the modulated energy estimate for $\calF(\tilde U^\tau|\tilde U)$, it remains to control the nonlinear coupling term
\begin{equation}\label{bee}
\intt (\rho_e^\tau - \rho_e)(u - u_e)\cdot \nabla(\Phi^\tau-\Phi) \, \dx
\end{equation}
which encodes the interaction between the electrostatic potentials and the difference between the macroscopic velocity $u$ and the auxiliary electron velocity $u_e$. While this term is relatively easier to handle when $\rho_e^\tau \le \rho_e$, the analysis becomes significantly more delicate when $\rho_e^\tau$ dominates. The challenge stems from the nonlinear structure of the integrand and the lack of uniform control in regions of high electron density.

To address this difficulty, we proceed with two observations. First, in view of Remark \ref{rmk:happy}, we regard the velocity difference
\[
\bar{u} := u - u_e \in L^\infty(0,T;C^1(\T^d))
\]
as a regular and controllable quantity. Second, by recalling the identities
\[
\Phi^\tau = \log \rho_e^\tau, \quad \Phi = \log \rho_e,
\]
we reformulate the integrand of \eqref{bee} as
\bq \label{eq:key:hydro}
(\rho_e^\tau - \rho_e)\bar{u} \cdot \nabla(\Phi^\tau - \Phi) 
 = \rho_e \bar{u} \left(\frac{\rho_e^\tau}{\rho_e} - 1\right) \nabla\log\left(\frac{\rho_e^\tau}{\rho_e}\right) = (\rho_e \bar{u}) \cdot \nabla \left( \frac{\rho_e^\tau}{\rho_e} - 1 - \log\left(\frac{\rho_e^\tau}{\rho_e}\right) \right),
\eq
where we used the identity
\[
(f - 1)\nabla \log f = \nabla (f - 1 - \log f).
\]

This reformulation allows us to integrate by parts in \eqref{eq:key:hydro}, placing the derivative on the smoother term $\rho_e \bar{u}$:
\[
- \nabla \cdot (\rho_e \bar{u}) \left( \frac{\rho_e^\tau}{\rho_e} - 1 - \log\left( \frac{\rho_e^\tau}{\rho_e} \right) \right).
\]
Since $\nabla \cdot (\rho_e \bar{u})$ is uniformly bounded in $L^\infty$, it suffices to control the expression
\[
\frac{\rho_e^\tau}{\rho_e} - 1 - \log\left( \frac{\rho_e^\tau}{\rho_e} \right),
\]
which is known to behave quadratically near $\rho_e^\tau = \rho_e$, and more importantly, remains positive and coercive when $\rho_e^\tau > \rho_e$.

This motivates a decomposition of the domain into regions where either $\rho_e^\tau$ or $\rho_e$ is larger, and a corresponding localization of estimates. However, due to the integration by parts, it is necessary to implement a smooth separation of regions. These ideas are encapsulated in the following key lemma, which is stated in a stationary setting, though it naturally extends to the time-dependent case.

\begin{lemma} \label{lem:key}
Suppose that $\bar{u} \in C^1(\T^d)$, $\Phi^\tau \in H^1(\T^d)$, and $\Phi \in C^1(\T^d)$ satisfy
\[
-\Delta \Phi^\tau = \rho - e^{\Phi^\tau}, \quad -\Delta \Phi = \rho - e^{\Phi}.
\]
Then, the following estimate holds:
\[
\left| \intt (\rho_e^\tau - \rho_e) \bar{u} \cdot \nabla(\Phi^\tau - \Phi) \, \dx \right| 
\le C \intt P(\rho_e^\tau | \rho_e) + |\nabla(\Phi^\tau - \Phi)|^2 \, \dx,
\]
where $C = C(\|\Phi\|_{C^1}, \|\bar{u}\|_{C^1}) > 0$.
\end{lemma}

\begin{proof}
To estimate the left-hand side of the inequality, we follow the strategy outlined in the heuristic discussion above. The main idea is to reformulate the integrand using the identity
\[
(\rho_e^\tau - \rho_e) \nabla (\Phi^\tau - \Phi) = \rho_e \nabla \left( \frac{\rho_e^\tau}{\rho_e} - 1 - \log\left( \frac{\rho_e^\tau}{\rho_e} \right) \right),
\]
and apply integration by parts to shift the derivative onto the smooth coefficient $\rho_e \bar{u}$. However, since this manipulation introduces a nonlinear term involving $\rho_e^\tau / \rho_e$, we must carefully handle the resulting expression when this ratio is large.

To localize the analysis, we introduce a smooth cutoff function $\chi_\delta : \mathbb{R} \to [0,1]$ for $\delta > 0$ small, satisfying
\[
\chi_\delta(a) = 
\begin{cases}
0 &\text{if } a \le 1, \\
1 &\text{if } a \ge 1+\delta
\end{cases}
\quad
\text{and} \quad |\chi_\delta'(a)| \le \frac{C}{\delta}.
\]
This function approximates the indicator of the set where $a>1$, i.e., 
\[
\chi_\delta(a) \to \textbf{1}_{a>1} \quad \text{as} \quad \delta \to 0.
\]
We split the domain using $\chi_\delta(\rho_e^\tau / \rho_e)$ and estimate the term via integration by parts. We define
\[
K_\delta^1 := \intt \chi_\delta\left( \frac{\rho_e^\tau}{\rho_e} \right) \rho_e \bar{u} \cdot \nabla\left( \frac{\rho_e^\tau}{\rho_e} - 1 - \log\left( \frac{\rho_e^\tau}{\rho_e} \right) \right) \, \dx,
\]
and claim that
\bq\label{claim: Idelta1}
|K_\delta^1| \le C \delta^2 + C \intt P(\rho_e^\tau \,|\, \rho_e) + |\nabla(\Phi^\tau - \Phi)|^2 \, \dx
\eq
for some $C>0$ independent of $\delta>0$.

To see this, we decompose $K_\delta^1$ as follows:
\[
\begin{aligned}
K_\delta^1 &= -\intt \chi_\delta'\left( \frac{\rho_e^\tau}{\rho_e} \right) \nabla\left( \frac{\rho_e^\tau}{\rho_e} \right) \cdot \rho_e \bar{u} \left( \frac{\rho_e^\tau}{\rho_e} - 1 - \log\left( \frac{\rho_e^\tau}{\rho_e} \right) \right) \dx \cr
&\quad - \intt \chi_\delta\left( \frac{\rho_e^\tau}{\rho_e} \right) \nabla \cdot (\rho_e \bar{u}) \left( \frac{\rho_e^\tau}{\rho_e} - 1 - \log\left( \frac{\rho_e^\tau}{\rho_e} \right) \right)  \dx \\
&=: K_\delta^{1,1} + K_\delta^{1,2}.
\end{aligned}
\]

We estimate $K_\delta^{1,1}$ first. The support of $\chi_\delta'$ is contained in the region
\[
\Omega_\delta := \left\{ x \in \T^d : 1 \le \frac{\rho_e^\tau(x)}{\rho_e(x)} \le 1 + \delta \right\}.
\]
Using the uniform bounds on $\bar{u}$ and the identity
\[
\nabla \left( \frac{\rho_e^\tau}{\rho_e} \right) = \left( \frac{\rho_e^\tau}{\rho_e} \right) \nabla(\Phi^\tau - \Phi),
\]
we obtain
\begin{align*}
|K_\delta^{1,1}| &\le \left\| \chi_\delta' \right\|_{L^\infty} \left\| \rho_e \bar{u} \right\|_{L^\infty} \int_{\Omega_\delta} \left| \frac{\rho_e^\tau}{\rho_e} - 1 - \log\left( \frac{\rho_e^\tau}{\rho_e} \right) \right| \left| \nabla(\Phi^\tau - \Phi) \right|  \dx \\
&\le C \delta \left\| \bar{u} \right\|_{L^\infty} \int_{\Omega_\delta} |\nabla(\Phi^\tau - \Phi)| \, \dx \\
&\le C \delta^2 + C \intt |\nabla(\Phi^\tau - \Phi)|^2 \, \dx,
\end{align*}
where we used the bound $|\rho_e^\tau/\rho_e - 1 - \log(\rho_e^\tau/\rho_e)| \le C \delta^2$ on $\Omega_\delta$.

For the second term $K_\delta^{1,2}$, note that $\rho_e^\tau \ge \rho_e$ on the support of $\chi_\delta(\rho_e^\tau /\rho_e)$. We recall that for any $x \ge 1$, the function
\[
P(x| 1) = x \log x - (x - 1) \ge c \left( x - 1 - \log x \right),
\]
for some constant $c>0$. Indeed, the function $l(x) = (x-1-\log x)/(x\log x - (x-1))$ is continuous on $(1,\infty)$ and satisfies $\lim_{x\downarrow 1} l(x) = 1$ and $\lim_{x\uparrow\infty}l(x) = 0$.  Hence, using $x = \rho_e^\tau / \rho_e$, we deduce 
\[
|K_\delta^{1,2}| \le \frac{C}{\inf \rho_e}\|\nabla(\rho_e \tilde{u})\|_{L^\infty} \intt P(\rho_e^\tau|\rho_e) \,\dx \le C e^{\|\Phi\|_{L^\infty}} \intt P(\rho_e^\tau|\rho_e) \,\dx,
\]
where we have used $\rho_e = e^{\Phi} \ge e^{-\|\Phi\|_{L^\infty}} > 0$.
This proves the claim \eqref{claim: Idelta1}.

We now turn to the remaining contribution:
\[
\begin{split}
K^2_\delta: = &\intt \lt(1-\chi_\delta\left( \frac{\rho_e^\tau}{\rho_e} \right)\rt) (\rho_e^\tau - \rho_e)\bar{u}\cdot \nabla(\Phi^\tau-\Phi) \,\dx \\ 
\le &\|(\rho_e^\tau-\rho_e)\textbf{1}_{\rho_e^\tau \le (1+\delta)\rho_e}\|_{L^2}\|\bar{u}\|_{L^\infty}\|\nabla(\Phi^\tau -\Phi)\|_{L^2} \\
\le  &\lt(\|(\rho_e^\tau -\rho_e)\textbf{1}_{\rho_e^\tau \le (1+\delta)\rho_e}\|_{L^2}^2 + \|\nabla(\Phi^\tau -\Phi)\|_{L^2}^2 \rt) \|\bar{u}\|_{L^\infty}\\
\le &\lt(2(1+\delta)e^{\|\Phi\|_{\infty}}\intt P(\rho_e^\tau \,|\, \rho_e) + |\nabla(\Phi^\tau - \Phi)|^2 \, \dx\rt)\|\bar{u}\|_{L^\infty},
\end{split}
\]
where we used
\[
P(\rho_e^\tau|\rho_e) \ge \frac{1}{2(1+\delta)\rho_e}(\rho_e^\tau -\rho_e)^2 \ge  \frac{1}{2(1+\delta)}e^{-\|\Phi\|_{\infty}}(\rho_e^\tau-\rho_e)^2\quad\mbox{for }  (1+\delta)\rho_e \ge \rho_e^\tau. 
\]
Hence, for all small $\delta>0$, we have
\[
|K_\delta^{2}| \le 2 (1+\delta) \|\tilde{u}\|_{L^\infty}e^{\|\Phi\|_{L^\infty}} \intt P(\rho_e^\tau |\rho_e) + |\nabla(\Phi^\tau  -\Phi)|^2 \, \dx.
\]
Combining the bounds on $K_\delta^1$ and $K_\delta^2$ and letting $\delta \to 0$ yields the desired estimate.
\end{proof}

 %%%%%%%%%%%%%%%%%%%%%%%%%%%%%%%%%%%%%%%%%%%%%%%%%%%%%%%%%
%
%
%
%
%
%%%%%%%%%%%%%%%%%%%%%%%%%%%%%%%%%%%%%%%%%%%%%%%%%%%%%%%%%
\subsection{Final convergence estimate and proof of Theorem \ref{thm:hdr}}

We now complete the proof of the convergence estimate \eqref{eq: conv mod egy} for the modulated energy functional $\calF(\tilde U^\tau|\tilde U)$, by carefully combining all the ingredients developed in the preceding subsections. The key idea is to control each term in the evolution identity derived in Lemma \ref{lem:fml}, using the uniform bounds and structural properties of the system.

\begin{lemma} \label{lem:eps}
Under the hypotheses {\bf (H1)}--{\bf (H2)} in Theorem \ref{thm:hdr}, we have
\[
\calF(\tilde U^\tau|\tilde U)(t) \le C \sqrt{\tau}, \quad \forall \, t \in [0,T],
\]
where $C>0$ depends on $T$ and $ \|(u,u_e, \Phi)\|_{L^\infty(0,T;C^1(\T^d))}$. 
\end{lemma}

\begin{proof}
As established in Lemma \ref{lem:fml}, the time evolution of the modulated energy satisfies
\bq\label{eq:fml}
\begin{split}
\calF(\tilde U^\tau|\tilde U) &=  \calF(\tilde U^\tau_0|\tilde U_0)  + \Big(\calF(\tilde U^\tau) - \calF(\tilde U^\tau_0) \Big)  +\int_0^t \intt \nabla u : \intr v\otimes \{(u^\tau - v)f^\tau - \kappa \nabla_v f^\tau\} \, \dv \dx \ds\\
&\quad - \int_0^t\intt \nabla u : \rho^\tau(u^\tau -u)\otimes (u^\tau -u) \,\dx \ds  +\int_0^t\intt \{-\Delta (\Phi^\tau - \Phi) u\} \cdot \nabla (\Phi^\tau - \Phi) \,\dx\ds \\
&\quad +\int_0^t\intt (\rho_e^\tau - \rho_e)(u - u_e)\cdot \nabla(\Phi^\tau -\Phi) \,\dx \ds \\
&=: \sum_{i=1}^{6} J_i.
\end{split}
\eq
We now estimate each term $J_i$ individually.

\medskip
\noindent
$\bullet$ Estimate of $J_1$: It is clear that $J_1 \le C \sqrt{\tau}$ by the assumption of well-prepared initial data {\bf (H1)}.

\medskip
\noindent
$\bullet$ Estimate of $J_2$: Recalling the definitions of $\calE$ from \eqref{def : energy.} and $\eta$ from \eqref{def: eta}, we observe that
\begin{align*}
    \calF(\tilde U^\tau) &= \calE[f^\tau] + \intt \left(\frac{1}{2}\rho^\tau |u^\tau|^2 + \kappa \rho^\tau \log \rho^\tau - \int_{\R^d} \Big(\frac{|v|^2}{2} f^\tau + \kappa f^\tau \log f^\tau + \kappa\frac{d}{2}\log(2\pi\kappa) f^\tau \Big)\,\dv\right)\dx \\
    &= \calE[f^\tau] + \intt \left( \eta(U^\tau) - \int_{\R^d} H[f^\tau](v)\,\dv \right)\dx .
\end{align*}
Thus, we can decompose $J_2$ and estimate as
\begin{align*}
    \calF(\tilde U^\tau) - \calF(\tilde U^\tau_0) &= \calE[f^\tau] +  \intt \left(\eta(U^\tau) - \int_{\R^d} H[f^\tau](v)\,\dv\right)\,\dx  \cr
    &\quad - \calE[f_0^\tau] - \intt \left(\eta(U^\tau_0) - \int_{\R^d} H[f^\tau_0](v) \,\dv\right)\,\dx \\
    &\le \Big(\calE[f^\tau] - \calE[f_0^\tau] \Big) + \intt \left(\int_{\R^d}H[f_0^\tau]\,\dv - \eta(U^\tau_0)\right) \,\dx  \\
    &\le C\sqrt{\tau},
\end{align*}
where we used \eqref{eq: min prin}, \eqref{eq : KEI}, and \textbf{(H2)}.

\medskip
\noindent
$\bullet$ Estimate of $J_3$: 
We note that
\begin{equation}\label{eq:diss}
\begin{split}
&\lt|\int_{\R^d} v \otimes \lt\{ (u^\tau - v) f^\tau - \kappa\nabla_v f^\tau\rt\} \dv \rt| \le  \lt(\int_{\R^d} |v|^2 f^\tau\, \dv\rt)^\frac{1}{2} \lt(\int_{\R^d} \frac{|\kappa\nabla_v f^\tau + (v - u^\tau)f^\tau|^2}{ f^\tau}\,\dv\rt)^\frac{1}{2}.
\end{split}
\end{equation}
Hence, \eqref{eq: moments tau} and \eqref{small:diss} yield
\begin{equation*}
\begin{split}
J_3 &\le \|\nabla u\|_{L^\infty} \lt(\int_0^t\iinttr |v|^2 f^\tau \, \dv\dx\ds\rt)^\frac{1}{2} \lt(\int_0^t\iinttr \frac{|\kappa\nabla_v f^\tau + (v - u^\tau)f^\tau|^2}{f^\tau} \, \dv\dx\ds\rt)^\frac{1}{2}\\
&\le C\|\nabla u\|_{L^\infty} \sqrt{T}(\calE[f_0^\tau] + C_{\kappa,d})\sqrt{\tau}.
\end{split}
\end{equation*}
It remains to justify that the initial kinetic energy $\calE[f_0^\tau]$ is uniformly bounded with respect to $\tau$. We first observe
\begin{align*}
    \calE[f_0^\tau] &= \iinttr H_\kappa[f_0^\tau]\,\dx\dv + \intt \frac{1}{2}|\nabla \Phi^\tau_0|^2 + e^{\Phi^\tau_0}(\Phi^\tau_0-1)+1\,\dx \cr
    &\le \intt \eta(U^\tau_0)\,\dx  +C\sqrt{\tau}+ \intt \frac{1}{2}|\nabla \Phi^\tau_0|^2 + e^{\Phi^\tau_0}(\Phi^\tau_0-1)+1\,\dx  \cr
    &= \calF(\tilde U^\tau_0) + C\sqrt{\tau},
\end{align*}
due to \textbf{(H2)}. We then notice that
\[
\calF(\tilde U^\tau_0) = \calF(\tilde U^\tau_0| \tilde U_0) + \calF(\tilde U_0) +\intt {\rm d}F(U_0)(\tilde U^\tau_0 - \tilde U_0).
\]
Since $U_0$ is smooth and fixed, the derivative ${\rm d}F(U_0)$ is bounded in $L^\infty$. Furthermore, using Lemmas \ref{lem:rho} and \ref{lem:rhou}, we estimate
\[
\|\tilde U^\tau_0 - \tilde U_0\|_{L^1} \le C  \calF(\tilde U^\tau_0| \tilde U_0)^{1/2}.
\]
Combining this with assumption {\bf (H1)}, which ensures $\calF(\tilde U_0^\tau|\tilde U_0) \le C\tau$, we conclude that
\[
\calE[f_0^\tau] \leq C 
\]
uniformly in $\tau > 0$, as claimed.

 \medskip
\noindent
$\bullet$ Estimate of $J_4$: It is straightforward that
\[
J_4 \le \|\nabla u\|_{L^{\infty}} \int_0^t \intt \eta(U^\tau|U) \,\dx\ds.
\]
$\bullet$ Estimate of $J_5$: As shown in \cite{kang2018}, integrating by parts yields 
\[
J_5 \le \frac{3}{2}T \|\nabla u\|_{\infty} \|\nabla (\Phi^\tau - \Phi)\|_2^2.
\]
$\bullet$ Estimate of $J_6$: By using Lemma \ref{lem:key} with $\bar{u}:=u-u_e$, we obtain the key bound
\[
J_6\le C \int_0^t \calF(U^\tau|U)\,\ds.
\]

Combining all, we arrive at
\[
\calF(\tilde U^\tau | \tilde U) \le C\sqrt{\tau} + C \int_0^t \calF(\tilde U^\tau | \tilde U) \,\ds, \quad \forall \,t\in[0,T].
\]
Finally, applying Gr\"onwall's inequality completes the proof.
\end{proof}

The modulated energy estimate established in Lemma \ref{lem:eps} provides a direct quantitative control of the difference between the relaxation solution $(\tilde U^\tau,f^\tau)$ and its hydrodynamic limit $(\tilde U,f)$. In particular, it immediately implies the convergence rates stated in \eqref{eq: tauconv l2}--\eqref{eq: conv of dist g1 k=0}, covering both the strong convergence of macroscopic quantities and the Wasserstein-type convergence of the kinetic distribution. The derivation of these statements follows from standard manipulations, where the bounds on the modulated energy are translated into norms measuring the discrepancy between the two solutions. Since these arguments are by now classical and involve no additional difficulties beyond those already addressed, we defer the detailed completion of the proof of Theorem \ref{thm:hdr} to Appendix \ref{app: QE}.

%%%%%%%%%%%%%%%%%%%%%%%%%%%%%%%%%%%%%%%%%%%%%%%%%%%%%%%%%
%
%
%
%
%
%%%%%%%%%%%%%%%%%%%%%%%%%%%%%%%%%%%%%%%%%%%%%%%%%%%%%%%%%

\section{Global existence of solutions to ionic VPFP system} \label{sec : weak solutions fokker planck}
 In this section, we prove Theorem \ref{thm:gws}, establishing the global-in-time existence of weak solutions to \eqref{main eq} for the ionic VPFP system. Our analysis proceeds through a carefully designed approximation scheme incorporating velocity regularization and mollification in both space and velocity variables. Uniform-in-time bounds are obtained from the kinetic entropy-dissipation structure, complemented by moment estimates that control the spatial spreading of particles and ensure tightness. Compactness arguments then allow us to pass to the limit in the nonlinear Poisson coupling, and the limiting solution is shown to satisfy the kinetic entropy inequality, preserving the fundamental entropy-dissipation structure of the original VPFP system.

%%%%%%%%%%%%%%%%%%%%%%%%%%%%%%%%%%%%%%%%%%%%%%%%%%%%%%%%%
%
%
%
%
%
%%%%%%%%%%%%%%%%%%%%%%%%%%%%%%%%%%%%%%%%%%%%%%%%%%%%%%%%%

\subsection{Preliminaries for the Poisson--Boltzmann equation}

Before constructing global solutions, we recall several key results on the Poisson--Boltzmann equation that will be essential in our analysis--particularly in controlling the electric potential and establishing stability estimates. These results, taken from \cite{choikoosong2025}, provide existence, uniqueness, and convergence properties for weak solutions to the Poisson--Boltzmann equation.
\begin{proposition}[Theorem 1.1, \cite{choikoosong2025}]
\label{thm : existence and uniqueness}
Let $d\ge 2$ and $p>1$. For each $\rho \in L^p_+(\T^d)$ with $\rho\not\equiv 0$, there exists a unique weak solution $\Phi[\rho]\in H^1(\T^d)$ to the Poisson--Boltzmann equation
\begin{align}\label{eq:pb}
    -\Delta \Phi[\rho] = \rho - e^{\Phi[\rho]}
\end{align}
which verifies $\|e^{\Phi[\rho]}\|_{L^r(\T^d)}\le \|\rho\|_{L^r(\T^d)}$ for all $r\in [1,p]$. Furthermore, equality holds when $r=1$:
\begin{align*}
    \intt e^{\Phi[\rho]}\,\dx = \intt \rho\,\dx.
\end{align*}
In particular, $\Phi[\rho]\in L^q(\T^d)$ for all $q\in [1,\infty)$.
\end{proposition}

\begin{lemma}[Corollary 2.10, \cite{choikoosong2025}] \label{lem: poi est}
    Let $d\ge 2$ with $p>1$. Suppose $\rho\in L^p_+(\T^d)$ satisfies the bounds
    \begin{align*}
        0 < m \le \|\rho\|_{L^1} \le \|\rho\|_{L^p} \le M < +\infty.
    \end{align*}
    Then the solution $\Phi$ to the Poisson--Boltzmann equation \eqref{eq:pb} satisfies
    \begin{align*}
        \|\Phi[\rho]\|_{H^1(\T^d)} \le C(d,m,M,p)
    \end{align*}
    for some constant $C(d,m,M,p)>0$.
\end{lemma}

\begin{lemma}[Lemma 2.18, \cite{choikoosong2025}]
\label{lem : stability for Lp}
Let $1<p\le 2$. If there are constants $m_1,M_1,M_p>0$ such that $\rho_1, \rho_2\ge 0$ satisfy the bounds
\begin{align*}
 m_1 \le \| \rho_i\|_{L^1} \le M_1, \quad \|\rho_i \|_{L^p} \le M_p,
\end{align*}
then the following stability estimate
\begin{equation*}
%\label{eq:key:sta}
\intt |\nabla\Phi_1 - \nabla\Phi_2|^2 \, \dx + C_1 \intt |\Phi_1 - \Phi_2|^{p'} \,\dx \le C_2 \intt |\rho_1 - \rho_2|^p \,\dx
\end{equation*}
holds for constants $C_1,C_2>0$ dependent only on $m_1,M_1,M_p,p,d$.
\end{lemma}

\begin{lemma}[Lemma 2.20, \cite{choikoosong2025}]
\label{lem : conv of energy}
    Assume that $p>1$, $\rho \in L^p_+(\T^d)$ is not identically zero, and $\rho_i \to \rho$ in $L^p(\T^d)$.
    Then
    \begin{align*}
        \int_{\T^d} (\Phi[\rho] - 1)e^{\Phi[\rho]} \,\dx = \lim_{i\to\infty} \int_{\T^d} (\Phi[\rho_i] - 1)e^{\Phi[\rho_i]} \,\dx.
    \end{align*}
\end{lemma}

Concerning the weak convergence of densities, we have the following weak stability and lower semi-continuity of energy functionals:

\begin{lemma}\label{lem:wsta}
If $\rho_n \weakto \rho$ weakly in $L^p(\T^d)$ for some $p>1$, then we have respectively:
\[
\nabla \Phi[\rho_n] \weakto \nabla \Phi[\rho] \quad\text{in}\quad L^2(\T^d)\qquad\text{and}\qquad e^{\Phi[\rho_n]} \weakto e^{\Phi[\rho]}\quad \text{in}\quad L^p(\T^d).
\]
Furthermore,
\[
 \intt \frac{1}{2}|\nabla \Phi[\rho]|^2 + e^{\Phi}(\Phi-1)+1\,\dx \le \liminf_{n \to \infty} \intt \frac{1}{2}|\nabla \Phi[\rho_n]|^2 + e^{\Phi[\rho_n]}(\Phi[\rho_n]-1)+1\,\dx.
\]
\end{lemma}

\subsection{Regularized system}
We begin by considering a regularized system which formally converges to \eqref{main eq}. Notice that we set an independent regularization parameter for the bulk velocity (note the extra convolution in $u$), due to the difficulties arising from this singular term. For $\ve,\delta>0$ let us set
\begin{equation}\label{eq:regfp}
\begin{cases}
\pa_t f^{\ve,\delta} + v\cdot\nabla_x f^{\ve,\delta} + E_\ve[f^{\ve,\delta}]\cdot \nabla_v f^{\ve,\delta} = (\kappa + \delta \mathbf{1}_{\kappa=0} ) \Delta_v f^{\ve,\delta} - \nabla_v \cdot \Big((\theta^\delta * u_{f^{\ve,\delta}}^{(\ve)} - v ) f^{\ve,\delta}\Big), \\
E_\ve[f^{\ve,\delta}]:= \theta^\ve * \nabla \Phi_\ve[f^{\ve,\delta}], \\
-\Delta \Phi_\ve[f^{\ve,\delta}] = \theta^\ve * \rho_{f^{\ve,\delta}} - e^{\Phi_\ve[f^{\ve,\delta}]},\\[5pt]
u_{f^{\ve,\delta}}^{(\ve)}:= \dfrac{\rho_{f^{\ve,\delta}}u_{f^{\ve,\delta}}}{\rho_{f^{\ve,\delta}} + \ve (1 + |\rho_{f^{\ve,\delta}}u_{f^{\ve,\delta}}|)},\\
f^{\ve,\delta}(0,x,v) = f^{(\delta,0)}(x,v) = f_0(x,v)e^{-\delta|v|^2}.
\end{cases}
\end{equation}
Observe that in the case of local alignment ($\kappa=0$), we added a small viscosity term $\delta \Delta_v f^{\ve,\delta}$ on the collisional kernel. In the above, $\theta$ is a smooth, radial function such that
\begin{align*}
    \theta\in C^\infty_c(B_1;\R), \quad \int_{B_1}\theta(x)\dx = 1, \quad \theta^\ve := \frac{1}{\ve^d} \theta\left(\frac{x}{\ve}\right), \quad \theta^\delta := \frac{1}{\delta^d}\theta\left(\frac{x}{\delta}\right).
\end{align*}

Let us mention some of the main advantages of this choice of regularization. We observe from a straightforward computation that for any $f,g\in L^1(\R^d,(1+|v|)\dv)$:
\begin{align*}
    |u^{(\ve)}_f - u^{(\ve)}_g| \le \frac{2|\rho_f u_f - \rho_g u_g|}{\ve} + \frac{|\rho_f - \rho_g|}{\ve^2} .
\end{align*}
This stability estimate is used in the following fashion. We denote the weighted $L^p$-space norm as
\begin{align*}
    \|f\|_{L^p_q} := \left(\iinttr (1+|v|^q) |f|^p\,\dx\dv\right)^{1/p} .
\end{align*}
For the case where $p=\infty$ we set
\begin{align*}
    \|f\|_{L^\infty_q}:= \esssup_{(x,v)\in \T^d\times\R^d} (1+|v|^q) |f(x,v)|.
\end{align*}
We write by $L^p_q(\T^d\times\R^d)$ naturally the subset of $L^p(\T^d\times\R^d)$ for which $\|f\|_{L^p_q} < + \infty$.
If $q>d+2$, then a simple calculation shows that for all $f,g\in L^2_q$:
\begin{equation} \label{eq: L2q interpolation}
\begin{split}
 |\rho_g - \rho_h |+ |\rho_g u_g - \rho_h u_h|  &\le \int_{\R^d} (1+|v|) |g-h|\,\dv \\
   &= \int_{\R^d} \frac{(1+|v|)(1+|v|^q)^{1/2}}{(1+|v|^q)^{1/2}}|g-h| \, \dv \\
  &  \le \left(\int_{\R^d} \frac{(1+|v|)^2}{1+|v|^q}\,\dv \right)^{1/2}\left(\int_{\R^d} (1+|v|^q)|g-h|^2 \,\dv\right)^{1/2}.
\end{split}
\end{equation}
In particular, the stability estimate for $u^{(\ve)}$ above implies that
\begin{align*}
    |u_f^{(\ve)} - u_g^{(\ve)}| \le C_\ve \left(\intr (1+|v|^q) |f-g|^2 \,\dv \right)^{1/2}.
\end{align*}
It is clear that \eqref{eq: L2q interpolation} also implies
\begin{align}\label{eq: L2 L2q}
    \|\rho_g - \rho_h\|_{L^2(\T^d)} + \|\rho_g u_g - \rho_h u_h\|_{L^2(\T^d)} \le \|g - h\|_{L^2_q}.
\end{align}
We furthermore remark that the regularized bulk velocity is always bounded and dominated by the original:
\begin{align*}
    |u_{f}^{(\ve)}| \le \min\left\{|u_f|, \frac{1}{\ve}\right\}.
\end{align*}

Regarding the electric field, we mention that there are two regularizations: in the Poisson--Boltzmann equation for $\Phi_\ve[f^{\ve,\delta}]$, namely $\eqref{eq:regfp}_3$, the density $\rho_{f^{\ve,\delta}}$ is regularized, and then in $\eqref{eq:regfp}_2$, the electric field $\nabla\Phi_\ve[f^{\ve,\delta}]$ is regularized once more. The advantage of this choice of double regularization is that it conserves the symmetry of the system, and so, in particular the following ``regularized energy functional''
\begin{align}\label{eq: reg energy}
    \calE^{(\ve)}[h] := \iinttr \frac{|v|^2}{2}h + \kappa h\log h\,\dx\dv + \intt \frac{|\nabla \Phi_\ve[h]|^2}{2} + (\Phi_\ve[h] - 1)e^{\Phi_\ve[h]} + 1\,\dx
\end{align}
can be associated to solutions of \eqref{eq:regfp}. We remark that this type of regularization has already appeared in the works of \cite{griffinpickeringiacobelli2021a,horst1990}.

In terms of stability of the electric field, we have, thanks to the mollification, that for any $f,g\in L^2_q(\T^d\times\R^d)$
\begin{align} \label{eq: reg elec stab}
    \|E_\ve[f] - E_\ve[g]\|_{C^1(\T^d)} \le C_\ve \|\rho_f - \rho_g\|_{L^1(\T^d)} \le \|f-g\|_{L^2_q},
\end{align}
where the first inequality is due to Lemma \ref{lem : stability for Lp} and the second is due to \eqref{eq: L2 L2q}.

We now state the well-posedness result for the regularized system.

%The idea of truncating the bulk velocity follows the works of \cite{karpermellettrivisa2013}.
\begin{proposition}\label{prop:fp:reg}
For all $\ve>0$ and $\delta>0$, there exists $f^{\ve,\delta} \in L^\infty_{\rm loc}((0,\infty);L^2_q(\T^d\times\R^d)) \cap L^\infty([0,\infty);L^1(\T^d\times\R^d))$, satisfying $\nabla_v f^{\ve,\delta} \in L^2_{\rm loc}((0,\infty);L^2_q(\T^d\times\R^d))$, which is the unique weak solution to the regularized system \eqref{eq:regfp} in the sense that for each $t\ge 0$ and $\phi\in C_c^\infty([0,t]\times\T^d\times\R^d)$
\begin{equation}\label{eq: weak form ve,delta}
\begin{split}
    &\iinttr f^{\ve,\delta}(t) \phi(t,x,v) \,\dx\dv - \iinttr f^{(\delta,0)} \phi(0,x,v)\,\dx\dv \\
    &\quad = \int_0^t \iinttr f^{\ve,\delta} \Big( \p_s \phi + v\cdot \nabla_x \phi + E_\ve[f^{\ve,\delta}]\cdot \nabla_v \phi) \,\dx\dv\ds \\
    &\qquad + \int_0^t \iinttr -\nabla_v\phi \cdot \Big((\kappa + \delta\mathbf{1}_{\kappa=0}) \nabla_v f^{\ve,\delta} - \Big( \big( \theta^\delta * u_{f^{\ve,\delta}}^{(\ve)} - v \big) f^{\ve,\delta}\Big) \,\dx\dv\ds.
\end{split}
\end{equation}
\end{proposition}

\begin{proof}
\textbf{Step I: The linear system.}
The starting point is the following linear system
\bq\label{eq:fp:lin}
\begin{cases}
\p_t f + v\cdot \nabla_x f + E_\ve[g] \cdot \nabla_v f = (\kappa+ \delta 1_{\kappa=0}) \Delta_v f - \nabla_v \cdot \Big((\theta^\delta * u_{g}^{(\ve)} - v )f\Big) , \\
f(0,\cdot,\cdot)=f^{(\delta,0)}(\cdot,\cdot),
\end{cases}
\eq
where $g$ is assumed given and extracted from
\[
g \in \calC_{\rm FP}^{q,T} :=\lt\{g \in L^\infty(0,T;L^2_q(\T^d\times \R^d) \cap C([0,T];L^1(\T^d\times\R^d)) : \|g(t)\|_{L^1} = \|g(0)\|_{L^1},\, g|_{t=0} =f^{(\delta,0)}\rt\}.
\]
In the above we suppose $q\in (d+2,\infty)$ and $T>0$ are fixed. Owing to the regularizations, Proposition \ref{thm : existence and uniqueness}, and $|u^{(\ve)}|\le 1/\ve$, we immediately obtain the following bounds:
\begin{equation} \label{eq: trivial bounds}
\begin{split}
    &\sup_{t\in [0,T]} \|E_\ve[g]\|_{C^1(\T^d)}\le C_{\ve,\delta,f_0}, \qquad \sup_{t\in [0,T]} \|\theta^\delta*u^{(\ve)}_{g}\|_{C^1(\T^d)}\le C_\ve.
\end{split}
\end{equation}
As the diffusion in $v$ is nondegenerate, we deduce the existence of a unique weak solution $f\in C([0,T];L^1(\T^d\times\R^d))$ via classical existence theory, see for instance \cite{degond1985,degond1986, karpermellettrivisa2013}. We remark that the solution to \eqref{eq:fp:lin} exhibits an explicit representation through the Feynman--Kac formula: utilizing it we can obtain, cf. \cite[Proposition 2.2]{choihwangyoo}
\begin{align*}
    \sup_{t\in [0,T]} \|f(t)\|_{L^\infty_q} \le C_{\ve,\delta,f_0,T},
\end{align*}
which justifies the integration by parts that are performed in the sequel of this proof. Another particular deduction is that $f\in C_{\rm FP}^{q,T}$.

\textbf{Step II: Explicit $L^2_q$-estimates.}  We now provide explicit estimates in $L^2_q$ to see what bounds we can obtain for $f$ and $\nabla_v f$. Notice that the solution $f$ to \eqref{eq:fp:lin} satisfies
\begin{equation} \label{eq: L2q est}
\begin{split}
    &\ddt \iinttr(1+|v|^q) |f|^2 \,\dx\dv + 2(\kappa + \delta\mathbf{1}_{\kappa=0}) \iinttr (1+|v|^q) |\nabla_v f|^2\,\dx\dv \\
    &\quad = \iinttr -(1+|v|^q) E_\ve[g]\cdot \nabla_v (f^2) + (1+|v|^q) 2f \nabla_v \cdot \Big( (\theta^\delta * u_g^{(\ve)} - v) f \Big) \,\dx\dv \\
    &\quad = \iinttr qv|v|^{q-2} \cdot E_\ve[g] |f|^2 \,\dx\dv  + \iinttr -2qv|v|^{q-2} |f|^2 \cdot \theta^\delta * u_g^{(\ve)} + 2q|v|^q |f|^2 \,\dx\dv\\
    &\qquad + \iinttr (1+|v|^q) (\theta^\delta * u_g^{(\ve)} - v) \cdot \nabla_v f^2 \,\dx\dv \\
    &\quad =: \sum_{i=1}^3 I_i.
\end{split}
\end{equation}
$\bullet$ Estimate of $I_1$: we have using \eqref{eq: trivial bounds} and Young's inequality
\begin{align*}
    |I_1| \le C_{\ve,q} \iinttr (1+|v|^q) |f|^2\,\dx\dv.
\end{align*}
$\bullet$ Estimate of $I_2$: In similar nature, using \eqref{eq: trivial bounds}:
\begin{align*}
    |I_2| \le C_{\ve,\delta} \iinttr (1+|v|^q) |f|^2\,\dx\dv.
\end{align*}
$\bullet$ Estimate of $I_3$: We integrate by parts to have
\begin{align*}
    |I_3| &= \left| \iinttr -qv|v|^{q-2} \cdot \theta^\delta * u_g^{(\ve)} |f|^2 + (q+d) (1+|v|^q) |f|^2\,\dx\dv \right| \le C_{q,\ve,\delta,d} \iinttr (1+|v|^q) |f|^2\,\dx\dv.
\end{align*}
Collecting the estimates, and recalling that $f^{(\delta,0)}$ has arbitrary algebraic decay in $v$, we conclude from Gr\"onwall's lemma that
\begin{align*}
    \sup_{t\in [0,T]} \iinttr (1+|v|^q) |f|^2\,\dx\dv \le C_{\ve,\delta,q,d,T} < +\infty.
\end{align*}
In view of this and the bounds for the $I_i$, we also deduce from \eqref{eq: L2q est} that
\begin{align}\label{eq:apri:vavi}
    \int_0^t\iinttr (1+|v|^q) |\nabla_v f|^2\,\dx\dv\ds \le C_{\kappa,\ve,\delta,q,d} \, t .
\end{align}
We emphasize that these estimates are independent of the choice of $g$.

\textbf{Step III: The iteration sequence and Cauchy estimates.} Let us remind the reader of some simple stability estimates: for any $h_1,h_2 \in \calC_{\rm FP}^{q,T}$, we recall from \eqref{eq: reg elec stab}
\begin{equation} \label{eq: elec Linfty, L2q}
 \|E_\ve[h_1]-E_\ve[h_2]\|_{L^\infty(\T^d)} \le C_\ve \|\rho_{h_1} - \rho_{h_2}\|_{L^2(\T^d)} \le \|h_1 - h_2\|_{L^2_q},
\end{equation}
and also for the regularized bulk velocity (see \eqref{eq: L2 L2q}):
\begin{equation} \label{eq: delta ve u stab}
\begin{split}
 \|\theta^\delta *u^{(\ve)}_{h_1} - \theta^\delta *u^{(\ve)}_{h_2}\|_{L^\infty} &\le C_{\ve,\delta} \lt(\|\rho_{h_1} u_{h_1} - \rho_{h_2} u_{h_2}\|_{L^2(\T^d)} + \|\rho_{h_1} -u_{h_2} \|_{L^2(\T^d)}\rt) \le \|h_1 - h_2\|_{L^2_q}.
\end{split}
\end{equation}
With the aid of these stability estimates, we now demonstrate that the following iterative sequence
\begin{align*}
    \begin{cases}
    \p_t f_{k+1} + v\cdot \nabla_x f_k + E_\ve[f_k]\cdot \nabla_v f_{k+1} = (\kappa + \delta \mathbf{1}_{\kappa = 0}) \Delta_v f_{k+1} - \nabla_v\cdot \Big( (\theta^\delta * u_{f_k}^{(\ve)} - v) f_{k+1} \Big) \quad \forall k\ge 1,\\
    f_{k}|_{t=0} = f^{(\delta,0)} \quad \forall k\ge 1,
    \end{cases}
\end{align*}
with initial step $f_1(t,x,v) := f^{(\delta,0)}(x,v)$, converges in $L^\infty(0,T;L^2_q(\T^d\times\R^d))$ to a solution of \eqref{eq:regfp}. A direct computation reveals the following Cauchy estimates
\begin{align*}
    &\ddt \iinttr (1+|v|^q) (f_{k+1} - f_k)^2 \,\dx\dv + (\kappa+\delta\mathbf{1}_{\kappa=0}) \iinttr (1+|v|^q) |\nabla_v (f_{k+1} - f_k)|^2\,\dx\dv \\
    &\quad = - \iinttr (1+|v|^q)(f_{k+1}-f_k) \nabla_v\cdot (E_\ve[f_k] f_{k+1} - E_\ve[f_{k-1}] f_k) \,\dx\dv\\
    &\qquad - \iinttr (1+|v|^q) (f_{k+1}-f_k) \nabla_v\cdot \Big( \theta^\delta * u_{f_k}^{(\ve)} f_{k+1} - \theta^\delta * u_{f_{k-1}}^{(\ve)} f_k \Big) \dx\dv\\
    &\qquad + \iinttr (1+|v|^q) (f_{k+1}-f_k) \nabla_v\cdot (v(f_{k+1} - f_k)) \,\dx\dv\\
    &\quad =: \sum_{i=1}^3 J_i.
\end{align*}
$\bullet$ Estimate of $J_1$: We telescope the difference as
\begin{align*}
    J_1 &= -\iinttr (1+|v|^q) (f_{k+1} - f_k) E_\ve[f_k] \cdot \nabla_v(f_{k+1} - f_k)\,\dx\dv \\
    &\quad - \iinttr (1+|v|^q) (f_{k+1}-f_k) (E_\ve[f_k]-E_\ve[f_{k-1}])\cdot \nabla_v f_k\,\dx\dv \\
    &=: J_{11} + J_{12},
\end{align*}
and estimate
\begin{align*}
    |J_{11}| = \left| \iinttr qv|v|^{q-2} \cdot E_\ve[f_k] (f_{k+1}-f_k)^2 \,\dx\dv  \right| \le C_{\ve,q} \iinttr (1+|v|^q) (f_{k+1} - f_k)^2\,\dx\dv.
\end{align*}
We majorize $J_{12}$ using \eqref{eq: elec Linfty, L2q} and the Cauchy--Schwarz inequality as
\begin{align*}
    |J_{12}| &\le \|E_\ve[f_k] - E_\ve[f_{k-1}]\|_{L^\infty} \iinttr (1+|v|^q) |f_{k+1}-f_k| |\nabla_v f_k|\,\dx\dv \\
    &\le C_\ve \|f_k - f_{k-1}\|_{L^2_q} \left(\iinttr (1+|v|^q)(f_{k+1}-f_k)^2 \dx\dv\right)^{1/2} \left(\iinttr(1+|v|^q)|\nabla_v f_k|^2 \dx\dv\right)^{1/2} \\
    &\le C_\ve \|f_k - f_{k-1} \|_{L^2_q} \|f_{k+1} - f_k\|_{L^2_q} \|\nabla_v f_k\|_{L^2_q}.
\end{align*}

\noindent
$\bullet$ Estimate of $J_2$: We also telescope the difference as
\begin{align*}
    J_2 &= -\iinttr (1+|v|^q) (f_{k+1} - f_k) \nabla_v\cdot \Big( \theta^\delta * (u_{f_k}^{(\ve)} - u_{f_{k-1}}^{(\ve)}) f_{k+1} \Big)\dx\dv\\
    &\quad - \iinttr (1+|v|^q) (f_{k+1} - f_k) \nabla_v \cdot \Big( \theta^\delta * u_{f_{k-1}}^{(\ve)} (f_{k+1} - f_k) \Big) \dx\dv\\
    &=: J_{21} + J_{22}.
\end{align*}
Since $\theta^\delta$ and $u^{(\ve)}$ are independent of $v$, $J_{21}$ can be estimated as
\begin{align*}
    |J_{21}| &= \left|\iinttr (1+|v|^q) (f_{k+1} - f_k) (\theta^\delta * (u_{f_k}^{(\ve)} - u_{f_{k-1}}^{(\ve)} ) ) \cdot \nabla_v f_{k+1} \,\dx\dv \right| \\
    &\le \|\theta^\delta * (u_{f_k}^{(\ve)} - u_{f_{k-1}}^{(\ve)})\|_{L^\infty} \iinttr (1+|v|^q) |f_{k+1} - f_k| |\nabla_v f_{k+1}|\,\dx\dv \\
    &\le C_{\kappa,\ve,\delta,d,t,q} \|f_k - f_{k-1}\|_{L^2_q} \|f_{k+1} - f_k\|_{L^2_q} \|\nabla_v f_{k+1}\|_{L^2_q},
\end{align*}
the last line follows from \eqref{eq: delta ve u stab} and again the Cauchy--Schwarz inequality. Next, by a simple integration by parts (and using the regularization), we find
\begin{align*}
    |J_{22}| &\le C_{\ve,\delta} \left|\iinttr q|v|^{q-1} (f_{k+1} - f_k)^2 \right|  \le C_{\ve,\delta,q} \iinttr (1+|v|^q) (f_{k+1} - f_k)^2\,\dx\dv.
\end{align*}
$\bullet$ Estimate of $J_3$: Integrating by parts, we immediately obtain
\begin{align*}
    |J_3| = \iinttr q|v|^q (f_{k+1} - f_k)^2 \,\dx\dv .
\end{align*}

Collecting the estimates and then applying Young's inequality, we find that
\begin{align}
    \ddt \|f_{k+1}-f_k\|_{L^2_q}^2 \le C_{\kappa,\ve,\delta,d,q}(1 + \|\nabla_v f_{k+1}\|_{L^2_q} + \|\nabla_v f_k\|_{L^2_q}) \Big(\|f_k - f_{k-1}\|_{L^2_q}^2 + \|f_{k+1} - f_k\|_{L^2_q}^2 \Big). \label{eq:regfp:key}
\end{align}
Introducing the notations
\begin{align*}
    a_k(t) := \|f_{k+1}(t) - f_{k}(t)\|_{L^2_q}^2, \quad b_k(t) := C_{\kappa,\ve,\delta,d,q}(1 + \|\nabla_v f_{k+1}(t)\|_{L^2_q} + \|\nabla_v f_k(t)\|_{L^2_q}),
\end{align*}
thanks to \eqref{eq:apri:vavi} and \eqref{eq:regfp:key}, we observe that Lemma \ref{lem: cauchy} below applies, which demonstrates that $\sum_{k=1}^\infty a_k(t)$ is uniformly convergent in $[0,T]$. The Cauchy criterion then shows
\begin{align*}
    \sup_{t \in [0,T]} \|f_m - f_\ell\|_{L^2_q}^2 \le \sup_{t\in [0,T]} \sum_{i=\ell+1}^m a_i(t) \xrightarrow[\ell,m\to\infty]{} 0,
\end{align*}
thus $\{f_k\}$ is Cauchy in $L^\infty(0,T;L^2_q(\T^d\times\R^d))$. It is then straightforward to verify that the limit $f^{\ve,\delta}$ is a solution to \eqref{eq:regfp} in the sense of distributions. Uniqueness of $f^{\ve,\delta}$ in $L^\infty(0,T;L^2_q(\T^d\times\R^d))$ follows as a straightforward consequence of the stability estimates \eqref{eq: elec Linfty, L2q}--\eqref{eq: delta ve u stab}.

Finally, we note that the arbitrariness of $T>0$ and the mass conserving nature of \eqref{eq:fp:lin} demonstrate that, by a diagonal argument, our solution extends globally to $f^{\ve,\delta}\in L^\infty([0,\infty);L^1(\T^d\times\R^d)) \cap L^\infty_{\rm loc}([0,\infty);L^2_q(\T^d\times\R^d))$. This completes the proof.

\end{proof}

In the above, we utilized the following lemma, which provides sufficient conditions for $a_k(t)$ to be Cauchy in $L^\infty(0,T)$. Let us record the statement and a quick proof.
\begin{lemma} \label{lem: cauchy}
    Assume that $\{a_k(t)\}_{k=1}^\infty \subset C^1([0,T])$ is a sequence of nonnegative functions with
    \begin{align*}
        &a_{k+1}'(t) \le b_{k+1}(t) (a_{k+1}(t) + a_k(t)),\\
        &b_{k+1}\ge 0, \quad \sup_k \int_0^T b_{k+1}^2(s)\,\ds \le M_T < +\infty.
    \end{align*}
    Suppose $a_k(0)=0$ for all $k\in \bbN$. Then we have
    \begin{align*}
        a_{k+1}(t) \le A \left(e^{M_t^{1/2} T^{1/2}} M_T^{1/2} \right)^k \left(\frac{t^k}{k!}\right)^{1/2},
    \end{align*}
    where
    \begin{align*}
        A := \sup_{t\in [0,T]} a_1(t).
    \end{align*}
    In particular the series $\sum_{k=1}^\infty a_k(t)$
    is uniformly convergent in $[0,T]$.
\end{lemma}
\begin{proof}
    The assertion can be proven directly via induction, but we first provide computations for $a_2$ and $a_3$ as they are more demonstrative of how the estimates are deduced.
    
    We observe from a direct computation and the given inequality for the $a_{k+1}'$ that
%    \begin{equation} \label{eq: seq key}
\[%    \begin{split}
        \left(e^{-B_{k+1}(t)} a_{k+1}(t) \right)' = e^{-B_{k+1}} a_{k+1}' - b_{k+1} e^{-B_{k+1}}a_{k+1}  \le b_{k+1} a_k e^{-B_{k+1}},
\]%    \end{split}
%    \end{equation}
    where 
    \begin{align*}
        B_{k+1}(t) := \int_0^t b_{k+1}(s) \,\ds \le M_T^{1/2} t^{1/2} \le M_T^{1/2} T^{1/2} \quad \forall \,t\in [0,T].
    \end{align*}
    Since $a_k(0) = 0$ for all $k \in \N$, we obtain
    \[
    a_{k+1}(t) \leq \int_0^t b_{k+1}(s) a_k(s) e^{B_{k+1}(t)-B_{k+1}(s)}\,\ds.
    \]
 In the case $k=1$, we get
    \begin{align*}
        a_2(t) \le \int_0^t A e^{B_2(t) - B_2(s)} b_2(s)\,\ds \le A e^{M_T^{1/2} T^{1/2}} B_2(t) \le Ae^{M_T^{1/2} T^{1/2}} M_T^{1/2} t^{1/2}.
    \end{align*}
    Then, in the case $k=2$ we proceed as
    \begin{align*}
        a_3(t) &\le \int_0^t b_3(s) a_2(s) e^{B_3(t)-B_3(s)} \,\ds \\
        &\le A \left(e^{M_T^{1/2} T^{1/2}}\right)^2 M_T^{1/2} \int_0^t b_3(s) s^{1/2} \,\ds \\
        &\le A e^{M_T T} M_T^{1/2} \left(\int_0^t b_3^2(s)\,\ds\right)^{1/2} \left(\int_0^t s\,\ds\right)^{1/2} \\
        &\le A e^{M_T T} M_T \left(\frac{s^2}{2}\right)^{1/2}.
    \end{align*}
    Inductively we can obtain
    \begin{align*}
        a_{k+1}(t) &\le \int_0^t b_{k+1}(s) a_k(s) e^{B_{k+1}(t)-B_{k+1}(s)}\,\ds \\
        &\le A \left(e^{M_T^{1/2} T^{1/2}} M_T^{1/2}\right)^k \cdot e^{M_T^{1/2} T^{1/2}} \int_0^t b_{k+1}(s) \left(\frac{s^k}{k!}\right)^{1/2} \,\ds \\
        &\le A \left(e^{M_T^{1/2} T^{1/2}} M_T^{1/2}\right)^k \cdot e^{M_T^{1/2} T^{1/2}} \left(\int_0^t b_{k+1}^2(s)\,\ds\right)^{1/2} \left(\int_0^t \frac{s^k}{k!}\,\ds\right)^{1/2} \\
        &\le A \left(e^{M_T^{1/2} T^{1/2}} M_T^{1/2} \right)^{k+1} \left(\frac{t^{k+1}}{(k+1)!}\right)^{1/2}.
    \end{align*}
    In virtue of Stirling's formula and the Weierstrass $M$-test we deduce that $\sum_{k=1}^\infty a_{k+1}(t)$ is uniformly convergent for $t\in [0,T]$. 
\end{proof}

%%%%%%%%%%%%%%%%%%%%%%%%%%%%%%%%%%%%%%%%%%%%%%%%%%%%%%%%%
%
%
%
%
%
%%%%%%%%%%%%%%%%%%%%%%%%%%%%%%%%%%%%%%%%%%%%%%%%%%%%%%%%%

%%%%%%%%%%%%%%%%%%%%%%%%%%%%%%%%%%%%%%%%%%%%%%%%%%%%%%%%%
%
%
%
%
%
%%%%%%%%%%%%%%%%%%%%%%%%%%%%%%%%%%%%%%%%%%%%%%%%%%%%%%%%%

%%%%%%%%%%%%%%%%%%%%%%%%%%%%%%%%%%%%%%%%%%%%%%%%%%%%%%%%%
%
%
%
%
%
%%%%%%%%%%%%%%%%%%%%%%%%%%%%%%%%%%%%%%%%%%%%%%%%%%%%%%%%%

In order to obtain compactness with respect to each regularization parameter, we need to establish uniform estimates for the solutions $f^{\ve,\delta}$. We begin with the goal of sending $\delta\to 0$.

\begin{lemma}[Uniform estimates for \eqref{eq:regfp}] \label{lem:fp:uni}
Let $f_0 \in L^1_2 \cap L^\infty (\T^d \times \R^d)$. For the solution $f^{\ve,\delta}$ to \eqref{eq:regfp} with initial datum $f^{(\delta,0)}$, the following uniform estimates hold for all $\delta\in (0,1]$:
\begin{enumerate}[(i)]
\item $L^p$-estimates: for any $p \in (1,\infty]$:
\bq\label{eq:fp:lp}
\|f^{\ve,\delta}(t)\|_{L^p} + (\kappa + \delta \mathbf{1}_{\kappa = 0}) \left(\int_0^t \|\nabla_v f^{p/2}(s)\|_{L^2}^2 \ds\right)^{1/p} \le \|f_0\|_{L^p}  e^{dt/p'} \quad \forall \,t>0,
\eq
where the middle term is interpreted as zero when $p=\infty$.
\item Mass conservation:
\bq\label{eq:fp:mc1}
\|f^{\ve,\delta}(t)\|_{L^1}=\|f^{(\delta,0)}\|_{L^1} \ge c_0>0 .
\eq
\item Bounds on electric fields:
\bq\label{eq:fp:el2}
\sup_{0\le t \le T} \|E_\ve[f^{\ve,\delta}] \|_{C^1(\T^d)} \le C_{\ve,f_0},
\eq
\bq\label{eq:fp:eff}
\sup_{0\le t \le T}\|E_\ve[f^{\ve,\delta}]f^{\ve,\delta}\|_{L^2(\T^d \times B_R)} \le C_{\ve,f_0,R,T}.
\eq
\item Moment bounds:
\bq\label{eq:fp:mb}
\sup_{0\le t\le T}\iint_{\T^d\times\R^d} |v|^2 f^{\ve,\delta}\,\dx\dv \le C_{\ve,f_0,T}.
\eq
\item Entropy bounds:
\begin{align} \label{eq:fp:entropy}
    \sup_{0\le t \le T} \kappa \iinttr f^{\ve,\delta}|\log f^{\ve,\delta}|\,\dx\dv \le C_{\ve,f_0,T}.
\end{align}
\item Density upper bounds: for each $T>0$,
\bq\label{eq:fp:dub}
\sup_{0\le t \le T}\|\rho_{f^{\ve,\delta}}\|_{L^{\frac{d+2}{d}}(\T^d)} \le C_{\ve,f_0,T}.
\eq
\item Bounds on alignment terms:
\begin{align}
&\sup_{0\le t \le T} \|(\theta^{\delta}*u^{(\ve)}_{f^{\ve,\delta}}) f^{\ve,\delta}\|_{L^2} \le C_{\ve,f_0,T}, \label{eq:fp:ufd}\\
& \sup_{0 \le t \le T} \| vf^{\ve,\delta}\|_{L^2} \le C_{\ve,f_0,T}. \label{eq:fp:vf}
\end{align}

\item Equicontinuity: for any $0\le t_1 < t_2 \le T$,
\bq\label{eq:fp:eqcd}
\|f^{\ve,\delta}(t_1) - f^{\ve,\delta}(t_2) \|_{W^{-2,2}(\T^d\times B_R)} \le C_{\ve,f_0,R,T}|t_1 - t_2|.
\eq
\end{enumerate}
\end{lemma}

\begin{proof}
The proof is rather standard and well-known in the literature, so we only give a sketch of the ideas and point to appropriate references.\\
\textbf{Proof of \eqref{eq:fp:lp}--\eqref{eq:fp:mc1}:} The $L^p$ estimate is obtained formally by testing $(f^{\ve,\delta})^{p-1}$ against \eqref{eq:regfp}: see for instance \cite[Lemma 2.1]{karpermellettrivisa2013}. The mass conservation \eqref{eq:fp:mc1} is also clear. The lower bound in \eqref{eq:fp:mc1} simply follows since
\[
\|f^{(\delta,0)}\|_{L^1} \ge \|f^{(\delta^*,0)}\|_{L^1} >0
\]
whenever $0<\delta \le \delta^*$, and so it is enough to restrict ourselves to a small range of $\delta>0$.

\noindent
\textbf{Proof of \eqref{eq:fp:el2}--\eqref{eq:fp:eff}:}
Eq. \eqref{eq:fp:el2} is easily obtained from the mollification. Indeed, since
\begin{align*}
    \|\theta^\ve * \rho_{f^{\ve,\delta}}\|_{L^p(\T^d)} \le C_\ve \|f^{\ve,\delta}\|_{L^1(\T^d\times\R^d)} = C_\ve \|f^{(\delta,0)}\|_{L^1(\T^d\times\R^d)} \le C_\ve \|f_0\|_{L^1(\T^d\times\R^d)} < + \infty \quad \forall p\in [1,\infty],
\end{align*}
Lemma \ref{lem: poi est} provides $H^1(\T^d)$-bounds for $\Phi_\ve[f^{\ve,\delta}]$. Since $E_\ve[f^{\ve,\delta}] = -\theta^\ve * \nabla\Phi_\ve[f^{\ve,\delta}]$, the additional mollification yields $C^1(\T^d)$-bounds which are $\ve$-dependent, but independent of $\delta$. Next, for each $R>0$ we prove \eqref{eq:fp:eff} using \eqref{eq:fp:lp}:
\[
\iint_{\T^d\times B_R}  |E_\ve[f^{\ve,\delta}] |^2 |f^{\ve,\delta}|^2\,\dx\dv \le CR^{d} \|f^{\ve,\delta}\|^2_{L^\infty} \intt |E_\ve[f^{\ve,\delta}]|^2\,\dx \le C_{f_0,\ve,R,T}.
\]

\noindent
\textbf{Proof of \eqref{eq:fp:mb}:}
The moment bound \eqref{eq:fp:mb} is obtained in a similar manner as in the proof of Proposition \ref{prop:fp:reg}, \textbf{Step II}. Owing to $f^{\ve,\delta} \in L^\infty_{\rm loc}((0,\infty);L^2_q(\T^d\times\R^d))$ and the regularization of the coefficients, we notice that $|v|^2$ is admissible through approximation as a test function in the weak formulation \eqref{eq: weak form ve,delta}. The identity obtained implies that
$t\mapsto \iinttr|v|^2 f^{\ve,\delta}\dx\dv$ is absolutely continuous and
\begin{equation} \label{eq: formal v^2}
\begin{split}
    \ddt \iinttr(1+|v|^2) f^{\ve,\delta}\dx\dv &= \iinttr E_\ve[f^{\ve,\delta}] \cdot 2v f^{\ve,\delta} + 2d (\kappa + \delta\mathbf{1}_{\kappa=0}) f^{\ve,\delta} \,\dx\dv\\
    &\quad + \iinttr -2|v|^2 f^{\ve,\delta} + \theta^\delta * u_{f^{\ve,\delta}}^{(\ve)} \cdot 2v f^{\ve,\delta}\,\dx\dv \\
    &\le C_{\ve,d,\kappa} \iinttr (1+|v|^2) f^{\ve,\delta}\,\dx\dv + C_\ve \iinttr |v|^2 f^{\ve,\delta}\,\dx\dv,
    \end{split}
\end{equation}
where we utilized that the pointwise bound $|u_{f^{\ve,\delta}}^{(\ve)}|\le 1/\ve$ implies $\|\theta^\delta * u_{f^{\ve,\delta}}^{(\ve)}\|_{L^\infty} \le 1/\ve$. Applying Gr\"onwall's lemma, we deduce \eqref{eq:fp:mb}.

\noindent
\textbf{Proof of \eqref{eq:fp:entropy}:} We give a rather formal proof here. The rigorous treatment is possible by using the arguments in \cite[Lemma 3.3]{bostan2024}. In that work, it is shown (for a similar linear Vlasov--Poisson--Fokker--Planck equation), using commutator estimates in the spirit of Diperna--Lions \cite{Dipernalions1989ODE}, that the solution $f^{\ve,\delta}$ is renormalized under the assumption of $L^\infty((0,T)\times\T^d)$ electric field, and that $z\mapsto z\log z$ is admissible as a renormalizing function. We remark that no subtleties arise in our situation, thanks to the regularizations in $E_\ve[f^{\ve,\delta}]$ and $\theta^\delta * u_{f^{\ve,\delta}}^{(\ve)}$. It is deduced that
\begin{equation} \label{eq: entropy identity}
\begin{split}
    \ddt\iinttr f^{\ve,\delta}\log f^{\ve,\delta}\,\dx\dv &= \iinttr -\frac{\kappa}{f^{\ve,\delta}} |\nabla_v f^{\ve,\delta}|^2 + (v - \theta^\delta * u_{f^{\ve,\delta}}^{(\ve)}) \cdot \nabla_v f^{\ve,\delta} \,\dx\dv  \\
    &\le \iinttr d f^{\ve,\delta}\,\dx\dv \\
    &\le C_d.
\end{split}
\end{equation}
In the above, the first inequality follows simply by discarding the negative term and integrating by parts. The second inequality is by \eqref{eq:fp:lp}. We find that $\sup_{t\in [0,T]}\iinttr f^{\ve,\delta}\log f^{\ve,\delta}\,\dx\dv \le C_{T,d}$. Then, from another classical interpolation between this and \eqref{eq:fp:mb} (see \cite[Section I]{dipernalions1988}), we obtain \eqref{eq:fp:entropy}.

\noindent
\textbf{Proof of \eqref{eq:fp:dub}:} The density upper bound \eqref{eq:fp:dub} is due to a classical interpolation of \eqref{eq:fp:lp} and \eqref{eq:fp:mb}. For unfamiliar readers we refer to \cite[Lemma 5.2]{griffinpickeringiacobelli2021a}.

\noindent
\textbf{Proof of \eqref{eq:fp:ufd}--\eqref{eq:fp:vf}:}
To confirm \eqref{eq:fp:ufd}, we again make use of $\lt\|\theta^{\delta}*u^{(\ve)}_{f^{\ve,\delta}}\rt\|_{L^\infty([0,\infty)\times \T^d)} \le \frac{1}{\ve}.$ This gives
\[
\iint_{\T^d\times \R^d} \lt|(\theta^{\delta}*u^{(\ve)}_{f^{\ve,\delta}}) f^{\ve,\delta} \rt|^2\,\dx\dv \le \frac{1}{\ve^2}\|f^{\ve,\delta} \|^2_{L^2} \le C_{\ve,f_0,T}
\]
from \eqref{eq:fp:lp}. It is also easy to observe \eqref{eq:fp:vf}, since we have
\[
\iint_{\T^d\times\R^d} |v|^2 (f^{\ve,\delta})^2\,\dx\dv \le \|f^{\ve,\delta}\|_{L^\infty} \iint_{\T^d\times\R^d} |v|^2 f^{\ve,\delta}\,\dx\dv \le C_{\ve,f_0,T}
\]
via \eqref{eq:fp:lp} and \eqref{eq:fp:mb}.

\noindent
\textbf{Proof of \eqref{eq:fp:eqcd}:}
Lastly, we confirm the $W^{-2,2}_{\rm loc}$ equicontinuity \eqref{eq:fp:eqcd}. We proceed similarly as in \cite[Lemma 6.3]{griffinpickeringiacobelli2021a}. For each $R>0$ we apply arbitrary $\phi(x,v)\in W^{2,2}_0(\T^d\times B_R)$ as a test function in the weak formulation \eqref{eq: weak form ve,delta}, then utilize the bounds \eqref{eq:fp:lp}, \eqref{eq:fp:eff}, \eqref{eq:fp:ufd}, and \eqref{eq:fp:vf} to easily find
\begin{align*}
    \sup_{t\in [0,T]} \sup_{ \substack{ \phi \in W^{2,2}_0(\T^d\times B_R) \\ \|\phi\|_{ W^{2,2}(\T^d\times B_R)}\le 1 } } \left(\ddt \iint_{\T^d\times B_R} f^{\ve,\delta}(t) \phi(x,v) \,\dx\dv \right) \le C_{\ve,f_0,R,T}
\end{align*}
for some constant $C_{\ve,f_0,R,T}>0$ independent of $\delta$. It follows that
\begin{align*}
    \|f^{\ve,\delta}(t_2) - f^{\ve,\delta}(t_1)\|_{W^{-2,2}(\T^d\times B_R)} &=  \sup_{ \substack{ \phi \in W^{2,2}_0(\T^d\times B_R) \\ \|\phi\|_{ W^{2,2}(\T^d\times B_R)}\le 1 } } \int_{t_1}^{t_2} \ddt \left(\iint_{\T^d\times B_R} f^{\ve,\delta}(t) \phi(x,v)\,\dx\dv\right) \,\dt \\
    &\le C_{\ve,f_0,R,T}|t_2-t_1|.
\end{align*}
This completes the proof.
\end{proof}

\subsection{Passing to the limit $\delta \to 0$}

Let us fix $\ve>0$ and pass to $\delta \to 0$, aiming to obtain a global weak solution to the following system:
\begin{equation}\label{eq:regfp:ve}
\begin{cases}
\pa_t f^\ve + v\cdot\nabla_x f^\ve + E_\ve[f^\ve]\cdot \nabla_v f^\ve = \kappa \Delta_v f^\ve - \nabla_v \cdot ( (u_{f^\ve}^{(\ve)} - v)f^\ve), \\
E_\ve[f^\ve]:= \theta^\ve * \nabla \Phi_\ve[f^\ve], \\
-\Delta \Phi_\ve[f^\ve] = \theta^\ve * \rho_{f^\ve} - e^{\Phi_\ve[f^\ve]},\\[5pt]
f^\ve(0,x,v) = f_0(x,v).
\end{cases}
\end{equation}

\begin{lemma}\label{lem:fp:geve}
Let the initial data $f_0$ satisfy the conditions of Theorem \ref{thm:gws}. Then for each $\ve>0$ there exists a weak solution $f^\ve\in L^\infty_{\rm loc}([0,\infty);L^1\cap L^\infty(\T^d\times\R^d))$ to \eqref{eq:regfp:ve} satisfying the kinetic entropy inequality
\bq\label{eq:fp:kei:ve}
\calE^{(\ve)}[f^\ve](t) + \int_0^t\iint_{\T^d\times\R^d} \lt|\kappa\nabla_v \ln f^\ve -(u_{f^\ve}^{(\ve)}-v)\rt|^2 f^\ve \,\dx\dv \ds \le \calE^{(\ve)}[f_0].
\eq
In the above, $\calE^{(\ve)}$ is the regularized energy functional introduced in \eqref{eq: reg energy}.

\end{lemma}
\begin{proof}
By Lemma \ref{lem:fp:uni}, Sobolev embedding, and \eqref{eq:fp:eqcd}, we observe that for each $T,R>0$, the maps
\begin{align*}
   (0,T)\ni t\mapsto f^{\ve,\delta}(t) \in W^{-2,2}(\T^d\times B_R)
\end{align*}
are equibounded and equicontinuous (with respect to $\delta$) as members of $C([0,T];W^{-2,2}(\T^d\times B_R))$. By Ascoli's theorem we deduce that $\left\{f^{\ve,\delta}\right\}_\delta$ converges as $\delta\to 0$ in $C([0,T];W^{-2,2}(\T^d\times B_R))$ (modulo some subsequence $\delta_k^R$) to a limit $f^{\ve}_R$. Taking any sequence $R\to\infty$, and extracting diagonally, we deduce the existence of a subsequence $\delta_k$ (now independent of $R$) and a limit $f^\ve \in C([0,T];W^{-2,2}_{\rm loc}(\T^d\times \R^d)$ for which
\begin{align}\label{eq: delta neg sob}
    f^{\ve,\delta_k} \xrightarrow[k\to \infty]{} f^\ve \quad \text{in }  C([0,T];W^{-2,2}_{\rm loc}(\T^d\times\R^d)).
\end{align}

On the other hand, the uniform bounds \eqref{eq:fp:lp} and \eqref{eq:fp:mb} imply also that for each $t\ge 0$, we are guaranteed the existence of a further subsequence $\left\{\delta_j^t\right\}_j$ (possibly dependent on $t$) of $\delta_k$, for which $\left\{f^{\ve,\delta_j^t}(t)\right\}_j$ is weakly$^*$ convergent in $L^1\cap L^\infty(\T^d\times\R^d)$ to a limit $g^\ve(t)$. Clearly this implies that $\left\{f^{\ve,\delta_j^t}(t)\right\}_j$ converges in $\calD'(\T^d\times\R^d)$ to $g^\ve(t)$. In view of \eqref{eq: delta neg sob}, however, we observe that any subsequence of $f^{\ve,\delta_k}(t)$ that is convergent in $\calD'(\T^d\times\R^d)$ must necessarily converge to the same limit $f^\ve(t)$. It follows by the unique identification of this limit that the full sequence $\left\{f^{\ve,\delta_k}(t)\right\}_k$ converges in $\calD'(\T^d\times\R^d)$ to $g^\ve(t) = f^\ve(t)$. It is further easily checked that $f^\ve$ as defined in \eqref{eq: delta neg sob} verifies
\begin{align*}
    f^{\ve,\delta_k} \xrightarrow[k\to\infty]{} f^\ve \quad \text{weakly$^*$ in }  L^\infty(0,T;L^1(\T^d\times\R^d)).
\end{align*}

Another consequence of the uniform bounds \eqref{eq:fp:lp} and \eqref{eq:fp:mb} is that, by the averaging lemma of \cite{karpermellettrivisa2013}, for any $\varphi\in C^0(\R^d)$ with $|\varphi(v)|\lesssim 1+|v|$ the family
\begin{align*}
    \left\{ \int_{\R^d} f^{\ve,\delta}(t,x,v)\varphi(v)\,\dv\right\}_{\delta\in (0,1]}
\end{align*}
is relatively compact in $L^1((0,T)\times\T^d)$. In particular, we have that (modulo a further subsequence of $\delta_k$)
\begin{equation}\label{eq: delta average}
\begin{split}
    &\rho_{f^{\ve,\delta}} \xrightarrow[\delta\to 0]{} \rho_{f^\ve} \quad \text{in }  L^1((0,T)\times\T^d),\\
    &\rho_{f^{\ve,\delta}} u_{f^{\ve,\delta}} \xrightarrow[\delta\to 0]{} \rho_{f^\ve} u_{f^\ve}\quad \text{in }   L^1((0,T)\times\T^d).
    \end{split}
\end{equation}
By $L^p$-interpolation with \eqref{eq:fp:dub}, this implies
\begin{align*}
    &\rho_{f^{\ve,\delta}} \xrightarrow[\delta\to0]{} \rho_{f^\ve} \quad \text{in} \quad L^p((0,T)\times\T^d), \quad \forall p\,\in \left[1,\frac{d+2}{d}\right).
\end{align*}

Let us now pass to the limit in each term of the weak formulation \eqref{eq: weak form ve,delta} satisfied by $f^{\ve,\delta}$. In regard of the left-hand side, the weak convergence of $\left\{f^{\ve,\delta}(t)\right\}_\delta$ in $L^1\cap L^\infty(\T^d\times\R^d)$ demonstrates that
\begin{align*}
    \iinttr f^{\ve,\delta}(t) \phi(t,x,v) \,\dx\dv \xrightarrow[\delta\to 0]{} \iinttr f^\ve(t) \phi(t,x,v)\,\dx\dv.
\end{align*}
For the term of the left-hand side, which contains the initial datum, the Lebesgue dominated convergence theorem shows
\begin{align*}
    \iinttr f^{(\delta,0)}(x,v) \phi(0,x,v)\,\dx\dv \xrightarrow[\delta\to 0]{} \iinttr f_0(x,v) \phi(0,x,v)\,\dx\dv.
\end{align*}

We now discuss the terms on the right-hand side of \eqref{eq: weak form ve,delta}. From Lemma \ref{lem : stability for Lp} and Young's inequality for convolutions, we easily deduce
\begin{align*}
    \|E_\ve[f^{\ve,\delta}] - E_\ve[f^\ve]\|_{L^2((0,T)\times\T^d)} \le C\|\rho_{f^{\ve,\delta}} - \rho_{f^\ve}\|_{L^p((0,T)\times\T^d)} \to 0 \quad \text{as}\quad \delta \to 0.
\end{align*}
This implies that for any test function $\varphi\in C_c^\infty([0,\infty)\times\T^d\times\R^d)$:
\begin{align}\label{eq: delta to zero 1}
    \int_0^\infty \iinttr f^{\ve,\delta} E_\ve[f^{\ve,\delta}]\cdot \nabla_v \varphi\,\dx\dv\dt \xrightarrow[\delta\to 0]{} \int_0^\infty \iinttr f^\ve E_\ve[f^\ve] \cdot \nabla_v \varphi\,\dx\dv\dt.
\end{align}

Next we are concerned with the term $(\kappa + \delta \mathbf{1}_{\kappa=0}) \Delta_v f^{\ve,\delta}$. In the case $\kappa=0$, then for any test function $\varphi$ we have that
\begin{align}\label{eq: delta to zero 2}
    \left|\int_0^\infty \iinttr \delta f^{\ve,\delta} \Delta_v \varphi\,\dx\dv\dt \right| \le \delta \|f^{\ve,\delta}\|_{L^\infty(\textnormal{supp}(\varphi))} \|\Delta_v \varphi\|_{L^1} \xrightarrow[\delta\to 0]{} 0.
\end{align}
In the case $\kappa>0$, then \eqref{eq:fp:lp} implies that $\nabla_v f^{\ve,\delta}$ is weakly compact in $L^2_{\rm loc}((0,\infty);L^2(\T^d\times\R^d))$ and thus
\begin{align} \label{eq: delta to zero 3}
    \int_0^\infty \iinttr \kappa \nabla_v f^{\ve,\delta} \cdot \nabla_v \varphi\,\dx\dv\dt \xrightarrow[\delta\to 0]{} \int_0^\infty \iinttr \kappa \nabla_v f^\ve \cdot \nabla_v \varphi\,\dx\dv\dt 
\end{align}
for all test functions $\varphi$.

Finally, we consider in \eqref{eq:regfp} the last nontrivial term $\theta^\delta * u_{f^{\ve,\delta}}^{(\ve)} f^{\ve,\delta}$. We claim that up to a subsequence
\begin{align}\label{eq: conv of reg u delta}
    \theta^\delta * u_{f^{\ve,\delta}}^{(\ve)} \xrightarrow[\delta\to 0]{} u_{f^\ve}^{(\ve)} \quad \text{in}\quad L^r_{\rm loc}([0,\infty)\times\T^d),\quad \forall r\in [1,\infty).
\end{align}
Recall that by definition, $\left\{u_{f^{\ve,\delta}}^{(\ve)}\right\}_{\delta}$ is bounded in $L^\infty([0,\infty)\times\T^d)$ (by $1/\ve$). We notice from \eqref{eq: delta average} that $u_{f^{\ve,\delta}}^{(\ve)}$ is almost everywhere convergent to $u_{f^\ve}^{(\ve)}$. Therefore by Vitali's dominated convergence theorem $u_{f^{\ve,\delta}}^{(\ve)}\to u_{f^\ve}$ strongly in $L^r([0,T]\times\T^d)$ for any $T>0$ and $r\in [1,\infty)$. Thereon, we use the fact that $\theta^\delta$ is an approximation of the identity, and then a diagonal argument (taking $T\to\infty$), to deduce that \eqref{eq: conv of reg u delta} holds. It follows that for any test function $\varphi$:
\begin{align}\label{eq: delta to zero 4}
    \int_0^\infty \iinttr \nabla_v\varphi \cdot \Big(( \theta^\delta * u_{f^{\ve,\delta}}^{(\ve)} - v ) f^{\ve,\delta} \Big) \,\dx\dv\dt \xrightarrow[\delta\to 0]{} \int_0^\infty \iinttr \nabla_v\varphi \cdot \Big((  u_{f^\ve}^{(\ve)} - v ) f^\ve \Big) \,\dx\dv\dt .
\end{align}
Collecting \eqref{eq: delta to zero 1}, \eqref{eq: delta to zero 2}, \eqref{eq: delta to zero 3}, and \eqref{eq: delta to zero 4}, we deduce that the limit $f^\ve$ is a weak solution to \eqref{eq:regfp:ve}.

We finally show the entropy inequality \eqref{eq:fp:kei:ve}. Let us remark that due to the $\ve$-regularization, we still have $\|E_\ve[f^\ve]\|_{L^\infty(\T^d)} \le C_\ve$ and $|u_{f^\ve}^{(\ve)}|\le 1/\ve$. Using these bounds, we find that $|v|^2$ is admissible as a test function in the weak formulation for $f^\ve$. Furthermore, as explained in the proof of \eqref{eq:fp:entropy}, thanks to these regularizations, we can show as in \cite[Lemma 3.3]{bostan2024} that $f^\ve$ is renormalized. In particular, it is deduced that $t\mapsto \calE^{(\ve)}[f^\ve](t)$ is absolutely continuous and
\begin{equation}\label{eq: reg kinetic identity}
\begin{split}
    \ddt \calE^{(\ve)}[f^\ve](t) &= - \iinttr \frac{1}{f^\ve}|\kappa \nabla_v f^\ve + (v- u_{f^\ve}^{(\ve)}) f^\ve|^2 \,\dx\dv  + \intt \rho_{f^\ve} (|u_{f^\ve}^{(\ve)}|^2 - u_{f^\ve}\cdot u_{f^\ve}^{(\ve)})\,\dx.
\end{split}
\end{equation}
The above identity is deduced via the Poisson--Boltzmann equation in the following way. Taking the time derivative of $\iinttr|v|^2 f^\ve(t)\dx\dv$, we notice that in the same manner as in \eqref{eq: formal v^2}, there is precisely one term which regards the electric field:
\begin{align}\label{eq: E vf}
    \iinttr E_\ve[f^\ve] \cdot v f^\ve\,\dx\dv.
\end{align}
To handle this term, we notice that for any $t\in (0,T]$, $\varphi := \mathbf{1}_{[0,t]\times\T^d\times\R^d}$ is admissible as a test function (through approximation) in the weak formulation of $f^\ve$, which shows the mass continuity equation
\begin{align*}
    \p_t \rho_{f^\ve} + \nabla_x \cdot (\rho_{f^\ve} u_{f^\ve}) = 0 \quad \text{in }\calD'((0,T)\times\T^d).
\end{align*}
From the defining equations
\begin{align*}
    E_\ve[f^\ve] := -\theta^\ve * \Delta\Phi_\ve[f^\ve], \quad -\Delta\Phi_\ve[f^\ve] = \theta^\ve * \rho_{f^\ve} - e^{\Phi_\ve[f^\ve]},
\end{align*}
we can therefore obtain
\begin{align*}
    -\Delta \p_t \Phi_\ve[f^\ve] &= \theta^\ve * \p_t \rho_{f^\ve} - \p_t\Phi_\ve[f^\ve] e^{\Phi_\ve[f^\ve]}\\
    &= -\theta^\ve * \nabla_x\cdot (\rho_{f^\ve}u_{f^\ve}) - \p_t\Phi_\ve[f^\ve] e^{\Phi_\ve[f^\ve]} \quad \text{in}\quad \calD'(\T^d).
\end{align*}
Utilizing these, we can rewrite \eqref{eq: E vf} as
\begin{align*}
    \iinttr E_\ve[f^\ve]\cdot vf^\ve\,\dx\dv &= -\intt \nabla_x (\theta^\ve * \Phi_\ve[f^\ve]) \cdot (\rho_{f^\ve}u_{f^\ve})\,\dx \\
    &= \intt (\theta^\ve * \Phi_\ve[f^\ve]) \nabla_x\cdot (\rho_{f^\ve} u_{f^\ve})\,\dx \\
    &= \intt \Phi_\ve[f^\ve]\, \Big( \theta^\ve * \nabla_x\cdot( \rho_{f^\ve} u_{f^\ve}) \Big) \,\dx \\
    &= \intt  \Phi_\ve[f^\ve] \Big( \Delta \p_t \Phi_\ve[f^\ve] - e^{\Phi_\ve[f^\ve]} \p_t \Phi_\ve[f^\ve] \Big) \,\dx \\
    &= \ddt \left(\frac{1}{2}\intt |\nabla\Phi_\ve[f^\ve]|^2 + (\Phi_\ve[f^\ve] - 1) e^{\Phi_\ve[f^\ve]} + 1 \,\dx\right).
\end{align*}
Combining this with the computations for the entropy, which may be obtained as in \eqref{eq: entropy identity}, we deduce \eqref{eq: reg kinetic identity}. Then, we note that by definition of $u_{f^\ve}^{(\ve)}$, the last term in \eqref{eq: reg kinetic identity} is non-positive. We discard it to deduce the inequality \eqref{eq:fp:kei:ve}.
\end{proof}

%%%%%%%%%%%%%%%%%%%%%%%%%%%%%%%%%%%%%%%%%%%%%%%%%%%%%%%%%
%
%
%
%
%
%%%%%%%%%%%%%%%%%%%%%%%%%%%%%%%%%%%%%%%%%%%%%%%%%%%%%%%%%

\subsection{Passing to the limit $\ve \to 0$ and proof of Theorem \ref{thm:gws}}

The importance of the kinetic entropy inequality \eqref{eq:fp:kei:ve} established in the previous section lies in that it immediately yields uniform-in-$\ve$ $2$-velocity moment estimates for $f^\ve$, provided that there is a lower bound for $\kappa\iinttr f^\ve\log f^\ve\,\dx\dv$ that is comparable with the second moments. Although the proof is classical, we include the material here for the sake of completeness.

\begin{lemma}\label{lem:flogf:ctrl}
For any $f \in L^1_+(\R^d_v)$, it holds that for any $c>0$:
\[
\int_{\R^d} f \log f \,\dv \ge \intr -\frac{1}{4c}|v|^2f\,\dv -C_{c,d},
\]
where we assume the integrals are finite on both sides.
\end{lemma}

We decompose the domain of $\R^d$ into four disjoint sets as:
\[
A_1:=\{f \ge 1\},\ A_2:= \{1>f \ge e^{-\frac{|v|^2}{4c}}\},\ A_3:=\{f < e^{-\frac{|v|^2}{4c}}, \,|v|\ge2\sqrt{c}\},\ A_4:=\{f < e^{-\frac{|v|^2}{4c}}, \,|v| <2\sqrt{c} \}.
\]
It is easy to note that
\[
\int_{A_1} f\log f \,\dv \ge 0,\quad \int_{A_2} f\log f\,\dv \ge \int_{A_2} -\frac{1}{4c}|v|^2f\,\dv.
\]
We then observe that the function
\[
(0,\infty)\ni s \mapsto s\log s
\]
is decreasing in $(0,e^{-1})$, increasing in $(e^{-1},\infty)$, and has minimum $-e^{-1}$. Using these, we obtain
\[
\int_{A_3} f\log f\,\dv \ge \int_{A_3} -\frac{|v|^2}{4c} e^{-\frac{|v|^2}{4c}}\,\dv = -C_{c,d}
\]
and
\[
\int_{A_4} f\log f\,\dv \ge \int_{|v|\le 2\sqrt{c}} -\frac{1}{e}\,\dv = -C_{c,d}.
\]

Let us now show the following analogue of Lemma \ref{lem:fp:uni}, which provides estimates uniform in $\ve$ and is a cornerstone for passing to the limit.
\begin{lemma}[Uniform estimates for \eqref{eq:regfp:ve}] \label{lem:fp:unive}
For the solution $f^\ve$ as constructed in Lemma \ref{lem:fp:geve}, it holds that
\begin{enumerate}
    \item[(i)] $L^p$-estimates: for all $p\in (1,\infty]$ and each $t\ge 0$,
    \begin{align*}
        \|f^\ve(t)\|_{L^p(\T^d\times\R^d)} + \kappa \left(\int_0^T \iinttr |\nabla_v (f^\ve)^{p/2}|^2 \dx\dv\dt\right)^{1/p} \le \|f_0\|_{L^p(\T^d\times\R^d)} e^{\frac{dT}{p'}},
    \end{align*}
    where the middle term is interpreted as zero when $p=\infty$.
    \item[(ii)] Mass conservation:
    \begin{align*}
        \|f^{\ve}(t)\|_{L^1(\T^d\times\R^d)} = \|f_0\|_{L^1(\T^d\times\R^d)} .
    \end{align*}
    \item[(iii)] Bounds on electric fields:
    \begin{align*}
        \sup_{0\le t \le T} \|E_\ve[f^\ve]\|_{L^2(\T^d)} \le C_{f_0,T}, \quad \sup_{0\le t \le T} \|E_\ve[f^\ve]f^\ve\|_{L^2(\T^d\times B_R)} \le C_{f_0,R,T}.
    \end{align*}
    \item[(iv)] Moment bounds:
    \begin{align*}
        \sup_{0\le t \le T} \iinttr |v|^2 f^\ve\,\dx\dv \le C_{f_0,T}.
    \end{align*}
    \item[(v)] Macroscopic upper bounds: for each $T>0$,
    \begin{align*}
        \sup_{0\le t \le T} \|\rho_{f^\ve}\|_{L^{\frac{d+2}{d}}(\T^d)} \le C_{f_0,T}, \quad \sup_{0\le t \le T} \|\rho_{f^\ve}u_{f^\ve}\|_{L^{\frac{d+2}{d+1}}(\T^d)} \le C_{f_0, T}.
    \end{align*}
    \item[(vi)] Bounds on alignment terms:
    \begin{align*}
        &\sup_{0\le t \le T} \|u_{f^\ve}^{(\ve)} f^\ve\|_{L^2(\T^d\times\R^d)} \le \sup_{0\le t \le T} \|u_{f^\ve}f^\ve\|_{L^2(\T^d\times\R^d)} \le  C_{f_0,T},\\
        &\sup_{0\le t \le T} \|vf^\ve\|_{L^2(\T^d\times\R^d)} \le C_{f_0,T}.
    \end{align*}
    \item[(vii)] Equicontinuity: for any $0\le t_1 \le t_2 \le T$ and each $R>0$:
    \begin{align*}
        \|f^\ve(t_1)-f^\ve(t_2)\|_{W^{-2,2}(\T^d\times B_R)} \le C_{f_0,R,T}|t_1-t_2|.
    \end{align*}
\end{enumerate}
\end{lemma}
\begin{proof}
(i) simply carries over from \eqref{eq:fp:lp}, due to the lower semicontinuity of $L^p$-norms. (ii) follows from \eqref{eq:fp:mc1}, since we showed in the proof of Lemma \ref{lem:fp:geve} that $f^{\ve,\delta}(t) \weakto f^\ve(t)$ weakly in $L^1(\T^d\times\R^d)$ for each $t\ge 0$, whereas clearly $f^{(\delta,0)}\to f_0$ strongly in $L^1(\T^d\times\R^d)$.

For (iii) and (iv), we recall Lemma \ref{lem:flogf:ctrl} that the Boltzmann entropy admits a lower bound which is comparable with the second moments. Hence, \eqref{eq:fp:kei:ve} provides uniform bounds for both $\iinttr|v|^2 f^\ve$ and $\|E_\ve[f^\ve]\|_{L^2}$. Using the latter along with (i), the remaining assertion of (iii) is easy to obtain.

Then, we obtain (v) by interpolating (i) and (iv), for which we again refer to \cite[Lemma 5.2]{griffinpickeringiacobelli2021a}. The bounds in (vi) are straightforward to obtain by utilizing (i), (iv), and the definition of $u_{f^\ve}^{(\ve)}$ as such:
\begin{align*}
    \iinttr |u_{f^\ve}^{(\ve)}|^2 (f^\ve)^2\,\dx\dv &\le \|f^\ve\|_{L^\infty} \iinttr |u_{f^\ve}|^2 f^\ve\,\dx\dv \\
    &= C_T \intt \rho_{f^\ve} |u_{f^\ve}|^2\,\dx \\
    &\le C_T \iinttr |v|^2 f^\ve\,\dx\dv.
\end{align*}
The final inequality of the above is due to the Cauchy--Schwarz inequality. Similarly
\begin{align*}
    \iinttr |v|^2 (f^\ve)^2\,\dx\dv \le \|f^\ve\|_{L^\infty} \iinttr|v|^2 f^\ve\,\dx\dv \le C_T \iinttr|v|^2 f^\ve\,\dx\dv,
\end{align*}
and so we deduce (vi).

Utilizing all of the above, we can obtain (vii) via the same way as in Lemma \ref{lem:fp:uni}, by taking $\phi\in W^{2,2}_0(\T^d\times B_R)$ as a test function in \eqref{eq:regfp:ve}, then using the uniform estimates established so far to obtain a bound for
\[
\sup_{t\in [0,T]}\sup_{ \substack{ \phi \in W^{2,2}_0(\T^d\times B_R) \\ \|\phi\|_{ W^{2,2}(\T^d\times B_R)}\le 1 } } \left(\ddt \iint_{\T^d\times B_R} f^\ve(t) \phi(x,v)\dx\dv \right).
\]
This completes the proof.
\end{proof}

Next, we utilize the uniform estimates of Lemma \ref{lem:fp:unive} in order to pass to the limit $\ve \to 0$.

\begin{lemma}
    Let $\{f^\ve\}$ denote the family of solutions to \eqref{eq:regfp:ve} as constructed in Lemma \ref{lem:fp:geve}, with initial data $f_0$. Then there is a subsequence $\ve_k$ such that $\left\{f^{\ve_k}\right\}_k$ converges as a sequence in $C([0,\infty);W_{\rm loc}^{-2,2}(\T^d\times\R^d))$, locally uniformly in time, to a limit $f(t)$:
    \begin{align*}
        \lim_{k\to\infty}\sup_{0\le t \le T}\left|\iinttr (f^{\ve_k}-f)\phi(x,v)\,\dx\dv\right| = 0, \quad \forall\, T>0, \quad \forall\, \phi \in C_c^1(\T^d\times \R^d).
    \end{align*}
    Moreover, we have that
    \begin{align*}
        &f^{\ve_k}(t) \weakto f(t) \quad \text{weakly in}\quad L^p(\T^d\times\R^d) \quad \forall\, t \ge 0, \quad \forall\, p\in[1,\infty),\\
        &f^{\ve_k}(t) \weakto f(t) \quad \text{weakly$^*$ in}\quad L^\infty(\T^d\times\R^d) \quad \forall\, t \ge 0,
    \end{align*}
    the limit $f$ satisfies the mass conservation
    \begin{align*}
    \|f(t)\|_{L^1(\T^d\times\R^d)} = \|f_0\|_{L^1(\T^d\times\R^d)} \quad \forall t\ge 0,  
    \end{align*}
    and is a solution to \eqref{main eq} in the sense of Theorem \ref{thm:gws}. 
\end{lemma} 
\begin{proof}
Similarly as in the proof of Lemma \ref{lem:fp:geve}, we utilize the time-equicontinuity of $f^\ve$ to deduce that, modulo a subsequence (which is independent of $t$), we have for each $t\ge 0$ that
\bq\label{eq:fp:ft:wc}
    f^{\ve}(t) \weakto f(t) \quad \text{weakly in }L^p(\T^d\times\R^d), \quad \text{weakly$^*$ in }L^\infty(\T^d\times\R^d), \quad \forall \, t\in [0,T], \ \forall \, p \in [1,\infty),
\eq
and
\begin{equation}\label{eq: limit f}
f^\ve \to f \quad \text{in} \quad C([0,T];W_{\rm loc}^{-2,2}), \quad \forall\, T>0.
\end{equation}
The first consequence of this is the convergence of the following term from the weak formulation of \eqref{eq:regfp:ve}:
\begin{align*}
    \iinttr f^\ve(t)\phi(t,x,v)\,\dx\dv \to \iinttr f(t) \phi(t,x,v)\,\dx\dv \quad \forall\, \phi\in C_c^2([0,t]\times\T^d\times\R^d), \quad \forall\, t\ge 0.
\end{align*}
Second, we obtain the mass conservation of $f$ by considering in \eqref{eq:fp:ft:wc} the case $p=1$. Lastly, it is also standard to check using any cutoff, that \eqref{eq:fp:ft:wc} implies the lower semicontinuity of moments 
\begin{align}\label{eq: lower semcont moments}
    \iinttr |v|^2 f(t)\,\dx\dv \le \liminf_{\ve\to 0} \iinttr |v|^2 f^\ve(t)\,\dx\dv \le C_{f_0}.
\end{align}

The uniform-in-$\ve$ estimates in Lemma \ref{lem:fp:unive}, along with averaging lemmas \cite{dipernalionsmeyer, karpermellettrivisa2013}, provide that up to a further subsequence
\bq\label{eq:fp:mvstr}
\begin{split}
    &\rho_{f^\ve} \to \rho_f \quad \text{in}\quad L^q((0,T)\times \T^d) \quad \forall\, q\in \left[1,\frac{d+2}{d}\right) \quad \text{and a.e.},\\
    &\rho_{f^\ve}u_{f^\ve} \to \rho_f u_f \quad \text{in}\quad L^r((0,T)\times\T^d) \quad \forall\, r\in \left[1,\frac{d+2}{d+1}\right) \quad \text{and a.e.}\\
\end{split}
\eq
We note that \eqref{eq:fp:mvstr} and Proposition \ref{lem : stability for Lp} imply
\begin{align*}
    E_\ve[f^\ve] \to E[f] \quad \text{in} \quad L^2([0,T]\times\T^d).
\end{align*}
Therefore, in passing to the limit in the other terms of the weak formulation of \eqref{eq:regfp:ve}, the only remaining difficulty is to establish
    \begin{align}
    {u^{(\ve)}_{f^\ve}} f^\ve \to u_f f \quad \text{in}\quad \calD'([0,\infty)\times\T^d\times\R^d) \label{eq FP : conv of uf}.
\end{align}

Let us enter the proof of \eqref{eq FP : conv of uf}. Fix any $T>0$ and small $\eta>0$. We split the domain $[0,T]\times\T^d$ into $\{(t,x):\rho_f \le \eta\}$ and $\{(t,x):\rho_f > \eta\}$. For the former, we observe that
\begin{align*}
\iiint_{\{\rho_f \le \eta\} \times \R^d } |u_f f| \,\dx\dt\dv &= \iint_{\{\rho_f \le \eta\}} |\rho_f u_f|\,\dx\dt \\
& \le \lt(\iint_{\{\rho_f \le \eta\} } \rho_f \,\dx\dt \rt)^{1/2} \lt(\int_0^T \intt \rho_f |u_f|^2\,\dx\dt \rt)^{1/2} \\
&\le \lt(\iint_{\{\rho_f \le \eta\} } \rho_f \,\dx\dt \rt)^{1/2} \left(\int_0^T\iinttr |v|^2 f\,\dx\dv\dt\right)^{1/2} \le C_T\sqrt{\eta}. 
\end{align*}
Similarly,
\begin{align*}
\iiint_{\{\rho_f\le \eta\}\times\R^d } |u_{f^\ve}^{(\ve)} f^\ve|\dx\dt\dv &\le \iiint_{\{\rho_f \le \eta\} \times \R^d} |u_{f^\ve} f^{\ve}| \,\dx\dt\dv \\
&= \iint_{\{\rho_f \le \eta\}} |\rho_{f^\ve} u_{f^\ve}|\,\dx\dt \\
&\le \lt(\iint_{\{\rho_f \le \eta\}} \rho_{f^\ve} \,\dx\dt \rt)^{1/2} \lt(\int_0^T \intt \rho_{f^\ve} |u_{f^\ve}|^2\,\dx\dt \rt)^{1/2} \\
&\le \lt(\iint_{\{\rho_f \le \eta\}} \rho_{f^\ve} \,\dx\dt \rt)^{1/2} \left(\int_0^T \iinttr |v|^2 f^\ve\,\dx\dv\dt\right)^{1/2}.
\end{align*}
Hence, applying the strong convergence $\eqref{eq:fp:mvstr}_1$ to the above inequality, we deduce
\[
\limsup_{\ve \to 0}\iiint_{\{\rho_f \le \eta\} \times \R^d} |u_{f^\ve}^{(\ve)} f^{\ve}| \,\dx\dv \dt \le C_T\sqrt{\eta},
\]
and altogether we get
\begin{align}\label{eq: sqrt eta}
    \|u^{(\ve)}_{f^\ve} f^\ve - u_f f\|_{L^1(\{\rho_f \le \eta\}\times\R^d)} \le C_T \sqrt{\eta}.
\end{align}

Next, let us discuss convergence on $\{(t,x):\rho_f > \eta\}\times \R^d$. We proceed as in \cite{carrillochoijung2021, karpermellettrivisa2013}, utilizing \eqref{eq:fp:mvstr} and Egorov's theorem to find a set $B_\eta$ for which
\begin{align*}
    &B_\eta \subset \{(t,x):\rho_f > \eta\},\quad  |B_\eta|\le \eta^{\frac{d+2}{2}},\\
    &\textnormal{$u_{f^\ve}^{(\ve)}\to u_f$ uniformly in $\{\rho_f>\eta\}\setminus B_\eta$}.
\end{align*}
On $B_\eta\times\R^d$ we have the following quick estimates:
\begin{align*}
    \iiint_{B_\eta\times \R^d} |u_{f^\ve}^{(\ve)}f^\ve - u_f f|\,\dx\dv\dt &\le \iint_{B_\eta} |\rho_{f^\ve}u_{f^\ve}| + |\rho_f u_f |\,\dx\dt \\
    &\le \Big(\|\rho_{f^\ve}u_{f^\ve}\|_{L^{\frac{d+2}{d+1}}} + \|\rho_f u_f\|_{L^{\frac{d+2}{d+1}}}\Big) |B_\eta|^{\frac{1}{d+2}} \\
    &\le C\sqrt{\eta},
\end{align*}
thanks to H\"older's inequality and the fact that $\{\rho_{f^\ve}u_{f^\ve}\}$ is bounded in $L^{\frac{d+2}{d+1}}([0,T]\times\T^d)$ (see Lemma \ref{lem:fp:unive} (v)). 

Now toward the aim of obtaining convergence on $(\{\rho_f > \eta\}\setminus B_\eta)\times\R^d$, we telescope the difference in \eqref{eq FP : conv of uf} as
\begin{equation}\label{eq: telescope uf}
u^{(\ve)}_{f^\ve} f^\ve - u_f f = (u^{(\ve)}_{f^\ve} -u_f)f^\ve +  u_f(f^\ve -f).    
\end{equation}
It is clear that
\begin{align*}
    \|(u_{f^\ve}^{(\ve)} - u_f) f^\ve\|_{L^1((\{\rho_f>\eta\} \setminus B_\eta) \times \R^d) } \le \|u_{f^\ve}^{(\ve)} - u_f\|_{L^\infty(\{\rho_f>\eta\}\setminus B_\eta) } \|\rho_{f^\ve}\|_{L^1(\{\rho_f>\eta\} \setminus B_\eta) } \to 0.
\end{align*}

To control the second term on the right-hand side of \eqref{eq: telescope uf}, let $\phi \in C_c^\infty([0,\infty)\times \T^d \times \R^d)$ be given. Assume that $T>0$ is large enough, and choose $R>0$ so that it holds:
\[
\text{supp} (\phi) \subset [0,T] \times \T^d \times B_R.
\]
We observe that from Lemma \ref{lem:fp:unive}, \eqref{eq: limit f}, and the uniqueness of weak limits, it holds that
\[
f^\ve \weakto f \quad \text{weakly* in } L^\infty([0,T]\times \T^d \times \R^d).
\]
Since
\[
u_f \mathbf{1}_{\rho_f >\eta} \mathbf{1}_{|v| \le R} \in L^1([0,T]\times \T^d \times \R^d),
\]
we find from the definition of weak$^*$ convergence
\[
\int_0^T\iint_{\{\rho_f > \eta\} \times \R^d} u_f (f^\ve - f) \phi\,\dx\dv\dt \to 0.
\]
We conclude that \eqref{eq: telescope uf} tends to zero as $\ve \to 0$ on $(\{\rho_f>\eta\}\setminus B_\eta) \times \R^d$. Combining this with the estimate \eqref{eq: sqrt eta}, we deduce \eqref{eq FP : conv of uf} since our choice of $\eta>0$ was arbitrary.
\end{proof}

%%%%%%%%%%%%%%%%%%%%%%%%%%%%%%%%%%%%%%%%%%%%%%%%%%%%%%%%%
%
%
%
%
%
%%%%%%%%%%%%%%%%%%%%%%%%%%%%%%%%%%%%%%%%%%%%%%%%%%%%%%%%%
%\subsection{Kinetic entropy inequality}%Proof of Theorem \ref{thm:gws} for FP case}

From the previous lemma we now have that the limit $f$ is indeed a weak solution to \eqref{main eq}. The proof of Theorem \ref{thm:gws} is completed once we verify that the kinetic entropy inequality \eqref{eq : KEI} holds for $f$. We verify now by showing that all terms arising in \eqref{eq:fp:kei:ve} either converge or attain lower semicontinuity at the limit.

Let us recall the kinetic entropy inequality here, for the reader's convenience. The inequality \eqref{eq:fp:kei:ve} states that
\begin{align}
&\iint_{\T^d\times\R^d} \frac{1}{2}|v|^2f^\e + \kappa f^\e \log f^\e \,\dx\dv + \int_{\T^d} \frac{1}{2}|\nabla\Phi_\e[f^\e]|^2 + (\Phi_\ve[f^\ve]-1)e^{\Phi_\ve[f^\ve]} \,\dx \label{eq : regkei, 1} \\
&\quad +\int_0^t \iint_{\T^d \times \R^d}\frac{1}{f^\e} \lt|\kappa\nabla_v f^\e -({u^{(\ve)}_{f^\ve}} -v)f^\e\rt|^2 \,\dx\dv\ds \label{eq : regkei, 2}\\
&\qquad \le \iint_{\T^d\times\R^d} \frac{1}{2}|v|^2f_0 + \kappa f_0 \log f_0 \,\dx\dv + \int_{\T^d}  \frac{1}{2}|\nabla\Phi_\e[f_0]|^2 + (\Phi_\ve[f_0]-1)e^{\Phi_\ve[f_0]} \,\dx. \label{eq: regkei rhs}
\end{align}
We check each line separately in the limit $\ve\to 0$, passing to subsequences if necessary. Our claim is basically that the right-hand side \eqref{eq: regkei rhs} is convergent, whereas \eqref{eq : regkei, 1}--\eqref{eq : regkei, 2} are lower semicontinuous. Notice that these assertions prove \eqref{eq : KEI}.

%%%%%%%%%%%%%%%%%%%%%%%%%%%%%%%%%%%%%%%%%%%%%%%%%%%%%%%%%
%
%
%
%
%
%%%%%%%%%%%%%%%%%%%%%%%%%%%%%%%%%%%%%%%%%%%%%%%%%%%%%%%%%

\subsubsection{Convergence of \eqref{eq: regkei rhs}}

Only the last integral of \eqref{eq: regkei rhs} is dependent on $\ve$. Recall that the strong stability of this energy term is provided in Lemma \ref{lem : conv of energy}. It suffices to check that $\theta^\ve * \rho_{f_0} \to \rho_{f_0}$ in $L^p(\T^d)$ for some $p>1$, which is clear by standard properties of the mollifier and the fact that $\rho_{f_0} \in L^{\frac{d+2}{d}}(\T^d)$.

%%%%%%%%%%%%%%%%%%%%%%%%%%%%%%%%%%%%%%%%%%%%%%%%%%%%%%%%%
%
%
%
%
%
%%%%%%%%%%%%%%%%%%%%%%%%%%%%%%%%%%%%%%%%%%%%%%%%%%%%%%%%%

\subsubsection{Lower semicontinuity of \eqref{eq : regkei, 1}}

From the weak $L^1$-convergence of $f^\ve$ asserted in \eqref{eq:fp:ft:wc}, we can use $\varphi = \varphi(x) \in L^\infty(\T^d\times\R^d)$ as a test function to deduce that for each $t\ge 0$, $\rho_{f^\ve}(t) \weakto \rho_f(t)$ weakly in $L^1(\T^d)$. The bounds for $\|\rho_{f^\ve}(t)\|_{L^{\frac{d+2}{d}}}$ in Lemma \ref{lem:fp:unive} and the uniqueness of weak limits imply further that
\begin{align*}
    \rho_{f^\ve}(t) \weakto \rho_f(t) \quad \text{weakly in }L^p(\T^d), \quad \forall t \ge 0, \quad \forall p\in \left[1,\frac{d+2}{d}\right].
\end{align*}
With the aid of Lemma \ref{lem:wsta}, we deduce
\begin{align} \label{eq: conv pot energy}
\intt \frac{1}{2}|\nabla\Phi[f]|^2 + (\Phi[f]-1)e^{\Phi[f]}\,\dx \le \lim_{\ve \to 0} \intt \frac{1}{2}|\nabla\Phi_\ve[f^\ve]|^2 + (\Phi_\ve[f^\ve]-1)e^{\Phi_\ve[f^\ve]}\,\dx.
\end{align}

If $\kappa=0$, then lower semicontinuity of \eqref{eq : regkei, 1} thus follows immediately from \eqref{eq: conv pot energy} and the lower semicontinuity of moments \eqref{eq: lower semcont moments}. In the case $\kappa\ne 0$, we need to additionally deduce the lower semicontinuity of $\kappa\iinttr f^\ve\log f^\ve\dx\dv$. This assertion is well illustrated in say \cite{AFP2000, FigalliGlaudo2023, JKO1998}, so we only sketch the ideas. Since $f^\ve(t) \to f(t)$ weakly in $L^1(\T^d\times\R^d)$, the convexity of $z\mapsto z\log z$ shows that the entropy is lower semicontinuous on sets of finite measure:
\begin{align*}
    \iint_{\T^d\times B_R} f\log f\,\dx\dv \le \liminf_{\ve\to 0} \iint_{\T^d\times B_R} f^\ve \log f^\ve \,\dx\dv, \quad \forall R>0.
\end{align*}
On the other hand, the map $z\mapsto \max\{0,z\log z\}$ is convex and nonnegative, thus
\begin{align*}
    \iint_{\T^d\times (\R^d\setminus B_R)} \max\{0,f\log f\} \,\dx\dv \le \liminf_{\ve\to 0} \iint_{\T^d\times (\R^d\setminus B_R)} \max\{0,f^\ve\log f^\ve\}\,\dx\dv.
\end{align*}
Utilizing from Lemma \ref{lem:fp:unive} the uniform bounds for the moments of $f^\ve$, we can obtain the estimate of \cite[P.7]{JKO1998}, which gives
\begin{align*}
    \lim_{R\to\infty} \sup_{\ve>0} \iint_{\T^d\times (\R^d\setminus B_R)} \min\{0, f^\ve \log f^\ve\}\,\dx\dv = 0.
\end{align*}
Combining the above, we deduce
\begin{align*}
    \iinttr f\log f\,\dx\dv \le \liminf_{\ve\to 0} \iinttr f^{\ve}\log f^\ve\,\dx\dv.
\end{align*}
In conclusion, we have shown lower semicontinuity of \eqref{eq : regkei, 1}:
\begin{equation}
\begin{split}
\label{eq : low sem cont 1}
    &\iint_{\T^d\times\R^d} \frac{1}{2}|v|^2 f + \kappa f\log f \,\dx\dv + \int_{\T^d}\frac{1}{2}|\nabla\Phi[f]|^2 + (\Phi[f]-1)e^{\Phi[f]} \,\dx  \le \liminf_{\ve\to 0}\eqref{eq : regkei, 1}.
\end{split}
\end{equation}

%%%%%%%%%%%%%%%%%%%%%%%%%%%%%%%%%%%%%%%%%%%%%%%%%%%%%%%%%
%
%
%
%
%
%%%%%%%%%%%%%%%%%%%%%%%%%%%%%%%%%%%%%%%%%%%%%%%%%%%%%%%%%
\subsubsection{Lower semicontinuity of \eqref{eq : regkei, 2}} \label{sec: lower semcont} We begin by rewriting the dissipation rate in \eqref{eq : regkei, 2}:
\begin{equation*}
\begin{split}
&\int_0^t\iinttr \kappa\frac{|\nabla_v f^\ve|^2}{f^\ve} - \kappa d f^\ve+|{u^{(\ve)}_{f^\ve}}-v|^2f^\ve \,\dx\dv\ds \\
&\quad = \int_0^t \iinttr 4\kappa|\nabla_v \sqrt{f^\ve}|^2 + |{u^{(\ve)}_{f^\ve}} - v|^2 f^\ve \,\dx\dv \ds - \kappa d \|f_0\|_{L^1}t.
\end{split}
\end{equation*}
Similarly, we get
\begin{align*}
&\int_0^t\iinttr \kappa\frac{|\nabla_v f|^2}{f} - \kappa d f +|u_f-v|^2 f \,\dx\dv\ds \\
&\quad = \int_0^t \iinttr 4\kappa|\nabla_v \sqrt{f}|^2 + |u_f - v|^2 f \,\dx\dv \ds - \kappa d \|f_0\|_{L^1}t.
\end{align*}
Thus, to prove the lower semicontinuity of \eqref{eq : regkei, 2}, it suffices to show that
\begin{align}
    &\int_0^t \iinttr |u_f - v|^2 f \,\dx\dv\ds \le \liminf_{\ve\to 0} \int_0^t \iinttr |u_{f^\ve}^{(\ve)} - v|^2 f^\ve\,\dx\dv\ds,   \label{eq : kei 2 ETS2}\\
&\int_0^t \iinttr 4\kappa |\nabla_v \sqrt{f}|^2\,\dx\dv\ds \le \liminf_{\ve\to 0} \int_0^t \iinttr 4\kappa|\nabla_v \sqrt{f^\ve}|^2  \,\dx\dv \ds. \label{eq: kei 2 ETS3}
\end{align}
Notice from the Cauchy--Schwarz and arithmetic geometric inequalities that
\begin{align*}
    \int_0^t \iinttr |u_{f^\ve}^{(\ve)} - v|^2 f^\ve\,\dx\dv\ds &= \int_0^t \intt \left(\intr |v|^2 f^\ve \dv\right) +  |u_{f^\ve}^{(\ve)}|^2 \rho_{f^\ve} - 2\rho_{f^\ve} u_{f^\ve} \cdot u_{f^\ve}^{(\ve)} \, \dx\ds \\
    &\ge \int_0^t \intt \left(\intr |v|^2 f^\ve \dv - \rho_{f^\ve} |u_{f^\ve}|^2 \right)\, \dx\ds.
\end{align*}
Observe next that the integrand of the right-hand side, as a function of $s,x$, is non-negative. We take the $\liminf$ of both sides of the above, then apply Fatou's lemma to the right-hand side, recalling \eqref{eq:fp:mvstr}:
\begin{align*}
    \liminf_{\ve\to 0} \int_0^t \iinttr |u_{f^\ve}^{(\ve)}-v|^2 f^\ve\,\dx\dv\ds &\ge \int_0^t \int_{\T^d} \liminf_{\ve \to 0} \left( \int_{\R^d}|v|^2 f^\ve \,\dv - \rho_{f^\ve}|u_{f^\ve}|^2\right)\,\dx\ds \\
    &\ge \int_0^t \intr \left(\intr |v|^2 f\dv - \rho_f|u_f|^2 \right)\,\dx\ds \\
    &= \int_0^t \iinttr |u_f - v|^2 f\,\dx\dv\ds.
\end{align*}
This proves $\eqref{eq : kei 2 ETS2}$ for all $\kappa\ge 0$. 

Next we prove \eqref{eq: kei 2 ETS3}. When $\kappa=0$, there is nothing to prove, so we let $\kappa>0$. In this case, let us observe the enhanced regularity that is provided by Lemma \ref{lem:fp:unive} (i). Said roughly, we know that the sequence $\{\nabla_v (f^\ve)^{p/2}\}$ is weakly compact in $L^2((0,T)\times\T^d\times\R^d)$ for any value of $p\in[1,\infty)$. Since the velocity averaging lemma provides compactness in the $t,x$ variables, whereas for each $t,x$ we have compactness in the $v$ variable (thanks to the $H^1(\R^d_v)$ bounds on $f^\ve$), we deduce the compactness of the sequence $\{f^\ve\}$ as a whole. This intuition is confirmed via the following result of \cite{sampaio2024}. We state their result below in a simpler version for the sake of our purposes.

\begin{lemma} \cite[Proposition 2]{sampaio2024}  \label{lem:ptwcon}
    Let $\{f^m\}$ be a sequence of functions with $f^m \weakto f$ in $L^1((0,T)\times\T^d\times\R^d)$, and weakly$^*$ in $L^\infty((0,T)\times\T^d\times\R^d)$. Suppose that there exists a constant $C>0$ such that
    \begin{align*}
        \int_0^T \iint_{\T^d\times \R^d} |\nabla_v f^m|^2\,\dx\dv\dt \le C,
    \end{align*}
    and assume also that for each $\varphi\in C_c^\infty(\R^d)$, the sequence $\left\{\int_{\R^d} f^m\varphi\,\dv\right\}_m$ is compact in $L^1((0,T)\times \T^d)$. Then up to some subsequence we have $f^m(t,x,v) \to f(t,x,v)$ almost everywhere in $[0,T]\times\T^d\times\R^d$.
\end{lemma} 

Lemma \ref{lem:ptwcon} applies directly to our sequence $\{f^\ve\}$, whenever $\kappa>0$, as a result of Lemma \ref{lem:fp:unive} (i). We deduce that $f^\ve(t,x,v) \to f(t,x,v)$ almost everywhere. Proceeding further, we note that the moment bounds in Lemma \ref{lem:fp:unive} (iv) assert the tightness of $\{f^\ve\}$, while the $L^\infty$-bounds of Lemma \ref{lem:fp:unive} (i) imply that $\{f^\ve\}$ is uniformly integrable. Thus, the Vitali dominated convergence theorem shows that
\begin{align*}
    f^\ve \to f \quad \text{in} \quad L^p((0,T)\times\T^d\times\R^d), \quad \forall p\in [1,\infty).
\end{align*}
We deduce (for instance, by using the moment bounds to localize, and then owing to Egorov's theorem) that the weak limit of $\{\nabla_v \sqrt{f^\ve}\}$ in $L^2((0,T)\times\T^d\times\R^d)$ is necessarily $\nabla_v \sqrt{f}$. Then, by the lower semicontinuity of the $L^2$-norm with respect to weak convergence, we deduce \eqref{eq: kei 2 ETS3}:
\begin{align*}
    \int_0^t \iinttr 4\kappa|\nabla_v \sqrt{f}|^2 \,\dx\dv\dt \le \liminf_{\ve\to 0}\int_0^t \iinttr 4\kappa|\nabla_v \sqrt{f^\ve}|^2 \,\dx\dv\ds.
\end{align*}
Together, \eqref{eq : kei 2 ETS2}--\eqref{eq: kei 2 ETS3} show the lower semicontinuity of \eqref{eq : regkei, 2}:
\begin{align*}
    \int_0^t \iinttr \frac{|\kappa \nabla_v f - (u_f - v)f|^2}{f}\, \dx\dv\ds \le \liminf_{\ve\to 0} \eqref{eq : regkei, 2}.
\end{align*}
Combining this result with that obtained in \eqref{eq : low sem cont 1}, we deduce \eqref{eq : KEI} for all $t\in [0,T]$. As our choice of $T>0$ was arbitrary, this completes the proof of Theorem \ref{thm:gws}.

%%%%%%%%%%%%%%%%%%%%%%%%%%%%%
%%%%%%%%%%%%%%%%%%%%%%%%%%%%%
%%%%%%%%%%%%%%%%%%%%%%%%%%%%%
%%%%%%%%%%%%%%%%%%%%%%%%%%%%%
%%%%%%%%%%%%%%%%%%%%%%%%%%%%%
%%%%%%%%%%%%%%%%%%%%%%%%%%%%%
%%%%%%%%%%%%%%%%%%%%%%%%%%%%%
%%%%%%%%%%%%%%%%%%%%%%%%%%%%%
%%%%%%%%%%%%%%%%%%%%%%%%%%%%%
%%%%%%%%%%%%%%%%%%%%%%%%%%%%%

\section*{Acknowledgments}
This work is supported by NRF grant no. 2022R1A2C1002820 and RS-2024-00406821.

%%%%%%%%%%%%%%%%%%%%%%%%%%%%%
%%%%%%%%%%%%%%%%%%%%%%%%%%%%%
%%%%%%%%%%%%%%%%%%%%%%%%%%%%%
%%%%%%%%%%%%%%%%%%%%%%%%%%%%%
%%%%%%%%%%%%%%%%%%%%%%%%%%%%%
%%%%%%%%%%%%%%%%%%%%%%%%%%%%%
%%%%%%%%%%%%%%%%%%%%%%%%%%%%%
%%%%%%%%%%%%%%%%%%%%%%%%%%%%%
%%%%%%%%%%%%%%%%%%%%%%%%%%%%%
%%%%%%%%%%%%%%%%%%%%%%%%%%%%%

\appendix

\section{Quantitative error estimates}\label{app: QE}

In this appendix, we establish \eqref{eq: tauconv l2}--\eqref{eq: conv of dist g1 k=0} in Theorem \ref{thm:hdr}, relying on the modulated energy bound \eqref{eq: conv mod egy} together with the kinetic entropy inequality \eqref{eq : KEI}.  
We begin by noting that the definition of $\calF(\cdot|\cdot)$ already yields the control
\begin{align*}
    \|\nabla \Phi[f^\tau] - \nabla \Phi\|_{L^\infty(0,T;L^2)} \le C\tau^{1/4},
\end{align*} 
which directly proves \eqref{eq: tauconv l2}.

%%%%%%%%%%%%%%%%%%%%%%%%%%%%%
%%%%%%%%%%%%%%%%%%%%%%%%%%%%%
%%%%%%%%%%%%%%%%%%%%%%%%%%%%%
%%%%%%%%%%%%%%%%%%%%%%%%%%%%%
%%%%%%%%%%%%%%%%%%%%%%%%%%%%%
%%%%%%%%%%%%%%%%%%%%%%%%%%%%%
%%%%%%%%%%%%%%%%%%%%%%%%%%%%%
%%%%%%%%%%%%%%%%%%%%%%%%%%%%%
%%%%%%%%%%%%%%%%%%%%%%%%%%%%%
%%%%%%%%%%%%%%%%%%%%%%%%%%%%%
\subsection{Density error estimates: ion and electron components}
We now establish a direct link between the $L^1$-distance of macroscopic densities and the modulated energy functional introduced in \eqref{eq: first-order expansion of eta}. This relation is quantified in the following lemma, which will serve as a basic tool in deriving moment error bounds.
\begin{lemma}\label{lem:rho}
Let $P(z) = z\log z - z + 1$ and denote by $P(\cdot|\cdot)$ its first-order modulation. For any $\rho_1, \rho_2 \in L_+^1(\T^d)$, we have
\[
\|\rho_1 - \rho_2\|_{L^1} \le \sqrt{2} \lt(\intt (\rho_1 + \rho_2) \,\dx\rt)^{\frac{1}{2}} \lt(\intt P(\rho_1|\rho_2) \,\dx\rt)^{\frac{1}{2}}.
\]
\end{lemma}
\begin{proof}
We start from the identity
\[
|\rho_1 - \rho_2| = \max\{\rho_1, \rho_2\}^{\frac{1}{2}} \min\lt\{\frac{1}{\rho_1}, \frac{1}{\rho_2}\rt\}^{\frac{1}{2}}|\rho_1 - \rho_2|.
\]
Applying H\"older's inequality yields
\[
\begin{split}
\intt |\rho_1 - \rho_2|\,\dx &\le \lt(\intt \max\{\rho_1, \rho_2\} \,\dx\rt)^{\frac{1}{2}}\lt(\intt \min\lt\{\frac{1}{\rho_1}, \frac{1}{\rho_2}\rt\} |\rho_1 - \rho_2|^2 \,\dx\rt)^{\frac{1}{2}} \\
&\le \sqrt2 \lt(\intt (\rho_1 + \rho_2) \,\dx\rt)^{\frac{1}{2}} \lt(\intt P(\rho_1|\rho_2) \,\dx\rt)^{\frac{1}{2}}
\end{split}
\]
due to Taylor's theorem:
\[
P(\rho_1|\rho_2) \ge \frac{1}{2}\min\lt\{\frac{1}{\rho_1}, \frac{1}{\rho_2}\rt\} (\rho_1-\rho_2)^2.
\]
This completes the proof.
\end{proof}

The above estimate is instrumental for connecting the modulated energy with the $L^p$-difference between densities. In our setting, when applied to specific pairs of densities, it immediately yields quantitative error estimates both for the density error and for the associated electric potential.

\begin{corollary}\label{cor:rho} Assume the hypotheses of Theorem \ref{thm:hdr}.  
\begin{enumerate}[(i)]
\item If $\kappa>0$, applying Lemma \ref{lem:rho} to the pair $(\rho^\tau,\rho)$, we obtain
\begin{align*}
    \|\rho^\tau - \rho\|_{L^1} \le \sqrt{\kappa} C [\calF(\tilde U^\tau|\tilde U)]^{\frac{1}{2}},
\end{align*}
where the uniform $L^1$-bounds follow from mass conservation.
\item For any $\kappa\ge 0$, applying Lemma \ref{lem:rho} to the pair
\[
(\rho_e^\tau,\rho_e)=(e^{\Phi[f^\tau]},e^\Phi)
\]
gives
\[
    \|e^{\Phi[f^\tau]} - e^{\Phi}\|_{L^1} \le C  [\calF(\tilde U^\tau|\tilde U)]^{\frac{1}{2}}.
\]
Here the neutrality condition (see, e.g., Proposition \ref{thm : existence and uniqueness}) guarantees that both $\{e^{\Phi[f^\tau]}\}$ and $\{e^\Phi\}$ are uniformly bounded in $L^\infty(0,T;L^1)$. Moreover,  one has $P(e^{\Phi[f^\tau]}|e^\Phi)\le \calF(\tilde U^\tau | \tilde U)$ (see \eqref{def: B_gamma kappa}). In particular, \eqref{eq: tauconv exp} follows.
\end{enumerate} 
\end{corollary}

%%%%%%%%%%%%%%%%%%%%%%%%%%%%%%%%%%%%%%%%%%%%%%%%%%%%%%%%%%
%
%
%
%
%
%%%%%%%%%%%%%%%%%%%%%%%%%%%%%%%%%%%%%%%%%%%%%%%%%%%%%%%%%%
\subsection{Macroscopic and mesoscopic convergence: pressured case} \label{sec: conv mac bgk fp}

In this subsection, we establish detailed error estimates in the pressured case ($\kappa>0$). 

%%%%%%%%%%%%%%%%%%%%%%%%%%%%%
%%%%%%%%%%%%%%%%%%%%%%%%%%%%%
%%%%%%%%%%%%%%%%%%%%%%%%%%%%%
%%%%%%%%%%%%%%%%%%%%%%%%%%%%%
%%%%%%%%%%%%%%%%%%%%%%%%%%%%%
%%%%%%%%%%%%%%%%%%%%%%%%%%%%%
%%%%%%%%%%%%%%%%%%%%%%%%%%%%%
%%%%%%%%%%%%%%%%%%%%%%%%%%%%%
%%%%%%%%%%%%%%%%%%%%%%%%%%%%%
%%%%%%%%%%%%%%%%%%%%%%%%%%%%%
\subsubsection{Convergence of the macroscopic variables}

We first quantify the convergence of the macroscopic variables proving \eqref{eq: convergences, kappa>0}. Let us begin by recording some elementary observations, whose proofs can be found in \cite[Lemma 2.2]{carrillochoijung2021}.

\begin{lemma}\label{lem:rhou} Suppose that $\rho^\tau , \rho, \rho^\tau|u^\tau|^2\in L^1(\T^d)$ and $u\in L^\infty(\T^d)$. Then the following inequalities hold:
\begin{align*}
&\|\rho^\tau u^\tau -\rho u\|_{L^1} \le \| \rho^\tau\|_{L^1}^\frac{1}{2}\lt(\intt \rho^\tau|u^\tau-u|^2 \,\dx \rt)^\frac{1}{2} + \|u\|_{L^\infty}\|\rho^\tau -\rho\|_{L^1}, \\
&\|\rho^\tau u^\tau \otimes u^\tau - \rho u \otimes u\|_{L^1}  \le \intt \rho^\tau|u^\tau-u|^2 \,\dx +2\| \rho^\tau\|_{L^1}^\frac{1}{2}\lt(\intt \rho^\tau|u^\tau-u|^2 \,\dx \rt)^\frac{1}{2}\|u\|_{L^\infty} + \|\rho^\tau-\rho\|_{L^1}\|u\|_{L^\infty}^2.
\end{align*}
\end{lemma}

We now proceed with the detailed proof of the convergence estimates stated in \eqref{eq: convergences, kappa>0}.

\begin{proof}[Proof of \eqref{eq: convergences, kappa>0}]

We recall from Lemma \ref{lem:eps} that
\[%\bq\label{eq:calF:eps}
\calF(\tilde U^\tau|\tilde U) \le C\sqrt{\tau}, \quad \forall t\in[0,T].
\]%\eq
The first assertion of \eqref{eq: convergences, kappa>0}, 
\[
\|\rho^\tau - \rho\|_{L^\infty(0,T;L^1)} \le C\tau^{1/4}
\]
was already proven in Corollary \ref{cor:rho}. Combining this with Lemma \ref{lem:rhou} yields
\[
\|\rho^\tau u^\tau - \rho u\|_{L^\infty(0,T;L^1)} \le C\tau^{1/4}
\]
and
\[
\|\rho^\tau u^\tau\otimes u^\tau - \rho u\otimes u\|_{L^\infty(0,T;L^1)} \le C\tau^{1/4}.
\]
Then we estimate
\[%\begin{equation} \label{eq: g1 tensor}
\begin{split}
&\int_0^t \intt \lt|\int_{\R^d} v\otimes v f^\tau \dv -\lt(\rho u \otimes u + \kappa \rho \bbI \rt) \rt| \dx\ds \\
&\quad \le  \int_0^t \intt \lt|\int_{\R^d} v\otimes v f^\tau \dv -\lt(\rho^\tau u^\tau \otimes u^\tau + \kappa\rho^\tau \bbI \rt) \rt| \dx\ds + \|\rho^\tau u^\tau\otimes u^\tau - \rho u\otimes u\|_{L^1(0;T;L^1)} \\
&\quad = \int_0^t \int_{\T^d} \left|\int_{\R^d} v\otimes \left\{ (u^\tau - v)f^\tau - \kappa \nabla_v f^\tau \right\}\,\dv \right| \,\dx\ds + \|\rho^\tau u^\tau\otimes u^\tau - \rho u\otimes u\|_{L^1(0;T;L^1)}.
\end{split}
\]%\end{equation}
For the first term on the right-hand side, we proceed as in \eqref{eq:diss} by using the kinetic entropy inequality. The second term is bounded by $C\tau^{1/4}$. Collecting the estimates, we deduce
\[
\lt\|\int_{\R^d} v\otimes v f^\tau \dv -\lt(\rho u \otimes u + \kappa\rho \bbI \rt)\rt\|_{L^1(0,T;L^1)} \le C\tau^{1/4}.
\]

\end{proof}
%%%%%%%%%%%%%%%%%%%%%%%%%%%%%%%%%%%%%%%%%%%%%%%%%%%%%%%%%
%
%
%
%
%
%%%%%%%%%%%%%%%%%%%%%%%%%%%%%%%%%%%%%%%%%%%%%%%%%%%%%%%%%
 %%%%%%%%%%%%%%%%%%%%%%%%%%%%%
%%%%%%%%%%%%%%%%%%%%%%%%%%%%%
%%%%%%%%%%%%%%%%%%%%%%%%%%%%%
%%%%%%%%%%%%%%%%%%%%%%%%%%%%%
%%%%%%%%%%%%%%%%%%%%%%%%%%%%%
%%%%%%%%%%%%%%%%%%%%%%%%%%%%%
%%%%%%%%%%%%%%%%%%%%%%%%%%%%%
%%%%%%%%%%%%%%%%%%%%%%%%%%%%%
%%%%%%%%%%%%%%%%%%%%%%%%%%%%%
%%%%%%%%%%%%%%%%%%%%%%%%%%%%%
\subsubsection{Convergence of the mesoscopic variables}
In this part, we prove that the control of the relative entropy implies the convergence of the distribution functions: 
\[
f^\tau \to M^{(\rho,u)}.
\]

%%%%%%%%%%%%%%%%%%%%%%%%%%%%%%%%%%%%%%%%%%%%%%%%%%%%%%%%%
%
%
%
%
%
%%%%%%%%%%%%%%%%%%%%%%%%%%%%%%%%%%%%%%%%%%%%%%%%%%%%%%%%%

Our approach relies on a classical logarithmic Sobolev inequality.

\begin{lemma}[Logarithmic Sobolev inequality]
Let $\mu$ denote the standard Gaussian measure on $\R^d$. Then
\[
\int_{\R^d} g(v)\log g(v)\, \dmu \le \frac{1}{2} \int_{\R^d} \frac{|\nabla_v g(v)|^2}{g(v)} \,\dmu + 2 \lt( \int_{\R^d} g(v) \, \dmu \rt) \log \lt( \int_{\R^d} g(v) \, \dmu \rt).
\]
\end{lemma}

This inequality allows us to derive estimates for the relative Boltzmann entropy between a function $f$ and the local Maxwellian. In particular, we have the following result.
\begin{lemma} \label{lem: KL div}
For $f \in C^\infty_c(\R^d)$ satisfying $f(v) \ge 0$ for all $v \in \R^d$ and  
\[
\rho_f := \int f(v)\,\dv >0,
\]
set, for $\kappa>0$ and $u\in \R^d$:
\[
M^{(\rho_f,u)}(v) = \frac{\rho_f}{(2\pi \kappa)^{\frac{d}{2}}} e^{-\frac{|v-u|^2}{2\kappa}}.
\]
Then the following inequality holds:
\[
\kappa \intr f(v)\log \frac{f(v)}{M^{(\rho_f,u)}(v)}\, \dv \le \,\,\frac{1}{2} \intr \frac{|\kappa \nabla_v f(v) - (u-v) f(v) |^2}{f(v)} \, \dv.
\]
\end{lemma}

\begin{proof}
For a smooth and compactly supported $f(v)$, we define $g$ with $f(v)=g(\frac{v-u}{\sqrt{\kappa}})$, and
\[
G(v;u,\kappa):=\frac{1}{(2\pi \kappa)^{\frac{d}{2}}} e^{-\frac{|v-u|^2}{2\kappa}}.
\]
Applying the Logarithmic Sobolev inequality to $g(\cdot)$ with a change of variable $v \mapsto \frac{v-u}{\sqrt{\kappa}}$, we obtain
\[
\begin{split}
&\intr f(v)\log f(v) G(v;u,\kappa)\,\dv \cr
&\quad \le \frac{\kappa}{2} \intr \frac{|\nabla_v f(v)|^2}{f(v)}G(v;u,\kappa)\,\dv   +2 \lt( \intr f(v) G(v;u,\kappa)\dv \rt) \log \lt( \intr f(v) G(v;u,\kappa)\,\dv \rt).
\end{split}
\]
By substituting $f(v)$ with $f(v)/M^{(\rho_f,u)}(v)$, the last term on the right-hand side vanishes. Then multiplying $\rho_f$ on both sides, we arrive at the result.
\end{proof}

We now estimate the relative Boltzmann entropy between $f^\tau$ and the local Maxwellian $M^{(\rho,u)}$.

\begin{lemma}
Let $f^\tau(v) \ge 0$ and $(\rho,u)\in \R_+\times \R^d$ be given, and denote $(\rho^\tau, \rho^\tau u^\tau) = (\rho_{f^{\tau}}, \rho_{f^{\tau}} u_{f^{\tau}})$. Then the following inequality holds:
\bq \label{eq: fokk rel tau}
\begin{split}
&\kappa \iinttr f^\tau \log \frac{f^\tau}{M^{(\rho,u)}} \dx\dv \cr
&\quad \le \frac{1}{2} \iinttr \frac{|\kappa \nabla_v f^\tau(v) - (u^\tau-v) f^\tau(v) |^2}{f^\tau(v)} \,\dx\dv  +\kappa \intt \rho^\tau \log \frac{\rho^{\tau}}{\rho} \,\dx + \frac{1}{2}\intt  \rho^\tau |u^\tau - u|^2 \, \dx.
\end{split}
\eq
\end{lemma}
\begin{proof}
We observe that
\[
\begin{split}
f^\tau \log \frac{f^\tau}{M^{(\rho, u)}} &= f^\tau \log \frac{f^\tau}{M^{(\rho^\tau,u^\tau)}} + f^\tau \log \frac{M^{(\rho^\tau,u^\tau)}}{M^{(\rho,u^\tau)}} + f^\tau \log \frac{M^{(\rho,u^\tau)}}{M^{(\rho, u)}} \\
&= f^\tau \log \frac{f^\tau}{M^{(\rho^\tau,u^\tau)}} + f^\tau \log \frac{\rho^{\tau}}{\rho} - f^\tau\lt(\frac{|u^\tau|^2 - |u|^2 - 2(u^\tau-u)\cdot v}{2\kappa}\rt).
\end{split}
\]
Thus, it follows that
\[
\begin{split}
\intr f^\tau \log \frac{f^\tau}{M^{(\rho, u)}} \dv &=  \intr f^\tau \log \frac{f^\tau}{M^{(\rho^\tau,u^\tau)}} \dv + \rho^\tau \log \frac{\rho^{\tau}}{\rho} - \lt(\frac{\rho^\tau|u^\tau|^2 - \rho^\tau|u|^2 - 2(u^\tau-u)\cdot \rho^\tau u^\tau}{2\kappa}\rt) \\
&= \intr f^\tau \log \frac{f^\tau}{M^{(\rho^\tau,u^\tau)}} \dv + \rho^\tau \log \frac{\rho^{\tau}}{\rho} + \lt(\frac{\rho^\tau|u^\tau|^2 + \rho^\tau|u|^2 - 2u\cdot \rho^\tau u^\tau}{2\kappa}\rt) \\
&= \intr f^\tau \log \frac{f^\tau}{M^{(\rho^\tau,u^\tau)}} \dv + \rho^\tau \log \frac{\rho^{\tau}}{\rho} + \lt(\frac{\rho^\tau |u^\tau - u|^2}{2\kappa}\rt).
\end{split}
\]
By estimating the remaining integral on the right-hand side with Lemma \ref{lem: KL div}, we obtain \eqref{eq: fokk rel tau}.
\end{proof}

We now conclude with the convergence result.

\begin{proof}[Proof of \eqref{eq: conv of dist g1 k>0}]
In view of \eqref{eq: conv mod egy}, \eqref{eq: convergences, kappa>0}, and the kinetic entropy inequality \eqref{eq : KEI}, we note that the right-hand side of \eqref{eq: fokk rel tau} is of order $\tau^{1/2}$. The convergence asserted in \eqref{eq: conv of dist g1 k>0} then follows from the Csisz\'ar--Kullback inequality.
\end{proof}

%%%%%%%%%%%%%%%%%%%%%%%%%%%%%%%%%%%%%%%%%%%%%%%%%%%%%%%%%
%
%
%
%
%
%%%%%%%%%%%%%%%%%%%%%%%%%%%%%%%%%%%%%%%%%%%%%%%%%%%%%%%%%
\subsection{Macroscopic and mesoscopic convergence: pressureless case}
In this part, we provide the bound estimates for the pressureless case ($\kappa = 0$). To streamline our proof, we first recall several results from \cite{carrillochoijung2021, choi2021, FK19}.

\begin{lemma} \label{lem: choi2021}
    \cite[Proposition 3.1]{choi2021} Let $\tilde\rho\in C([0,T];\calP(\T^d))$ be a solution to the continuity equation
    \begin{align*}
        \p_t \tilde\rho + \nabla_x\cdot (\tilde\rho \tilde u) = 0,
    \end{align*}
    where $\tilde u$ is a Borel vector field with
    \begin{align*}
        \intt \tilde\rho |\tilde u|^2 \,\dx < \infty.
    \end{align*}
    Assume that $(\rho, u)$ is another solution to the continuity equation with $\nabla u\in L^\infty((0,T)\times\T^d)$. Then there exists $C>0$, dependent only on $T$ such that
    \begin{align*}
        W_1(\tilde\rho, \rho)(t) \le Ce^{C\|\nabla u\|_{L^\infty}} \left(W_1(\tilde\rho_0,\rho_0) +\left(\int_0^t \intt \tilde\rho |\tilde u - u|^2\,\dx\ds \right)^{\frac{1}{2}} \right).
    \end{align*}
    Here $W_1$ denotes the first-order Wasserstein distance.
\end{lemma}

\begin{lemma} \label{lem: carrillochoijung}
    \cite[Lemma 2.2]{carrillochoijung2021} There exists a constant $C>0$ dependent only on $\|u\|_{L^\infty(0,T;W^{1,\infty}(\T^d))}$ such that
    \begin{align*}
        &d_{\rm BL}(\rho^\tau u^\tau, \rho u) \le \|\rho^\tau\|_{L^1}^{\frac{1}{2}} \left(\intt \rho^\tau |u^\tau - u|^2\,\dx\right)^{\frac{1}{2}} + C d_{\rm BL}(\rho^\tau, \rho), \\
        &d_{\rm BL}(\rho^\tau u^\tau \otimes u^\tau, \rho u\otimes u) \le \intt \rho^\tau |u^\tau - u|^2\,\dx + C\|\rho^\tau\|_{L^1}^{\frac{1}{2}}\left(\intt \rho^\tau |u^\tau - u|^2\,\dx\right)^{\frac{1}{2}} + Cd_{\rm BL}(\rho^\tau,\rho),\\
        &d_{\rm BL}(f^\tau, \rho \otimes \delta(\cdot - u)) \le \|\rho^\tau\|_{L^1}^{\frac{1}{2}}\left( \left(\iinttr |v-u^\tau|^2 f^\tau\,\dx\dv \right)^{\frac{1}{2}} + \left(\intt \rho^\tau|u^\tau - u|^2\,\dx\right)^{\frac{1}{2}} \right) + Cd_{\rm BL}(\rho^\tau,\rho).
    \end{align*}
\end{lemma}

\begin{proof}[Proof of \eqref{eq: convergences kappa=0}]
    We apply Lemma \ref{lem: choi2021} with $(\tilde\rho,\tilde\rho\tilde u) = (\rho^\tau, \rho^\tau u^\tau)$ and take to be $(\rho, u)$ the strong solution to the Euler equations \eqref{eq : ionic EP}. Using the equivalence between the Wasserstein-1 distance and the bounded Lipschitz distance (which holds in bounded domains), we deduce
    \begin{align}\label{eq: dbl rho}
        d_{\rm BL}(\rho^\tau, \rho)(t) \le C_T \left(d_{\rm BL}(\rho^\tau_0,\rho_0) + \sup_{0\le t \le T}[\calF(U^\tau|U)(t)]^{\frac{1}{2}} \right) \le C_T \tau^{1/4}.
    \end{align}
    Consequently, Lemma \ref{lem: carrillochoijung} immediately provides
    \begin{align*}
        d_{\rm BL}(\rho^\tau u^\tau, \rho u) \le C\tau^{1/4}, \quad d_{\rm BL}(\rho^\tau u^\tau \otimes u^\tau, \rho u\otimes u) \le C \tau^{1/4}.
    \end{align*}
    Finally, using \eqref{eq : KEI}, \eqref{eq: conv mod egy}, Lemma \ref{lem: carrillochoijung} and \eqref{eq: dbl rho}, we obtain
    \begin{align*}
        &\int_0^T d_{\rm BL}^2\Big(f^\tau(t) \;,\; \rho(t)\otimes\delta(\cdot-u(t))\Big) \,\dt \\
        &\quad \le C \int_0^T \iinttr |v-u^\tau|^2 f^\tau\,\dx\dv\dt + C\int_0^T \intt \rho^\tau |u^\tau-u|^2\,\dx\dt + C \int_0^T d_{\rm BL}^2(\rho^\tau,\rho) \,\dt \\
        &\quad \le C_T \, \tau^{1/2}.
    \end{align*}
    This completes the proof.
\end{proof}

%%%%%%%%%%%%%%%%%%%%%%%%%%%%%%%%%%%%%%%%%%%%%%%%%%%%%%%%%
%%%%%%%%%%%%%%%%%%%%%%%%%%%%%%%%%%%%%%%%%%%%%%%%%%%%%%%%%%%
%
%
%
%
%
%
%
%
%
%%%%%%%%%%%%%%%%%%%%%%%%%%%%%%%%%%%%%%%%%%%%%%%%%%%%%%%%%%%

\bibliographystyle{abbrv}
%\bibliography{ionhydro}

\begin{thebibliography}{10}

\bibitem{AFP2000}
L.~Ambrosio, N.~Fusco, and D.~Pallara.
\newblock {\em Functions of bounded variation and free discontinuity problems}.
\newblock Oxford Mathematical Monographs. The Clarendon Press, Oxford
  University Press, New York, 2000.

\bibitem{BCK24}
J.~Bae, J.~Choi, and B.~Kwon.
\newblock Formation of singularities in plasma ion dynamics.
\newblock {\em Nonlinearity}, 37(4):Paper No. 045011, 29, 2024.

\bibitem{bardosgolsenguyensentis2018}
C.~Bardos, F.~Golse, T.~T. Nguyen, and R.~Sentis.
\newblock The {M}axwell-{B}oltzmann approximation for ion kinetic modeling.
\newblock {\em Phys. D}, 376/377:94--107, 2018.

\bibitem{berthelinvasseur2005}
F.~Berthelin and A.~Vasseur.
\newblock From kinetic equations to multidimensional isentropic gas dynamics
  before shocks.
\newblock {\em SIAM J. Math. Anal.}, 36(6):1807--1835, 2005.

\bibitem{bonillacarrillosoler1997}
L.~L. Bonilla, J.~A. Carrillo, and J.~Soler.
\newblock Asymptotic behavior of an initial-boundary value problem for the
  {V}lasov-{P}oisson-{F}okker-{P}lanck system.
\newblock {\em SIAM J. Appl. Math.}, 57(5):1343--1372, 1997.

\bibitem{bostan2024}
M.~Bostan and A.-T. VU.
\newblock Asymptotic behavior of the two-dimensional
  {V}lasov-{P}oisson-{F}okker-{P}lanck equation with a strong external magnetic
  field.
\newblock {\em Kinet. Relat. Models}, 18(1):101--147, 2025.

\bibitem{bouchut1991}
F.~Bouchut.
\newblock Global weak solution of the {V}lasov-{P}oisson system for small
  electrons mass.
\newblock {\em Comm. Partial Differential Equations}, 16(8-9):1337--1365, 1991.

\bibitem{Bouchut2002}
F.~Bouchut.
\newblock Hypoelliptic regularity in kinetic equations.
\newblock {\em J. Math. Pures Appl. (9)}, 81(11):1135--1159, 2002.

\bibitem{CarrapatosoMischler}
K.~Carrapatoso and S.~Mischler.
\newblock The {K}inetic {F}okker--{P}lanck equation in a domain:
  {U}ltracontractivity, hypocoercivity, and long-time asymptotic behavior.
\newblock {\em Atti Accad. Naz. Lincei Rend. Lincei Mat. Appl.},
  35(4):643--680, 2024.

\bibitem{Carrillo1998}
J.~A. Carrillo.
\newblock Global weak solutions for the initial-boundary-value problems to the
  {V}lasov-{P}oisson-{F}okker-{P}lanck system.
\newblock {\em Math. Methods Appl. Sci.}, 21(10):907--938, 1998.

\bibitem{CC20}
J.~A. Carrillo and Y.-P. Choi.
\newblock Quantitative error estimates for the large friction limit of {V}lasov
  equation with nonlocal forces.
\newblock {\em Ann. Inst. H. Poincar\'e{} C Anal. Non Lin\'eaire},
  37(4):925--954, 2020.

\bibitem{carrillochoijung2021}
J.~A. Carrillo, Y.-P. Choi, and J.~Jung.
\newblock Quantifying the hydrodynamic limit of {V}lasov-type equations with
  alignment and nonlocal forces.
\newblock {\em Math. Models Methods Appl. Sci.}, 31(2):327--408, 2021.

\bibitem{cesbroniacobelli2021}
L.~Cesbron and M.~Iacobelli.
\newblock Global well-posedness of {V}lasov-{P}oisson-type systems in bounded
  domains.
\newblock {\em Anal. PDE}, 16(10):2465--2494, 2023.

\bibitem{Choi2016}
Y.-P. Choi.
\newblock Global classical solutions of the {V}lasov-{F}okker-{P}lanck equation
  with local alignment forces.
\newblock {\em Nonlinearity}, 29(7):1887--1916, 2016.

\bibitem{choi2021}
Y.-P. Choi.
\newblock Large friction limit of pressureless {E}uler equations with nonlocal
  forces.
\newblock {\em J. Differential Equations}, 299:196--228, 2021.

\bibitem{choihwangyoo}
Y.-P. Choi, B.-H. Hwang, and Y.~Yoo.
\newblock Global existence of weak solutions to the nonlinear
  {V}lasov-{F}okker-{P}lanck equation.
\newblock {\em J. Differential Equations}, 444:Paper No. 113573, 53, 2025.

\bibitem{CJ20}
Y.-P. Choi and J.~Jung.
\newblock Asymptotic analysis for {V}lasov-{F}okker-{P}lanck/compressible
  {N}avier-{S}tokes equations with a density-dependent viscosity.
\newblock In {\em Hyperbolic problems: theory, numerics, applications},
  volume~10 of {\em AIMS Ser. Appl. Math.}, pages 145--163. Am. Inst. Math.
  Sci. (AIMS), Springfield, MO, [2020] \copyright 2020.

\bibitem{CJ23}
Y.-P. Choi and J.~Jung.
\newblock On the dynamics of charged particles in an incompressible flow: from
  kinetic-fluid to fluid-fluid models.
\newblock {\em Commun. Contemp. Math.}, 25(7):Paper No. 2250012, 78, 2023.

\bibitem{CJ24}
Y.-P. Choi and J.~Jung.
\newblock Incompressible {N}avier-{S}tokes limit from nonlinear
  {V}lasov-{F}okker-{P}lanck equation.
\newblock {\em Appl. Math. Lett.}, 158:Paper No. 109214, 7, 2024.

\bibitem{CJpre}
Y.-P. Choi and J.~Jung.
\newblock Incompressible {E}uler limits from a nonlinear
  {V}lasov-{F}okker-{P}lanck equation with constant temperature.
\newblock {\em Appl. Math. Lett.}, to appear.

\bibitem{CKKTpre}
Y.-P. Choi, D.-h. Kim, D.~Koo, and E.~Tadmor.
\newblock Critical thresholds in pressureless {E}uler-{P}oisson equations with
  background states.
\newblock {\em Ann. Inst. H. Poincar\'e C Anal. Non Lin\'eaire}, to appear.

\bibitem{choikoosong2025}
Y.-P. Choi, D.~Koo, and S.~Song.
\newblock Global existence of {L}agrangian solutions to the ionic
  {V}lasov--{P}oisson system, arXiv:2501.13872.

\bibitem{dafermos1979}
C.~M. Dafermos.
\newblock The second law of thermodynamics and stability.
\newblock {\em Arch. Rational Mech. Anal.}, 70(2):167--179, 1979.

\bibitem{dafermos2000}
C.~M. Dafermos.
\newblock {\em Hyperbolic conservation laws in continuum physics}, volume 325
  of {\em Grundlehren der mathematischen Wissenschaften [Fundamental Principles
  of Mathematical Sciences]}.
\newblock Springer-Verlag, Berlin, fourth edition, 2016.

\bibitem{degond1985}
P.~Degond.
\newblock Existence globale de solutions de l'\'{e}quation de
  {V}lasov-{F}okker-{P}lanck, en dimension {$1$} et {$2$}.
\newblock {\em C. R. Acad. Sci. Paris S\'{e}r. I Math.}, 301(3):73--76, 1985.

\bibitem{degond1986}
P.~Degond.
\newblock Global existence of smooth solutions for the
  {V}lasov-{F}okker-{P}lanck equation in {$1$} and {$2$} space dimensions.
\newblock {\em Ann. Sci. \'{E}cole Norm. Sup. (4)}, 19(4):519--542, 1986.

\bibitem{dipernalions1988}
R.~J. DiPerna and P.-L. Lions.
\newblock On the {F}okker-{P}lanck-{B}oltzmann equation.
\newblock {\em Comm. Math. Phys.}, 120(1):1--23, 1988.

\bibitem{Dipernalions1989ODE}
R.~J. DiPerna and P.-L. Lions.
\newblock Ordinary differential equations, transport theory and {S}obolev
  spaces.
\newblock {\em Invent. Math.}, 98(3):511--547, 1989.

\bibitem{dipernalionsmeyer}
R.~J. DiPerna, P.-L. Lions, and Y.~Meyer.
\newblock {$L^p$} regularity of velocity averages.
\newblock {\em Ann. Inst. H. Poincar\'{e} C Anal. Non Lin\'{e}aire},
  8(3-4):271--287, 1991.

\bibitem{FigalliGlaudo2023}
A.~Figalli and F.~Glaudo.
\newblock {\em An invitation to optimal transport, {W}asserstein distances, and
  gradient flows}.
\newblock EMS Textbooks in Mathematics. EMS Press, Berlin, [2023] \copyright
  2023.
\newblock Second edition [of 4331435].

\bibitem{FK19}
A.~Figalli and M.-J. Kang.
\newblock A rigorous derivation from the kinetic {C}ucker-{S}male model to the
  pressureless {E}uler system with nonlocal alignment.
\newblock {\em Anal. PDE}, 12(3):843--866, 2019.

\bibitem{FG24}
P.~Flynn and Y.~Guo.
\newblock The massless electron limit of the {V}lasov-{P}oisson-{L}andau
  system.
\newblock {\em Comm. Math. Phys.}, 405(2):Paper No. 27, 73, 2024.

\bibitem{GolseImbertMouhutVasseur2019}
F.~Golse, C.~Imbert, C.~Mouhot, and A.~F. Vasseur.
\newblock Harnack inequality for kinetic {F}okker-{P}lanck equations with rough
  coefficients and application to the {L}andau equation.
\newblock {\em Ann. Sc. Norm. Super. Pisa Cl. Sci. (5)}, 19(1):253--295, 2019.

\bibitem{grenierguopausadersuzuki2020}
E.~Grenier, Y.~Guo, B.~Pausader, and M.~Suzuki.
\newblock Derivation of the ion equation.
\newblock {\em Quart. Appl. Math.}, 78(2):305--332, 2020.

\bibitem{griffinpickering2024}
M.~Griffin-Pickering.
\newblock A probabilistic mean field limit for the {V}lasov--{P}oisson system
  for ions, arXiv:2410.10612.

\bibitem{griffinpickeringiacobelli2020}
M.~Griffin-Pickering and M.~Iacobelli.
\newblock Singular limits for plasmas with thermalised electrons.
\newblock {\em J. Math. Pures Appl. (9)}, 135:199--255, 2020.

\bibitem{griffinpickeringiacobelli2021b}
M.~Griffin-Pickering and M.~Iacobelli.
\newblock Global strong solutions in {$\Bbb R^3$} for ionic {V}lasov-{P}oisson
  systems.
\newblock {\em Kinet. Relat. Models}, 14(4):571--597, 2021.

\bibitem{griffinpickeringiacobelli2021a}
M.~Griffin-Pickering and M.~Iacobelli.
\newblock Global well-posedness for the {V}lasov-{P}oisson system with massless
  electrons in the 3-dimensional torus.
\newblock {\em Comm. Partial Differential Equations}, 46(10):1892--1939, 2021.

\bibitem{griffinpickeringiacobelli2021c}
M.~Griffin-Pickering and M.~Iacobelli.
\newblock Recent developments on the well-posedness theory for {V}lasov-type
  equations.
\newblock In {\em From particle systems to partial differential equations},
  volume 352 of {\em Springer Proc. Math. Stat.}, pages 301--319. Springer,
  Cham, [2021] \copyright 2021.

\bibitem{GP11}
Y.~Guo and B.~Pausader.
\newblock Global smooth ion dynamics in the {E}uler-{P}oisson system.
\newblock {\em Comm. Math. Phys.}, 303(1):89--125, 2011.

\bibitem{hankwan2011}
D.~Han-Kwan.
\newblock Quasineutral limit of the {V}lasov-{P}oisson system with massless
  electrons.
\newblock {\em Comm. Partial Differential Equations}, 36(8):1385--1425, 2011.

\bibitem{Hormander1967}
L.~H\"{o}rmander.
\newblock Hypoelliptic second order differential equations.
\newblock {\em Acta Math.}, 119:147--171, 1967.

\bibitem{horst1990}
E.~Horst.
\newblock Global solutions of the relativistic {V}lasov-{M}axwell system of
  plasma physics.
\newblock {\em Dissertationes Math. (Rozprawy Mat.)}, 292:63, 1990.

\bibitem{JKO1998}
R.~Jordan, D.~Kinderlehrer, and F.~Otto.
\newblock The variational formulation of the {F}okker-{P}lanck equation.
\newblock {\em SIAM J. Math. Anal.}, 29(1):1--17, 1998.

\bibitem{kang2018}
M.-J. Kang.
\newblock From the {V}lasov-{P}oisson equation with strong local alignment to
  the pressureless {E}uler-{P}oisson system.
\newblock {\em Appl. Math. Lett.}, 79:85--91, 2018.

\bibitem{kangvasseur2015}
M.-J. Kang and A.~F. Vasseur.
\newblock Asymptotic analysis of {V}lasov-type equations under strong local
  alignment regime.
\newblock {\em Math. Models Methods Appl. Sci.}, 25(11):2153--2173, 2015.

\bibitem{karpermellettrivisa2013}
T.~K. Karper, A.~Mellet, and K.~Trivisa.
\newblock Existence of weak solutions to kinetic flocking models.
\newblock {\em SIAM J. Math. Anal.}, 45(1):215--243, 2013.

\bibitem{KMT14}
T.~K. Karper, A.~Mellet, and K.~Trivisa.
\newblock On strong local alignment in the kinetic {C}ucker-{S}male model.
\newblock In {\em Hyperbolic conservation laws and related analysis with
  applications}, volume~49 of {\em Springer Proc. Math. Stat.}, pages 227--242.
  Springer, Heidelberg, 2014.

\bibitem{KMT2015}
T.~K. Karper, A.~Mellet, and K.~Trivisa.
\newblock Hydrodynamic limit of the kinetic {C}ucker-{S}male flocking model.
\newblock {\em Math. Models Methods Appl. Sci.}, 25(1):131--163, 2015.

\bibitem{lanneslinaressaut2012}
D.~Lannes, F.~Linares, and J.-C. Saut.
\newblock The {C}auchy problem for the {E}uler-{P}oisson system and derivation
  of the {Z}akharov-{K}uznetsov equation.
\newblock In {\em Studies in phase space analysis with applications to {PDE}s},
  volume~84 of {\em Progr. Nonlinear Differential Equations Appl.}, pages
  181--213. Birkh\"{a}user/Springer, New York, 2013.

\bibitem{li2025globalclassicalsolutionsionic}
F.~Li and Y.~Wang.
\newblock Global classical solutions to the ionic
  {V}lasov-{P}oisson-{B}oltzmann system near {M}axwellians, arXiv:2502.05745.

\bibitem{lifshitzpitaevskii1981}
E.~M. Lifshitz and L.~P. Pitaevskii.
\newblock {\em Physical Kinetics}.
\newblock Pergamon Press, 1981.

\bibitem{Liu06}
H.~Liu.
\newblock Wave breaking in a class of nonlocal dispersive wave equations.
\newblock {\em J. Nonlinear Math. Phys.}, 13(3):441--466, 2006.

\bibitem{MelletVasseur2008}
A.~Mellet and A.~Vasseur.
\newblock Asymptotic analysis for a {V}lasov-{F}okker-{P}lanck/compressible
  {N}avier-{S}tokes system of equations.
\newblock {\em Comm. Math. Phys.}, 281(3):573--596, 2008.

\bibitem{sampaio2024}
P.~Sampaio.
\newblock Global solutions to the {L}andau-{F}ermi-{D}irac equations,
  arXiv:2410.12681.

\bibitem{spohn1980}
H.~Spohn.
\newblock Kinetic equations from {H}amiltonian dynamics: {M}arkovian limits.
\newblock {\em Rev. Modern Phys.}, 52(3):569--615, 1980.

\bibitem{titchmarsh1958}
E.~C. Titchmarsh.
\newblock {\em Eigenfunction expansions associated with second-order
  differential equations. {V}ol. 2}.
\newblock Oxford, at the Clarendon Press, 1958.

\bibitem{Zhu2024}
Y.~Zhu.
\newblock Regularity of kinetic {F}okker-{P}lanck equations in bounded domains.
\newblock {\em Ann. H. Lebesgue}, 7:1323--1366, 2024.

\bibitem{Zhu2025}
Y.~Zhu.
\newblock Averaging lemmas and hypoellipticity.
\newblock {\em Kinet. Relat. Models}, 18(5):800--823, 2025.

\end{thebibliography}

\end{document}